\documentclass[3p]{elsarticle}

\usepackage{amsfonts, amsthm, amsmath, amssymb}
\usepackage{bm} % for bold print of greek letters
\usepackage{dsfont} % for characteristic function in \mathds{1}
\usepackage{tikz-cd} %insert after xcolor
\usepackage[colorlinks]{hyperref} % no frame when referring
\hypersetup{breaklinks=true}
\usepackage{appendix} % else Proposition Appendix A.1.
\numberwithin{equation}{section} % Numbering in Equations in Section 4 as (4.1), (4.2), ... instead of (21), (22), ...

\newtheorem{theorem}{Theorem}[section]
\newtheorem{assumption}{Assumption}
\newtheorem{corollary}[theorem]{Corollary}
\newtheorem{definition}[theorem]{Definition}

\newtheorem{proposition}[theorem]{Proposition}
\newtheorem{lemma}[theorem]{Lemma}
\theoremstyle{remark} % remarks non-italic
\newtheorem{remark}[theorem]{Remark}

\makeatletter
\newcommand{\neutralize}[1]{\expandafter\let\csname c@#1\endcsname\count@}
\makeatother

\newenvironment{assprime}[1]
  {\renewcommand{\theassumption}{\ref*{#1}$'$}%
   \neutralize{assumption}\phantomsection
   \begin{assumption}}
  {\end{assumption}}

\journal{Journal de Math\'{e}matiques Pures et Appliqu\'{e}es}

\begin{document}

% titlepage
\begin{frontmatter}

 \title{Long-time shadow limit for reaction-diffusion-ODE systems} %\tnoteref{t1,t2}}
 
%\tnotetext[t1]{In memoriam of Andro Mikeli\'{c}}
%  \tnotetext[t1]{This document is the results of the research project funded by the National Science Foundation.} 
%\tnotetext[t2]{The second title footnote which is a longer text matter to fill through the whole text width and overflow into another line in the footnotes area of the first page.}

 \author[1]{Chris Kowall}% \fnref{fn1}}
 \ead{kowall@math.uni-heidelberg.de}
\author[2]{Anna Marciniak-Czochra\corref{cor1}}%\fnref{fn2}}
\ead{anna.marciniak@iwr.uni-heidelberg.de} 
\author[3]{Andro Mikeli\'{c}\fnref{fn3}} 
%\ead{andro.mikelic@univ-lyon1.fr}

\cortext[cor1]{Corresponding author}

%\fntext[fn1]{Footnote of first author} 
%\fntext[fn2]{This work is supported by the Deutsche Forschungsgemeinschaft (DFG, German Research Foundation) under Collaborative Research Center 1324 (SFB1324, project B6).}%Another author footnote, this is a very long footnote and it should be a really long footnote. But this footnote is not yet sufficiently long enough to make two lines of footnote text
\fntext[fn3]{Dedicated to the Memory of Professor Andro Mikeli\'{c}, deceased on November 28, 2020.}

\address[1]{Institute of Applied Mathematics and IWR, Heidelberg University,
Im Neuenheimer Feld 205, 69120 Heidelberg, Germany}
\address[2]{Institute of Applied Mathematics, IWR and BIOQUANT, Heidelberg University, Im Neuenheimer Feld 205, 69120 Heidelberg, Germany}
\address[3]{Univ Lyon, Universit\'{e} Claude Bernard Lyon 1, CNRS UMR 5208, Institut Camille Jordan, 43 blvd. du 11 novembre 1918, F-69622 Villeurbanne cedex, France} 
 
 \begin{abstract}
Shadow systems are an approximation of reaction-diffusion-type problems obtained in the infinite diffusion coefficient limit. They allow reducing complexity of the system and hence facilitate its analysis.  The quality of approximation can be considered in three time regimes: (i) short-time intervals taking account for the initial time layer, (ii) long-time intervals scaling with the diffusion coefficient and tending to infinity for diffusion tending to infinity, and (iii) asymptotic state for times up to $T = \infty$.  In this paper we focus on uniform error estimates in the long-time case. Using linearization at a time-dependent shadow solution, we derive sufficient conditions for control of the errors. The employed methods are cut-off techniques and $L^p$-estimates combined with stability conditions for the linearized shadow system.  Additionally, we show that the global-in-time extension of the uniform error estimates may fail without stronger assumptions on the model linearization. The approach is presented on example of reaction-diffusion equations coupled to ordinary differential equations (ODEs), including classical reaction-diffusion system. The results are illustrated by examples showing necessity and applicability of the established conditions.
\end{abstract}

% keywords (maximal 6)

\end{frontmatter}

%\linenumbers

%\tableofcontents %to be deleted before submission

\section{Introduction} 

Reaction-diffusion systems  with zero flux boundary conditions and  a large diffusion coefficient can be studied by means of their model reduction as diffusion tends to infinity. %Such a limit is applied in various areas of applied mathematics including 
The infinite diffusion limit is used in various areas of applied mathematics such as recently in control theory \cite{Hernandez, Zuazua} and optimization \cite{Mazari}. The resulting model reduction, called \emph{shadow limit}, is classically used to investigate existence and stability of stationary solutions of reaction-diffusion systems with a large ratio of diffusion rates \cite{He3, Miy05, Takagi, Wei}. Furthermore, a shadow limit reduction facilitates the analysis of the spatio-temporal evolution of solutions of reaction-diffusion systems \cite{Hale89, Kondo}. Originally considered for a system of two coupled reaction-diffusion equations, starting from the works \cite{Hale89, Keener78, Nishiura, Takagi}, the shadow limit has attracted a considerable interest also in the case of reaction-diffusion-ODE systems \cite{Bobrowski2, Hernandez, KMCMlinear, dynspike, AMCAM:2015A}. Many ecological or biological applications lead to models coupling reaction-diffusion equations with ordinary differential equations, e.g., \cite{Hock, Klika, Marasco, MCK08}. The non-diffusing components may lead to a lower spatial regularity of model solutions and even loss of stability of regular stationary solutions \cite{Steffen, KMCT20}.\\

In this paper, we focus on a rigorous proof of a large diffusion limit for reaction-diffusion-ODE models, pursuing the work \cite{KMCMlinear} and extending the analysis to nonlinear models {on long-time intervals}. This work includes classical reaction-diffusion systems as a special case of reaction-diffusion-ODE systems. Our results provide understanding of the long-term dynamics of the reaction-diffusion-ODE system from results obtained for its shadow limit.\\

%Classical shadow systems are considered in \cite{Kondo} as well as recently in \cite{Zuazua} in the context of control theory and also in the context of optimization \cite{Mazari}. The infinite diffusion limit also applies in the context of control theory \cite{Zuazua} and optimization \cite{Mazari}.  

%We consider reaction-diffusion-ODE systems, including classical reaction-diffusion models. 
We assume that $\Omega \subset \mathbb{R}^n$ is a bounded domain with Lipschitz boundary $\partial \Omega \in C^{0,1}$ and consider the following quasilinear system (vector-valued quantities and operators are highlighted in bold print)
\begin{align} \label{fullsys}
\begin{split} 
     \frac{\partial \mathbf{u}_\varepsilon}{\partial t} - \mathbf{D} \Delta  \mathbf{u}_{\varepsilon} &= \mathbf{f}( \mathbf{u}_{\varepsilon} , v_{\varepsilon} , x,t) \quad \text{in} \quad \Omega_T:=\Omega \times (0,T), \qquad   \mathbf{u}_{\varepsilon} (\cdot, 0) =\mathbf{u}^0  \quad \text{in} \quad \Omega, \\
        \frac{\partial v_{\varepsilon}}{\partial t} - \frac{1}{\varepsilon} \Delta v_\varepsilon &=   g ( \mathbf{u}_\varepsilon , v_\varepsilon , x,t)  \quad \text{in} \quad \Omega_T, \hspace{2.8cm} \, v_\varepsilon (\cdot, 0) =v^0 \quad \text{in} \quad \Omega, \\
 \frac{\partial \mathbf{u}_{\varepsilon}}{\partial \nu} & =\mathbf{0},  \quad \frac{\partial v_\varepsilon}{\partial \nu} =0 \quad \text{on} \quad \partial \Omega \times (0,T),
 \end{split}
\end{align}
where $\mathbf{D} \in \mathbb{R}_{\ge 0}^{m \times m}$ is a diagonal matrix with non-negative entries $D_i \ge 0$ for $i=1, \dots, m$ and $\varepsilon>0$ is a small parameter. This problem consists of a vector-valued function $\mathbf{u}_\varepsilon$, possibly with non-diffusing components for which $D_i=0$, and a scalar diffusing component $v_\varepsilon$ with a large diffusion coefficient $1/\varepsilon$. With a slight abuse of notation, the boundary conditions mean that each diffusing component of system \eqref{fullsys} satisfies zero-flux boundary conditions. \\ 

The shadow limit reduction of system \eqref{fullsys} yields the following system of integro-differential equations
\begin{align}\label{shadow}
\begin{split}
\frac{\partial \mathbf{u}}{\partial t}- \mathbf{D} \Delta  \mathbf{u} &= \mathbf{f}( \mathbf{u} , v, x,t ) \qquad \qquad \quad \mbox{in} \quad \Omega_T, \quad  \;  \mathbf{u} (\cdot, 0) =\mathbf{u}^0  \quad \text{in} \quad \Omega, \quad \; \frac{\partial  \mathbf{u}}{\partial \nu}=0 \quad \text{on} \quad \partial \Omega \times (0,T),\\
   \frac{\mathrm{d} v}{\mathrm{d} t} &=\langle g ( \mathbf{u} (\cdot, t) , v (t) , \cdot,t) \rangle_\Omega \quad \mbox{in} \quad (0,T), \quad v(0) = \langle v^0 \rangle_\Omega,
   \end{split}
\end{align}
where we abbreviate the spatial mean value of a function $z \in L^1(\Omega)$ by
\[
\langle z \rangle_\Omega = \frac{1}{|\Omega|} \int_\Omega z(x) \; \mathrm{d}x.
\]
For reaction-diffusion-ODE models, the shadow limit has been studied in \cite{Bobrowski2} and independently in \cite{KMCMlinear, AMCAM:2015A}, showing an error estimate in terms of the large diffusion parameter $1/\varepsilon$. In both approaches, the obtained estimates depend exponentially on the length of the time interval, see \cite{KMCMlinear} for details.  This means that for long time intervals the error estimate deteriorates significantly and no conclusion on the long-time behavior of solutions to system \eqref{fullsys} can be drawn from the behavior of its shadow limit reduction \eqref{shadow}. \\

The current paper is devoted to analysis of the uniform shadow-limit approximation  for long-time intervals, i.e., for time intervals $(0,T)$ of the length $T\sim \varepsilon^{-\ell}$ for some $\ell>0$. Additionally, we study the approximation on the global time scale for times up to $T= \infty$. We provide error estimates in terms of powers of the inverse $\varepsilon$ of the large diffusion parameter, which are uniform with respect to space and time. Departing from the analysis of a linear reaction-diffusion-ODE system with space-independent coefficients, considered in \cite{KMCMlinear}, we approach the nonlinear problem using linearization of the system of errors at the shadow solution. This step requires extending the analysis of \cite{KMCMlinear} to a system with space-dependent coefficients and formulating sufficient stability conditions on the linearization. Such an extension allows considering heterogeneous environments as well as explicit time-dependence of the nonlinearities in system \eqref{fullsys} and \eqref{shadow}, such as \cite{He3, Peng, Wang}. Moreover, we may include the case of classical reaction-diffusion systems
for which, to the best knowledge of the authors, a similar quantitative convergence result on long-time scales is not known.\\

The paper is organized as follows. After a brief description of the used approach and assumptions in Section \ref{sec:main}, the main results concerning uniform error estimates on long-time intervals are given in Theorem \ref{TheoremToy} and Theorem \ref{Theorem}. We introduce the notion of a mild solution for problem \eqref{fullsys} and \eqref{shadow} and provide preliminary results in Section \ref{sec:preliminary}. A detailed proof of the main results is given in Section \ref{sec:proofs}. Section \ref{sec:stabverification} is devoted to possible analytical ways of verifying the used stability conditions, including the case of a stationary shadow limit and a dissipativity condition. In Section \ref{sec:modelex}, a linear model and a predator-prey system exemplify the results. Essential properties concerning solutions of the heat equation are deferred to the Appendix.

%%%%%%%%%%%%%%%%%%%%%%%%%%%%%%%%%%%%%%%%%%%%%%%%%%%%%%%%%%%%%%%%%%%%%%%%%%%%%%%%%%%%%%%%%%%%%%%%

\section{Main results} \label{sec:main}

In this work, solutions of the reaction-diffusion-ODE system \eqref{fullsys} are compared to the corresponding solution of the shadow limit \eqref{shadow} with respect to the norm $L^\infty(\Omega_T)$ where $\Omega_T= \Omega \times (0,T)$. The choice of $L^\infty(\Omega_T)$ is suitable for bounded, discontinuous solutions of a wide range of nonlinear problems which can be solved using the method of Rothe \cite{Rothe}. Such a choice is motivated by spatial discontinuities of solutions of reaction-diffusion-ODE systems \cite{Steffen, KMCT20}. We consider {mild solutions given by an implicit integral equation} of the reaction-diffusion-type system and its shadow limit. At the same time, boundary regularity of the spatial domain $\Omega$ is relaxed to Lipschitz regularity which is necessary for %validity of Sobolev embeddings, existence of the heat semigroup and 
a proper notion of the boundary condition. The semigroup framework used in this work allows applying the uniform shadow limit approximation to space- and time-dependent reaction-diffusion-type problems of the form \eqref{fullsys}. \\ %heterogeneous environments 
The presented analysis involves several technical challenges. These partly arise from low regularity of model parameters, initial data or the considered domain. Furthermore, the spatial mean values in the integro-differential system \eqref{shadow} imply that the shadow system is a singular limit of the diffusing system \eqref{fullsys}. Hence, in the error estimates, we involve a {correction term} $\psi_{\varepsilon}$. This term includes the initial time layer which originates from different initial values $v^0$ and $\langle v^0 \rangle_\Omega$. Moreover, there is a discrepancy between nonlinearities of both systems {for which the correction term $\psi_{\varepsilon}$ takes account (English correct?) to streamline the analysis (for an easy linearization of $g(u_\varepsilon) - g(u)$ instead of $g(u_\varepsilon)- \langle g (u) \rangle_\Omega$)}. The precise definition and properties of the mean value correction $\psi_\varepsilon$ can be found in Section \ref{sec:meancorrect}. 

\subsection*{Main assumptions and approach}

To describe the main assumptions of the following results, we provide an overview of the used approach. When comparing solutions $(\mathbf{u}_{\varepsilon}, v_\varepsilon)$ of the reaction-diffusion-ODE system \eqref{fullsys} with the solution $(\mathbf{u}, v)$ of the shadow system \eqref{shadow}, estimates for the error $\mathbf{U}_\varepsilon = \mathbf{u}_{\varepsilon} - \mathbf{u}$ and $V_\varepsilon = {v}_{\varepsilon}- {v}- \psi_{\varepsilon}$ are established with respect to the $L^\infty(\Omega_T)$ norm. Similar to \cite[Section 3]{AMCAM:2015A}, the system 
\begin{align} \label{errorsys}
\begin{split}
\frac{\partial \mathbf{U}_\varepsilon }{\partial t} - \mathbf{D} \Delta  \mathbf{U}_{\varepsilon}  & = \mathbf{f} (\mathbf{u}_\varepsilon,  v_\varepsilon, x, t)- \mathbf{f} (\mathbf{u}, v, x, t) \quad  \mbox{in} \quad  \Omega_T, \qquad \mathbf{U}_\varepsilon (\cdot, 0)  = \mathbf{0} \quad  \mbox{in} \quad  \Omega,  \\
\frac{\partial V_\varepsilon}{\partial t} -\frac{1}{\varepsilon} \Delta V_\varepsilon & = g(\mathbf{u}_\varepsilon, v_\varepsilon, x,t)-g(\mathbf{u},v,x,t) \quad \mbox{in} \quad  \Omega_T, \qquad  \; V_\varepsilon (\cdot,0) = 0 \quad  \mbox{in} \quad  \Omega,\\
 \frac{\partial \mathbf{U}_{\varepsilon}}{\partial \nu} & =\mathbf{0},  \quad \frac{\partial V_\varepsilon }{\partial \nu}  = 0 \; \; \mbox{on} \; \;  \partial \Omega \times (0,T)  
  \end{split}
\end{align}
of errors is linearized at a given shadow limit $(\mathbf{u},v)$, hence, we assume sufficient regularity.

\begin{assumption}[Nonlinearity]
 \label{AssumptionN}  Let $\mathbf{f}, g$ be continuously differentiable with respect to $(\mathbf{u},v) \in \mathbb{R}^{m+1}$. For every fixed $(\mathbf{u},v) \in \mathbb{R}^{m+1}$, $\mathbf{h} = (\mathbf{f}, g)$ and its derivatives $\nabla_u f_i, \nabla_v f_i, i=1, \dots, m,$ and $\nabla_u g, \partial_v g$ are measurable functions in the variables $(x,t) \in {\overline \Omega} \times \mathbb{R}_{\ge 0}$. Furthermore, $\mathbf{h}$ and all derivatives $\nabla_u f_i, \nabla_v f_i$ and $\nabla_u g, \partial_v g$ satisfy the following local Lipschitz condition: For every bounded set $B \subset \mathbb{R}^{m+1} \times \overline{\Omega} \times \mathbb{R}_{\ge 0}$ there exists a constant $L(B)>0$ such that for all $(\mathbf{u},v,x,t), (\mathbf{y}, z,x,t) \in B$
   \begin{align} \label{Loclip}
   \begin{split}
   |\mathbf{h}(\mathbf{u},v, x,t)| & \le L(B),\\
    |\mathbf{h}(\mathbf{u}, v, x,t)-\mathbf{h}(\mathbf{y}, z, x,t)| & \le L(B)\left( |\mathbf{u}- \mathbf{y}| + |v -z| \right). 
        \end{split}
   \end{align}
\end{assumption}

\noindent For instance, such a local Lipschitz regularity is satisfied for an autonomous, i.e., space- and time-independent, right-hand side $(\mathbf{f}, g)$ which is of class $C^2$ with respect to the unknown variables $(\mathbf{u}, v)$. \\
Furthermore, as we are interested in long-time behavior of the model solutions, we assume a globally defined, uniformly bounded shadow solution $(\mathbf{u},v)$ with the following properties. 

 \begin{assumption}[Bounded shadow limit]
\label{AssumptionB} For bounded initial datum $(\mathbf{u}^0,  v^0) \in L^\infty (\Omega)^{m+1}$, the solution $(\mathbf{u},v)$ of the shadow limit \eqref{shadow} is globally defined and uniformly bounded, i.e., $\mathbf{u} \in L^\infty(\Omega \times \mathbb{R}_{\ge 0})^m$ and $v \in L^\infty(\mathbb{R}_{\ge 0})$. Furthermore, let $g$ and the linearized parts $\nabla_{(\mathbf{u},v)} \mathbf{f}$, $\nabla_{(\mathbf{u},v)} g$, evaluated at the shadow solution, be uniformly bounded in $(x,t) \in \overline{\Omega} \times \mathbb{R}_{\ge 0}$. 
 \end{assumption}
 
\noindent We refer to Section \ref{sec:shadow} for well-posedness of the shadow problem and for a discussion of Assumption \ref{AssumptionB}. Note that Assumption \ref{AssumptionB} excludes a consideration of a shadow limit which is unbounded in time or even blows up in finite time such as in \cite{dynspike}. Moreover, Assumptions \ref{AssumptionN}--\ref{AssumptionB} imply that the correction term $\psi_{\varepsilon}$ is negligible for larger times, see Lemma \ref{decayestmean}.  \\

In order to estimate solutions $\mathbf{U}_\varepsilon = \mathbf{u}_{\varepsilon} - \mathbf{u}$ and $V_\varepsilon = {v}_{\varepsilon}- {v}- \psi_{\varepsilon}$ of the nonlinear system \eqref{errorsys}, it is necessary to control the growth of the nonlinear remainder. For this we apply a cut-off technique similar to \cite{AMCAM:2015A}. This consists of restricting the remaining nonlinear part to a bounded region to ensure that the linear part determines the evolution of the whole nonlinear system \eqref{errorsys}. Denoting the truncated nonlinear remainders by $\mathbf{F}_\varepsilon, G_\varepsilon$, this procedure leads us to a truncated problem of the form
\begin{align} \label{truncsys}
\begin{split}
\frac{\partial \bm{\alpha}_{\varepsilon} }{\partial t} - \mathbf{D} \Delta  \bm{\alpha}_{\varepsilon}  & =\nabla_\mathbf{u} \mathbf{f} \cdot \bm{\alpha}_{\varepsilon} +
 \nabla_v \mathbf{f} \cdot (\beta_\varepsilon +  \psi_\varepsilon )  + \mathbf{F}_\varepsilon (\bm{\alpha}_{\varepsilon} , \beta_\varepsilon,x,t) \quad \mbox{in} \quad \Omega_T, \qquad  \bm{\alpha}_{\varepsilon} (\cdot, 0)   = \mathbf{0} \quad \text{in} \quad \Omega, \\
\frac{\partial \beta_\varepsilon}{\partial t} - \frac{1}{\varepsilon} \Delta \beta_\varepsilon & =  \nabla_\mathbf{u} g \cdot  \bm{\alpha}_{\varepsilon} +
 \nabla_v g  \cdot (\beta_\varepsilon+ \psi_\varepsilon) + G_\varepsilon (\bm{\alpha}_{\varepsilon}, \beta_\varepsilon, x,t)  \quad  \mbox{in} \quad  \Omega_T,   \qquad \, \beta_\varepsilon (\cdot,0) = 0 \quad \text{in} \quad \Omega, \\
  \frac{\partial \bm{\alpha}_{\varepsilon}}{\partial \nu}  & =\mathbf{0},  \quad  \frac{\partial \beta_\varepsilon }{\partial \nu} =0 \quad \mbox{on} \quad \partial \Omega \times (0,T).    
 \end{split}
 \end{align}
Herein, we abbreviated the Jacobians $\nabla_\mathbf{u} \mathbf{f}, \nabla_v \mathbf{f}, \nabla_\mathbf{u} g, \nabla_v g$ which are evaluated at the shadow solution $(\mathbf{u}, v)$ and depend in general on space and time. By Assumption \ref{AssumptionB}, these quantities are uniformly bounded. Basic properties and estimates of the solution $(\bm{\alpha}_{\varepsilon}, \beta_\varepsilon)$ and the precise choice of the truncation in $(\mathbf{F}_\varepsilon, G_\varepsilon)$ are given in Section \ref{sec:truncation}. Moreover, the following strategy of the proof is depicted in Figure \ref{figstrategy}. Due to exponential decay of the heat semigroup for functions with zero mean, the component $\beta_\varepsilon- \langle \beta_\varepsilon \rangle_\Omega$ can be estimated in terms of $\varepsilon$ and the remaining functions $\bm{\alpha}_\varepsilon$ and $\langle \beta_\varepsilon \rangle_\Omega$.  To control the truncated error $\bm{\alpha}_\varepsilon$ and the spatial mean $\langle \beta_\varepsilon \rangle_\Omega$, we focus on the linearization of the corresponding nonlinear shadow problem for $(\bm{\alpha}_\varepsilon, \langle \beta_\varepsilon \rangle_\Omega)$. Due to truncation, the nonlinear remainder is small in terms of $\varepsilon$ and the behavior of the linear system determines the behavior of the full nonlinear system. Hence, we formulate a stability condition on the linear shadow system
\begin{align} \label{shadlinear}
\begin{split}
\frac{\partial \bm{\xi}_1}{\partial t} - \mathbf{D} \Delta \bm{\xi}_1 &= \nabla_\mathbf{u} \mathbf{f} \cdot  \bm{\xi}_1 +\nabla_v \mathbf{f} \cdot  \xi_2 \qquad \qquad \text{in} \quad \Omega \times \mathbb{R}_{>0}, \qquad  \bm{\xi}_1 (\cdot, 0)  =\bm{\xi}_1^0  \quad \text{in} \quad \Omega,\\
   \frac{\mathrm{d} \xi_2}{\mathrm{d} t} &= \langle \nabla_\mathbf{u} g \cdot  \bm{\xi}_1\rangle_\Omega + \langle \nabla_v g \cdot  \xi_2 \rangle_\Omega  \quad \, \mbox{in} \quad \mathbb{R}_{>0}, \qquad \qquad \; \; \, \xi_2 (0) =  \xi_2^0
   \end{split}
\end{align}
endowed with zero Neumann boundary conditions for diffusing components of $\bm{\xi}_1$. The unique solution $\bm{\xi}$ of the homogeneous linear shadow problem \eqref{shadlinear} induces an evolution system, denoted by $\mathcal{W}$, with evolution operators  $\mathbf{W}(t,s)$ for $t,s \in \mathbb{R}_{\ge 0}, s \le t,$ defined by $\bm{\xi}(\cdot, t) = \mathbf{W}(t,s) \bm{\xi}(\cdot, s)$. To control the solutions of the nonlinear system, we assume the following condition on the linearized shadow system \eqref{shadlinear}.

\begin{assumption}[Stability of linearized shadow system]
 \label{AssumptionL1p} The shadow evolution system $\mathcal{W}$ is uniformly stable on $L^p(\Omega)^m \times \mathbb{R}$, i.e., there exist constants $1 \le p \le \infty$ and $C>0, \sigma \ge 0$ which are independent of time such that for all times $s, t \in \mathbb{R}_{\ge 0}, s \le t,$ there holds
\begin{align*}
\|\mathbf{W}(t,s) \bm{\xi}^0\|_{L^p(\Omega)^m \times \mathbb{R}} &\le C \mathrm{e}^{-\sigma (t-s)} \|\bm{\xi}^0\|_{L^p(\Omega)^m \times \mathbb{R}} \qquad \forall \; \bm{\xi}^0 = (\bm{\xi}^0_1, {\xi}^0_2) \in L^p(\Omega)^m \times \mathbb{R}.
\end{align*}
\end{assumption}

\noindent Note that the case $\sigma>0$ corresponds to uniform exponential stability of the evolution system $\mathcal{W}$. We refer to Section \ref{sec:stabverification} for a detailed discussion of this stability assumption which is crucial for the following results. Moreover, a brief description of the results of Section \ref{sec:stabverification} is given at the end of the next paragraph.\\ %we provide a way to verify Assumption \ref{AssumptionL1p} in the case of a stationary shadow limit and in the case when system \eqref{shadlinear} is dissipative. 

Concerning Assumption \ref{AssumptionL1p}, we distinguish between uniform stability in the space $L^\infty(\Omega)^m \times \mathbb{R}$ and in the spaces $L^p(\Omega)^m \times \mathbb{R}$ for finite $p< \infty$.

\subsection*{Stability in $L^\infty(\Omega)^m \times \mathbb{R}$}

On the basis of stability condition \ref{AssumptionL1p}, uniform estimates of the truncated errors $(\bm{\alpha}_\varepsilon, \beta_\varepsilon)$ are obtained within the proof. Finally, for an appropriate choice of the truncation, we gain estimates for the original errors $(\mathbf{U}_\varepsilon , V_\varepsilon)$ for all sufficiently large diffusion coefficients and a possibly restricted time interval.  
 
   \begin{theorem} \label{TheoremToy} Let $(\mathbf{f}, g)$ be twice continuously differentiable functions of $(\mathbf{u}, v)$ or, more generally, let Assumption \ref{AssumptionN} be fulfilled. Let Assumptions \ref{AssumptionB}--\ref{AssumptionL1p} hold with a stable evolution system $\mathcal{W}$ for $p= \infty$ and stability exponent $\sigma \ge 0$. Then there exist bounds $\alpha_0 \in (0,1)$, $\varepsilon_0 >0$ and a constant $C>0$ such that for any $\alpha \in [\alpha_0 , 1)$, $ \varepsilon \le \varepsilon_0$ and $T \le \varepsilon^{\alpha-1}$ we have
 \begin{align} \label{longestToy}
  \| \mathbf{u}_\varepsilon- \mathbf{u} \|_{L^\infty (\Omega_T)^m} +   \|  v_\varepsilon  - v - \psi_{\varepsilon} \|_{L^\infty (\Omega_T)} \leq C \varepsilon^{1-\alpha}.
 \end{align}
Moreover, if $\sigma>0$, we obtain the following global error estimate
 \begin{align} \label{globalestToy}
\|\mathbf{u}_\varepsilon - \mathbf{u}\|_{L^\infty(\Omega \times \mathbb{R}_{\ge 0})^m} +  \|v_\varepsilon-v-\psi_{\varepsilon}\|_{L^\infty(\Omega \times \mathbb{R}_{\ge 0})}   \le  C\varepsilon.
 \end{align}
 \end{theorem}
 
Theorem \ref{TheoremToy} is an extension of \cite[Theorem 3]{AMCAM:2015A} to long time ranges. The above result enables us to check various models for uniform approximation on long-time scales, see for instance the model examples in \cite{He3, Kondo, Miy05, Peng, Wei} or Section \ref{sec:modelex}. Such a result yields understanding of the long-term dynamics of the reaction-diffusion-ODE system \eqref{fullsys} based on results obtained for its shadow limit \eqref{shadow}. If the evolutionary system $\mathcal{W}$ is uniformly stable with respect to $L^\infty(\Omega)^m \times \mathbb{R}$, we obtain error estimates with explicit dependence on the time interval length $T$ that scales with the diffusion parameter $1/\varepsilon$. Stability implies estimate \eqref{longestToy} which is valid on $\Omega_T= \Omega \times (0,T)$ for $T \sim \varepsilon^{-\ell}$ and some $0< \ell <1$. Unfortunately, as \cite[Example 3]{KMCMlinear} shows, accuracy of the approximation for these transient states does not imply a valuable asymptotic approximation as $T \to \infty$.  However, assuming $\sigma>0$, i.e., uniform exponential stability of the linearized shadow system $\mathcal{W}$, yields the global error estimate \eqref{globalestToy} on $\Omega \times \mathbb{R}_{\ge 0}$. Note that the correction term $\psi_\varepsilon$ is of order $\varepsilon$ for all large times, see Lemma \ref{decayestmean}. Thus, it is negligible for a derivation of the long-time behavior of the solution of the reaction-diffusion-type system \eqref{fullsys} from the behavior of its shadow reduction. The proof of Theorem \ref{TheoremToy} is given in Section \ref{sec:errorestimates}. \\

Furthermore, Theorem \ref{TheoremToy} is an extension of the preceding study \cite{KMCMlinear} of a linear reaction-diffusion-ODE system with only time-dependent coefficients to the general nonlinear system \eqref{fullsys} with space- and time-dependent coefficients. In \cite{KMCMlinear}, however, a stability condition on the subsystem of non-diffusing components of the linear shadow system  \eqref{shadlinear} is used. To recall the latter condition, we delete all rows and columns of $\nabla_\mathbf{u} \mathbf{f}(\cdot,t)$ for which the diffusion coefficient is positive to obtain the function $t \mapsto \mathbf{A}_0(\cdot,t) \in L^\infty(\Omega)^{m_0 \times m_0}$ for some $0 \le m_0 \le m$. The solutions of the corresponding initial value problem
\begin{equation*}
\frac{\partial \bm{\psi}}{\partial t} = \mathbf{A}_0(\cdot, t) \bm{\psi} \qquad \text{in} \quad \Omega \times \mathbb{R}_{>0}, \qquad \bm{\psi}( \cdot, 0) = \bm{\psi}^0 \quad \text{in} \quad \Omega
\end{equation*}
define an evolution system $\mathbf{\mathcal{U}}$ consisting of evolution operators  $\mathbf{U}(t,s)$ on $L^\infty(\Omega)^{m_0}$ for $t,s \in \mathbb{R}_{\ge 0}$, $s \le t,$ given by $\bm{\psi}(\cdot, t) = \mathbf{U}(t,s) \bm{\psi}(\cdot, s)$ \cite[Ch. III, \S 1]{Daleckii}. In  \cite[Theorem 1]{KMCMlinear}, the following stability condition is used.

\begin{assumption}[Stability of ODE subsystem]
 \label{AssumptionL0} 
 The evolution system $\mathcal{U}$ is uniformly stable in $L^\infty(\Omega)^{m_0}$, i.e., there exist constants $C>0, \mu \ge 0$ independent of time such that 
\begin{align*}
\|\mathbf{U}(t,s) \bm{\psi}^0\|_{L^\infty(\Omega)^{m_0}} & \le C \mathrm{e}^{-\mu(t-s)} \|\bm{\psi}^0\|_{L^\infty(\Omega)^{m_0}} \qquad \forall \; \bm{\psi}^0 \in L^\infty(\Omega)^{m_0}, s,t \in \mathbb{R}_{\ge 0}, s \le t.
\end{align*} 
\end{assumption}

This stability condition is closely related to Assumption \ref{AssumptionL1p}. Although the case of a stationary shadow solution (see Proposition \ref{HardtLemma}) and the case of a dissipative linearized shadow problem (see Proposition \ref{Propdissfull}) show that Assumption \ref{AssumptionL1p} for some $1 \le p \le \infty$ implies Assumption \ref{AssumptionL0}, the authors could not prove whether this implication is true in general.  As shown in \cite{KMCMlinear}, Assumption \ref{AssumptionL0} is sufficient for space-independent problems. However, Model \ref{sec:expgrowthlin} in this work shows that problems with space-dependent shadow solutions require more assumptions than just uniform stability of the ODE subsystem $\mathcal{U}$, even in the linear case. This leads us to the stability condition \ref{AssumptionL1p} on the whole linearized shadow problem \eqref{shadlinear}. Note that Theorem \ref{TheoremToy} does not explicitly require the stability condition \ref{AssumptionL0}. As shown in the next paragraph, the stability assumption \ref{AssumptionL1p} for finite $p< \infty$ combined with Assumption \ref{AssumptionL0} on the non-diffusing components can be used to obtain a uniform error estimate.\\

%\subsection*{Verification of stability assumptions}

It is usually difficult to directly check the stability conditions \ref{AssumptionL1p}--\ref{AssumptionL0} for the linearization at the shadow limit solution. However, in the stationary case, stability properties of the linearization can be deduced from the spectrum of the corresponding linear operator. This is a consequence of the well-known spectral mapping theorem for analytic semigroups. The spectrum of the linearized shadow operator and its ODE subsystem is characterized for bounded stationary solutions in Section \ref{sec:spectral}. This allows not only to verify stability conditions \ref{AssumptionL1p}--\ref{AssumptionL0}, as shown in Corollary \ref{stability} and Remark \ref{rem:stabilityshadow}, but also allows general nonlinear stability results for bounded stationary solutions of the shadow system \eqref{shadow} obtained from a linear stability analysis. \\
Whereas numerical simulations can be employed to verify the stability assumptions in applications, we present another analytical possibility for time-dependent problems %of verifying the two stability assumptions described above.
in Section \ref{sec:stabverification}. As a particular case of stable systems we study dissipative linear shadow systems which imply the stability conditions \ref{AssumptionL1p}--\ref{AssumptionL0}. The notion of dissipativity is useful in particular if diffusion is involved in the corresponding
shadow system: We only impose conditions on the linear shadow part
without diffusion to obtain a stable shadow evolution system including diffusing components, see Corollary \ref{TheoremDissip}.

\subsection*{Stability in $L^p(\Omega)^m \times \mathbb{R}$ for $p< \infty$}

Stability condition \ref{AssumptionL1p} can also be considered in the space $L^p(\Omega)^m \times \mathbb{R}$ for sufficiently big $p<\infty$. In such case, parabolic $L^{p,r}$ estimates can be employed to conclude on $L^\infty$ bounds for the diffusing components, see Proposition \ref{thees}. However, for an $L^\infty$ estimate of the non-diffusing components, we additionally assume the uniform stability condition \ref{AssumptionL0} on the ODE subsystem as in \cite{KMCMlinear}. This procedure yields the second main result of this work.

 \begin{theorem} \label{Theorem} Let Assumptions \ref{AssumptionN}--\ref{AssumptionL0} hold for some $1 \le p< \infty$ and $\mu, \sigma \ge 0$ with $p>n/2$ if $n \ge 2$ and let $r \in (1, \infty)$ be given by the relation $1/r + n/(2p) <1.$ Then there exist bounds $\alpha_0 = \alpha_0(r) \in (0,1), \varepsilon_0 >0$ and a constant $C>0$ such that for any $\alpha \in [\alpha_0 , 1)$, $\varepsilon \le \varepsilon_0$ and times $T \le \varepsilon^{\alpha-1}$ we have the uniform estimates
 \begin{align} \label{longest}
 \begin{split}
  \| \mathbf{u}_\varepsilon- \mathbf{u} \|_{L^\infty (\Omega_T)^m} & \leq C \varepsilon^{3(1-\alpha)}, \\
 \|  v_\varepsilon  - v - \psi_\varepsilon \|_{L^\infty (\Omega_T)} & \leq C \varepsilon^{4(1-\alpha) },\\  
 \| \langle v_\varepsilon \rangle_\Omega - v \|_{L^\infty ((0,T))} & \leq C \varepsilon^{4(1-\alpha) }.
 \end{split}
 \end{align}
Moreover, if $\mu>0$, all components of $\mathbf{u}_\varepsilon - \mathbf{u}$ and  $v_\varepsilon  - v - \psi_\varepsilon$ can be estimated by $C\varepsilon^{4(1-\alpha) }$.
 \end{theorem}

Although estimates in \eqref{longest} are weaker than the corresponding result of Theorem \ref{TheoremToy}, they provide quantitative information on the convergence rate, i.e., the rate of how fast the diffusing solution $(\mathbf{u}_\varepsilon, v_\varepsilon)$ converges to its shadow limit $(\mathbf{u}, v)$ for large times. If one restricts to space dimensions $n \le 3$, we may consider the Hilbertian case $p=2$ which is usually studied in applications \cite{Kondo, Miy05, Wei}. Note that the additional Assumption \ref{AssumptionL0} does not have to be verified explicitly in the case of a stationary shadow solution and the case of a dissipative linearized shadow problem, see Section \ref{sec:stabverification}. The proof of Theorem \ref{Theorem} is given in Section \ref{sec:errorestimates}. \\
Necessity of the restriction of the time interval to $T \le \varepsilon^{\alpha -1}$ is already shown in the linear case in \cite[Example 3]{KMCMlinear}. The stability conditions \ref{AssumptionL1p}--\ref{AssumptionL0} are optimal in the sense that there are examples where uniform convergence on long-time intervals may not be achievable in the absence of uniform boundedness in $L^\infty(\Omega)^{m_0}$ or $L^p(\Omega)^m \times \mathbb{R}$, compare the linear examples \cite[Example 2]{KMCMlinear} and Model \ref{sec:expgrowthlin}. The restriction of the time interval in the case of an exponentially stable evolution system $\mathcal{W}$ is due to stability in $L^p(\Omega)^m \times \mathbb{R}$ for some finite $p< \infty$ instead of $p=\infty$. In the proof of Theorem \ref{Theorem} it is not possible to infer estimates on the global time scale as we employ a parabolic bootstrap argument.

 \begin{remark}
We refer to \cite[Ch. 7]{Kowall} for shadow systems with different boundary conditions or differential operators. {Although the semigroup framework used in this work seems very versatile, an adaption to similar problems on long-time scales is left as a future task. Formulation ok?}%we strongly believe that the results are adaptable to similar problems on long-time scales.  %consider classical shadow systems and reaction-diffusion-ODE systems simultaneously.
\end{remark}

%%%%%%%%%%%%%%%%%%%%%%%%%%%%%%%%%%%%%%%%%%%%%%%%%%%%%%%%%%%%%%%%%%%%

\section{Well-posedness and preliminary results}\label{sec:preliminary}

This section is devoted to basic properties of mild solutions of the main systems \eqref{fullsys} and \eqref{shadow} such as existence and uniqueness on short-time intervals. Furthermore, we define and study the mean value correction $\psi_\varepsilon$ and show that it is small in terms of $\varepsilon$ for all large times. 

\subsection{Reaction-diffusion-ODE problem} 
Let us recall short-time well-posedness of the reaction-diffusion-ODE system \eqref{fullsys} for nonlinearities
  \begin{equation*}
\mathbf{f}: \mathbb{R}^m \times \mathbb{R} \times \overline{\Omega} \times \mathbb{R} \to \mathbb{R}^m \qquad \text{and} \qquad g: \mathbb{R}^m \times\mathbb{R} \times \overline{\Omega} \times \mathbb{R} \to \mathbb{R}
\end{equation*} 
fulfilling Assumption \ref{AssumptionN}. Usually, local-in-time existence of a solution $\mathbf{u}_\varepsilon: \Omega_T \rightarrow \mathbb{R}^{m}$ and  $v_\varepsilon: \Omega_T \rightarrow \mathbb{R}$ follows from standard theory \cite{Ladyparab, Pazy, Rothe}. %Henry
Since we want to impose only low regularity on the spatial domain $\Omega$ and the solutions, we recollect some definitions and the main idea of the proof of \cite[Part II, Theorem 1]{Rothe} which also applies in this case. Therefore, we rewrite system \eqref{fullsys} as a system of $m+1$ reaction-diffusion-type equations:
\begin{align}
\frac{\partial \bm{\Psi}}{\partial t} - \mathbf{D}_\varepsilon \Delta \bm{\Psi}   = \mathbf{h}(\bm{\Psi},x,t) \quad \text{in} \quad \Omega_T, \quad \; 
\bm{\Psi}(\cdot, 0)  =(\mathbf{u^0}, v^0) \quad \text{in} \quad \Omega, \quad \;  \frac{\partial \bm{\Psi}}{\partial \nu}  =\mathbf{0} \quad \mbox{on} \quad  \partial \Omega \times (0,T) \label{Rotheset1}
\end{align}
Here, $\bm{\Psi}=(\mathbf{u}_\varepsilon, v_\varepsilon), \mathbf{h} = (\mathbf{f}, g)$, and $\mathbf{D}_\varepsilon = \mathrm{diag}(\mathbf{D}, \varepsilon^{-1}) \in \mathbb{R}_{\ge 0}^{(m+1) \times (m+1)}$ is a diagonal matrix. If we consider this system for the decoupled case, i.e., for $\mathbf{h}=\mathbf{0}$, the component $v_\varepsilon$ induces an analytic semigroup $(S_\Delta(\tau))_{\tau \in \mathbb{R}_{\ge 0}}$ on $L^p(\Omega)$ for each $p \in (1, \infty)$ by $v_\varepsilon( \cdot,t) = S_\Delta(t/\varepsilon) v^0$ for initial values $v^0 \in L^p(\Omega)$, see Proposition \ref{heathom}. This semigroup then can be restricted to $L^\infty(\Omega)$ independently of $p$ and we obtain a formal contraction semigroup on $L^\infty(\Omega)$ which is in not strongly continuous \cite[Part I, Lemma 2]{Rothe}. %However, due to the maximum principle, the semigroup $(S_\Delta(\tau))_{\tau \in \mathbb{R}_{\ge 0}}$ is contractive on $L^\infty(\Omega)$, i.e.,
%\[
%\|S_\Delta(\tau) z\|_{L^\infty(\Omega)} \le \|z \|_{L^\infty(\Omega)} \qquad \forall \; z \in L^\infty(\Omega), \tau \in \mathbb{R}_{\ge 0}.
%\]
Using this semigroup approach, we define similarly to Rothe, \cite[Part II, Definition 2]{Rothe}:

\begin{definition} \label{mildsol}
Let $0 < T \le \infty$. A mild solution of problem \eqref{Rotheset1} on the interval $[0, T)$ for initial values $\bm{\Psi}^0=(\mathbf{u}^0, v^0) \in L^\infty(\Omega)^{m+1}$ is a measurable function
$
\bm{\Psi}: \Omega \times [0,T) \to \mathbb{R}^{m+1}
$
satisfying for all $t \in (0,T)$
\begin{itemize}
\item[(i)] $\bm{\Psi}(\cdot,t) \in L^{\infty}(\Omega)^{m+1}$ and \, $\sup_{s \in (0,t)} \|\bm{\Psi}(\cdot,s)\|_{L^\infty(\Omega)^{m+1}} < \infty$,

\item[(ii)] the integral representation
\begin{align}
\bm{\Psi}(\cdot,t) = \mathbf{S}(t)\bm{\Psi}^0 + \int_0^t \mathbf{S}((t-s)) \mathbf{h}(\bm{\Psi}(\cdot,s), \cdot, s) \; \mathrm{d}s, \label{Rotheint}
\end{align}
\end{itemize}
where the integral is an absolute converging Bochner integral in $L^{\infty}(\Omega)^{m+1}$ and the semigroup $(\mathbf{S}(\tau))_{\tau \in \mathbb{R}_{\ge 0}}$ on $L^\infty(\Omega)^{m+1}$ has components $S_i(\tau)=S_\Delta(D_i \tau)$ for $i=1, \dots, m$ and $S_{m+1}(\tau) = S_\Delta(\tau/\varepsilon)$. Specifically for $D_i=0$, $S_i(\tau) = I$ is the identity on $L^\infty(\Omega)$ for all $\tau \in \mathbb{R}_{\ge 0}$.
\end{definition}

Following the proof of \cite[Part II, Theorem 1]{Rothe} we obtain a mild solution of problem \eqref{Rotheset1} if we assume the local Lipschitz condition \eqref{Loclip} for $\mathbf{h}$ and bounded initial data $\bm{\Psi}^0$. Although \cite{Rothe} uses higher regularity of the boundary $\partial \Omega$, the proof is analog, applying a Picard iteration on $L^\infty(\Omega)^{m+1}$ to the implicit integral representation \eqref{Rotheint} and using the semigroup $(\mathbf{S}(\tau))_{\tau \in \mathbb{R}_{\ge 0}}$ for $\partial \Omega \in C^{0,1}$. 

 \begin{proposition} \label{prop1} Let Assumption \ref{AssumptionN} and $(\mathbf{u}^0,  v^0) \in L^\infty (\Omega)^{m+1}$ hold. Then for each $\varepsilon>0$ there exists a maximal time interval $[0, T_{\max})$ for $T_{\max} = T_{\max}(\varepsilon) >0$, such that problem \eqref{fullsys} has a unique mild solution $\bm{\Psi}$ on $[0, T_{\max})$ satisfying $\bm{\Psi} = (\mathbf{u}_\varepsilon, v_\varepsilon) \in L^\infty (\Omega_T)^{m+1}$ for each $T < T_{\max}.$ Furthermore, the solutions $(\mathbf{u}_\varepsilon, v_\varepsilon)_{\varepsilon>0}$ locally exist on the same time interval, i.e., $\inf_{\varepsilon \in (0,1]} T_{\max}(\varepsilon) >0$.
 \end{proposition}

%\begin{remark}
The integral representation for $\mathbf{u}_\varepsilon$ yields temporal continuity of non-diffusing components in $t=0$, i.e., $u_{\varepsilon,i} \in C ([0,T ]; L^\infty (\Omega ))$. Choosing more regular initial data, nonlinearities and a more regular boundary%, i.e., $\partial \Omega \in C^{2,\alpha}$ and $\mathbf{u}^0 \in C^{0,\alpha}(\overline{\Omega})^m$ and $v^0  \in C^{2,\alpha}(\overline{\Omega})$ which satisfies the compatibility condition 
%\[
%\frac{\partial v^0} {\partial \nu} = 0 \qquad \mathrm{on} \quad \partial \Omega,
%\]
%then with some additional H\"older continuity of $\mathbf{f},g$ in $x$ and $t$
, we obtain H\"older continuous solutions that satisfy the equations in the classical sense \cite[Part II, Theorem 1]{Rothe}. %and \cite[pp. 7-8]{Ladyparab} for the corresponding parabolic H\"older spaces.
%\end{remark}

\subsection{Shadow problem} \label{sec:shadow}

Similar to the definition of a mild solution of the reaction-diffusion-ODE system \eqref{fullsys}, a mild solution of the corresponding shadow limit satisfies the following integral equations for $i=1, \dots,m$
\begin{align} \label{int}
\begin{split}
u_i (\cdot,t) &= S_\Delta(D_i t) u^0_i + \int_0^t S_\Delta(D_i (t-s)) f_i(\mathbf{u}(\cdot,s),v(s),\cdot,s) \; \mathrm{d}s,\\
v(t) & = \langle v^0 \rangle_\Omega + \int_0^t \langle  g(\mathbf{u}(\cdot,s),v(s),\cdot,s) \rangle_\Omega \; \mathrm{d}s.
\end{split}
\end{align} 
Under the assumptions of Proposition \ref{prop1}, an application of a Picard iteration yields a unique solution of equations \eqref{int}, compare also \cite[Theorem 1]{AMCAM:2015A}. 

\begin{proposition} \label{prop2} Let Assumption \ref{AssumptionN} and $(\mathbf{u}^0,  v^0) \in L^\infty (\Omega)^{m+1}$ hold. Then there exists a maximal time interval $[0, T_{\max})$ for $T_{\max} >0$, such that problem \eqref{shadow} has a unique mild solution $(\bm{u}, v)$ on $[0, T_{\max})$ satisfying $(\mathbf{u}, v) \in L^\infty (\Omega_T)^{m} \times C([0,T];\mathbb{R}) $ for each $T < T_{\max}.$ 
\end{proposition}

Let us note that, in general, diffusing components are not continuous up to $t=0$ with respect to $L^\infty(\Omega)$ \cite[Part I, Lemma 2]{Rothe}. However, non-diffusing components satisfy  $u_{i} \in C ([0,T ]; L^\infty (\Omega ))$ by  formula  \eqref{int}. \\

Assumption \ref{AssumptionB} essentially requires a globally defined, uniformly bounded solution of the shadow limit \eqref{shadow}. %Finally, using equation \eqref{int}, Assumption \ref{AssumptionB} implies $(\mathbf{u}, v) \in C(\mathbb{R}_{\ge 0}; L^\infty(\Omega)^m \times \mathbb{R})$ for the shadow limit. 
 Clearly,  if  there exist continuous functions $c_1, c_2: \mathbb{R}_{\ge 0} \to \mathbb{R}_{\ge 0}$ such that
  \begin{equation*} %\label{Growth}
     \|\mathbf{f} (\mathbf{y} , \cdot , t) \|_{L^\infty(\Omega)^m} + |
   \langle g(\mathbf{y},  \cdot , t) \rangle_\Omega |\leq c_1(t) + c_2(t) | \mathbf{y} | \qquad \forall \; \mathbf{y} \in \mathbb{R}^{m+1},
   \end{equation*}
then every maximal solution to the shadow problem \eqref{shadow} is global. The global existence can be, however, achieved also for weaker assumptions on the model nonlinearities and it has to be checked separately for specific models. Additionally, one has to verify uniform boundedness of the global solution for fulfillment of Assumption \ref{AssumptionB}. This task strongly depends on the considered model and different methods are used, e.g., see \cite{Kondo, Rothe} or Model \ref{sec:PPmodel}. We mention that it is often useful to consider the corresponding differential equation for the masses $(\langle \mathbf{u} \rangle_\Omega, v)$ in order to derive properties of the shadow limit $( \mathbf{u}, v)$ itself.  If, in addition to a uniformly bounded shadow limit, $g$ and the gradients of the nonlinearities are uniformly bounded in the space and time variable, then Assumption \ref{AssumptionB} is satisfied. The latter condition implies a bounded linearization to facilitate the analysis and, moreover, it implies a certain decay estimate of the mean value correction $\psi_\varepsilon$ which we study next.

\subsection{Mean value correction} \label{sec:meancorrect}

This problem is linked to the fact that in the shadow limit equation \eqref{shadow} only the mean value of $g$ resp. $v^0$ appears. We will use the mean value correction to compare both solutions, the diffusing solution and its shadow limit. In \cite{AMCAM:2015A}, it is distinguished between initial time layer and a mean value correction for $g$, but -- since we are interested in long-time behavior -- we combine them. In the remainder, we use the mean value correction $\psi_\varepsilon$ which solves 
\begin{align}
\begin{split} \label{Initlay} 
        \frac{\partial \psi_\varepsilon}{\partial t} - \frac{1}{\varepsilon}\Delta \psi_\varepsilon &=   g ( \mathbf{u}, v , x,t) - \langle g(\mathbf{u}, v, \cdot, t) \rangle_\Omega \quad \mbox{in} \quad \Omega\times (0, \infty), \qquad \frac{\partial \psi_\varepsilon}{\partial \nu} =0 \quad \mbox{on} \quad \partial \Omega \times (0,\infty),\\
 \psi_\varepsilon (\cdot, 0) & =v^0 - \langle v^0 \rangle_\Omega \quad \text{in} \quad \Omega, 
 \end{split}
\end{align} 
where $(\mathbf{u}, v)$ is the given shadow limit. By Assumption \ref{AssumptionB}, this correction term $\psi_\varepsilon$ can be estimated in terms of the small parameter $\varepsilon>0$. 

  \begin{lemma} \label{decayestmean}
Let Assumptions \ref{AssumptionN}--\ref{AssumptionB} be satisfied. Let $\lambda_1$ denote the first positive eigenvalue of $-\Delta$ endowed with zero Neumann boundary conditions on $L^2(\Omega)$. Then the mean value correction $\psi_\varepsilon$ satisfies the estimate 
\begin{equation}
\| \psi_\varepsilon(\cdot, t) \|_{L^{\infty}(\Omega)}  \le C_{v^0} \mathrm{e}^{-\lambda_1  t/\varepsilon}  + C_g \varepsilon \qquad \forall \; \varepsilon>0, t \in \mathbb{R}_{\ge 0} \label{indecay2}
\end{equation}
for some constants $C_{v^0}, C_g >0$ that do not depend on $t$ or $\varepsilon$, but on bounds of $v^0$ resp. $g$.
\end{lemma}

\begin{proof}
 The well-known Duhamel formula reads
\[
\psi_\varepsilon (\cdot, t) = S_\Delta(t/\varepsilon) (v^0-\langle v^0 \rangle_\Omega) + \int_0^t S_\Delta((t-s)/\varepsilon) \Big\{g ( \mathbf{u}(\cdot, s), v(s) , \cdot,s) - \langle g(\mathbf{u}(\cdot,s), v(s), \cdot, s) \rangle_\Omega \Big\} \; \mathrm{d}s.
\]
The mean value correction $\psi_\varepsilon$ fulfills $\langle \psi_\varepsilon(\cdot, t) \rangle_\Omega = 0$ for all $t \in \mathbb{R}_{\ge 0}$ such that the exponential decay estimate from Lemma \ref{Winterlemma} implies the estimate 
\begin{align*}
\| \psi_\varepsilon(\cdot, t) \|_{L^{\infty}(\Omega)} & \le \overline{C} \mathrm{e}^{-\lambda_1  t/\varepsilon} \| v^0-\langle v^0 \rangle_\Omega\|_{L^{\infty}(\Omega)} \\
& \quad + \overline{C}\int_0^t  \mathrm{e}^{-\lambda_1  (t-s)/\varepsilon} \| g ( \mathbf{u}(\cdot, s), v(s) , \cdot,s) - \langle g(\mathbf{u}(\cdot,s), v(s), \cdot, s) \rangle_\Omega \|_{L^\infty(\Omega)} \; \mathrm{d}s.
\end{align*}
Finally, by Assumption \ref{AssumptionB}, the right-hand side is bounded %at least for finite times,
 and integration yields the desired result \eqref{indecay2}.
\end{proof}

 In view of Lemma \ref{decayestmean}, the mean value correction $\psi_\varepsilon$ is of order $\varepsilon$ for times $t \ge T(\varepsilon)$ where 
\[
T(\varepsilon) := \max \left\{0, \frac{-\varepsilon \log \left(\|v^0-\langle v^0 \rangle_\Omega\|_{L^\infty(\Omega)} \varepsilon \right)}{\lambda_1} \right\}.
\] 
Note that $\lim_{\varepsilon \to 0} T(\varepsilon) =0$. This property allows to compare directly $v_\varepsilon$ and $v$ for large times if estimates as in Theorem \ref{TheoremToy} or Theorem \ref{Theorem} are valid. Moreover, these results imply that the component $v_\varepsilon$ becomes almost spatially homogeneous as time grows, compare the next result to \cite[Theorem 3]{Kondo}.

\begin{corollary} \label{Homogencor}
Let $(\mathbf{u}_\varepsilon, \mathbf{v_\varepsilon})$ and $(\mathbf{u}, v)$ be uniformly bounded, globally defined diffusing solutions of system \eqref{fullsys} and a corresponding shadow solution of system \eqref{shadow}, respectively. If $g-\langle g \rangle_\Omega$ evaluated at the solutions $(\mathbf{u}_\varepsilon, \mathbf{v_\varepsilon})$ and $(\mathbf{u}, v)$ is uniformly bounded in $\overline{\Omega} \times \mathbb{R}_{\ge 0}$, the error $V_\varepsilon = v_\varepsilon-v-\psi_\varepsilon$ satisfies the uniform estimate
\begin{equation}
\|V_\varepsilon-\langle V_\varepsilon \rangle_\Omega \|_{L^\infty(\Omega \times \mathbb{R}_{\ge 0})} \le C\varepsilon \label{VDhomogen}
\end{equation}
for some constant $C>0$ independent of diffusion. 
\end{corollary}

\begin{proof}
The difference $W_\varepsilon:= V_\varepsilon-\langle V_\varepsilon \rangle_\Omega$ solves 
\begin{align*}
\frac{\partial W_\varepsilon }{\partial t} - \frac{1}{\varepsilon}\Delta W_\varepsilon & =g(\mathbf{u}_\varepsilon, v_\varepsilon, x,t)-g(\mathbf{u},v,x,t)  -  \langle g(\mathbf{u}_\varepsilon, v_\varepsilon, \cdot ,t)-g(\mathbf{u},v,\cdot ,t) \rangle_\Omega  \quad \mbox{in} \quad  \Omega \times \mathbb{R}_{>0}
\end{align*} 
endowed with homogeneous zero flux boundary and zero initial conditions. Uniform boundedness of the right-hand side in the latter equation yields estimate \eqref{VDhomogen} by the same arguments as in Lemma \ref{decayestmean}. 
\end{proof}

\section{Convergence results}\label{sec:proofs}

Let $(\mathbf{u}_\varepsilon,  v_\varepsilon)$ and $(\mathbf{u}, v)$ be solutions of system \eqref{fullsys} and the corresponding shadow limit \eqref{shadow} satisfying Assumptions \ref{AssumptionN}--\ref{AssumptionB}. We consider the error functions
\begin{equation}\label{Error}
  \mathbf{U}_\varepsilon = \mathbf{u}_\varepsilon - \mathbf{u} \qquad \mbox{and} \qquad V_\varepsilon = v_\varepsilon - v - \psi_\varepsilon,
\end{equation}
where we took into account the mean value correction $\psi_\varepsilon$ defined by problem \eqref{Initlay}. Our goal is to estimate the error functions in terms of the small parameter $\varepsilon$. In Section \ref{sec:truncation}, we linearize the semilinear problem \eqref{fullsys} at a given shadow limit $(\mathbf{u},v)$ and follow the idea of truncation of the system of errors \eqref{errorsys}, as in the case of short-time intervals \cite{AMCAM:2015A}. This approach combined with stability conditions \ref{AssumptionL1p}--\ref{AssumptionL0} allows deducing the main results Theorem \ref{TheoremToy} and Theorem \ref{Theorem} in Section \ref{sec:errorestimates}.

\subsection{Truncation of errors}  \label{sec:truncation}

To assure boundedness of the model variables in the subsequent analysis, we localize the problem by introducing certain cut-off functions. This is similar to localization procedures in \cite{Yagi} or classical center manifold theory in \cite{Carr, Haragus} and allows estimating the solution. Let us start from system \eqref{errorsys} for the errors $\mathbf{U}_\varepsilon, V_\varepsilon$, i.e.,
\begin{align*} 
\begin{split}
\frac{\partial \mathbf{U}_\varepsilon }{\partial t}  - \mathbf{D} \Delta  \mathbf{U}_{\varepsilon}  & = \mathbf{f} (\mathbf{u}_\varepsilon,  v_\varepsilon, x, t)- \mathbf{f} (\mathbf{u}, v, x, t) \quad  \mbox{in} \quad  \Omega_T, \qquad \mathbf{U}_\varepsilon (\cdot, 0)  = \mathbf{0} \quad  \mbox{in} \quad  \Omega,  \\
\frac{\partial V_\varepsilon}{\partial t} -\frac{1}{\varepsilon} \Delta V_\varepsilon & = g(\mathbf{u}_\varepsilon, v_\varepsilon, x,t)-g(\mathbf{u},v,x,t) \quad \mbox{in} \quad  \Omega_T, \qquad  \; V_\varepsilon (\cdot,0) = 0 \quad  \mbox{in} \quad  \Omega,\\
 \frac{\partial \mathbf{U}_{\varepsilon}}{\partial \nu} & =\mathbf{0},  \quad \frac{\partial V_\varepsilon }{\partial \nu}  = 0 \; \; \mbox{on} \; \;  \partial \Omega \times (0,T).
  \end{split}
\end{align*}
 In this section, we introduce a cut-off in system \eqref{errorsys} in such a way that, in a small neighborhood of the origin, the cut-off does not affect the nonlinearities. % if one restricts to the trajectory of the solution starting in $\mathbf{0}$.
 Let $(\bm{\alpha}_{\varepsilon}, \beta_\varepsilon)$ be the solution of the truncated system of errors \eqref{truncsys} which we specify below on page \pageref{truncsys2}. Reducing on stability properties of the linearized shadow system, % allows deducing an error estimate of the truncated error functions. 
we will show in Proposition \ref{nonerrorest} and \ref{LinftynD} that the truncated solution $(\bm{\alpha}_{\varepsilon}, \beta_\varepsilon)$ is located within the small neighborhood of $\mathbf{0}$ with radius $\varepsilon^{\delta_0}$ for sufficiently large diffusion $1/\varepsilon$ and a certain $\delta_0>0$. The smallness of the truncated solution can be verified on a long-time interval which usually depends on the diffusion parameter. As solutions to problem \eqref{errorsys} are unique, this implies $(\bm{\alpha}_{\varepsilon}, \beta_\varepsilon) = (\mathbf{U}_\varepsilon, {V}_\varepsilon)$ in this small neighborhood of the origin, and we regain estimates for the original error $(\mathbf{U}_\varepsilon, {V}_\varepsilon)$ in Section \ref{sec:errorestimates}. \\

We start with a linearization of the above system of errors at the bounded shadow limit $(\mathbf{u}, v)$. Using Taylor's expansion, we write for $\mathbf{h}=(\mathbf{f}, g)$
\begin{align}
\mathbf{h} (\mathbf{u}_\varepsilon,  v_\varepsilon, x, t)- \mathbf{h} (\mathbf{u}, v, x, t)
& =  \nabla_\mathbf{u} \mathbf{h} (\mathbf{u}, v, x, t) \mathbf{U}_\varepsilon +
 \partial_v \mathbf{h} (\mathbf{u}, v , x, t) (V_\varepsilon+  \psi_\varepsilon )+ \mathbf{H} (\mathbf{U}_\varepsilon , V_\varepsilon +  \psi_\varepsilon ,x,t),\label{Taylor1}
\end{align}
where the remainder $\mathbf{H}=(\mathbf{F}, G)$ is given due to the mean value theorem by
\begin{align*}
\mathbf{H}(\mathbf{y}, z + \psi_\varepsilon,x,t) =  \int_0^1 \nabla \mathbf{h} (\mathbf{u} + \vartheta \mathbf{y}, v + \vartheta (z+ \psi_\varepsilon),x,t)- \nabla \mathbf{h} (\mathbf{u}, v,x,t) \; \mathrm{d}\vartheta \cdot \begin{pmatrix}
\mathbf{y}\\
z+ \psi_\varepsilon
\end{pmatrix}. %\label{Taylor2}
\end{align*}

Following the idea of truncation from \cite{AMCAM:2015A}, see also \cite{Carr, Haragus, Yagi}, we construct a suitable cut-off for the possibly unbounded right-hand side $\mathbf{H}= (\mathbf{F},G)$. We first modify the $\mathbf{y}$-component using a cut-off function $\Theta \in C^{0,1}(\mathbb{R})$ defined by
\begin{equation}\label{Ucut}
  \Theta (z) = \begin{cases}
\mathrm{sgn}(z) \cdot   \varepsilon^{\delta_0} & \text{for} \; |z| > \varepsilon^{\delta_0},\\
z & \text{for} \; |z| \leq \varepsilon^{\delta_0}
  \end{cases}
\end{equation}
for $\delta_0 \in (0, 1/2]$. The function $\Theta$ satisfies $|\Theta(z)| \le  \min \{\varepsilon^{\delta_0}, |z|\}$. For simplicity, in the vector-valued case, we write $\bm{\Theta}(\mathbf{y}) := (\Theta(y_1), \ldots, \Theta(y_m))$.
%\begin{remark}
Note that the choice of $\delta_0$ is restricted to the interval $(0, 1/2]$. This restriction is linked to the long-time error estimates in Proposition \ref{LinftynD} in which the convergence rate is weaker than in case of finite time analysis in \cite{AMCAM:2015A}.\\
%\end{remark}
We further consider the function $\overline{\mathbf{H}}_\varepsilon=(\overline{\mathbf{F}}_\varepsilon, \overline{G}_\varepsilon)$ given by $ \overline{\mathbf{H}}_\varepsilon (\mathbf{y},z, x, t)  :=  \mathbf{H}(\bm{\Theta}(\mathbf{y}), z + \psi_\varepsilon,x,t)$. Recall from Lemma \ref{decayestmean} that $\psi_\varepsilon$ is uniformly bounded in time. To control the $z$-component of the truncation, we use a symmetric cut-off function $\rho \in C_c^{\infty}(\mathbb{R}; [0,1])$ defined by
\begin{equation}\label{Tronc1}
  \rho (z ) = \begin{cases}
                      1 & \hbox{for} \; |z |\leq L, \\
                      0 & \hbox{for} \; |z|\geq 2L.
                    \end{cases}
\end{equation}
Here, $L:= C_{v^0} + 2$, where $C_{v^0}>0$ is the same constant from the time decay estimate \eqref{indecay2} of $\psi_\varepsilon$. Such a choice of $\rho$ is possible by mollifying the characteristic function $\mathds{1}_{(-r,r)}$ where $L<r< 2L$ \cite[Theorem 2.29]{Adams}. We define a further truncation by
\begin{align}\label{Tronc2}
  \mathbf{H}_\varepsilon (\mathbf{y},z, x, t) & := \overline{\mathbf{H}}_\varepsilon( \mathbf{y}, z,x,t) \cdot \left[  \rho \left(\frac{\varepsilon^{-\delta_0} |\overline{z}| }{L}\right)  \left(1- \rho \left( \frac{2\lambda_1t}{-\varepsilon \log \varepsilon} \right) \right) +\rho \left(\frac{|\overline{z}|}{L}\right)  \rho\left( \frac{2\lambda_1 t}{-\varepsilon \log \varepsilon} \right) \right]
\end{align}
for $\overline{z} := z + \psi_\varepsilon$. Recall that $\lambda_1$ is the first positive eigenvalue of $-\Delta$ endowed with zero flux boundary conditions on $L^2(\Omega)$, see problem \eqref{spect}. A similar modification is done in \cite[Lemma 2]{AMCAM:2015A}. Properties of the truncated function $\mathbf{H}_\varepsilon=(\mathbf{F}_\varepsilon, G_\varepsilon)$ are given next.

\begin{lemma} \label{quadcut} Let Assumptions \ref{AssumptionN}--\ref{AssumptionB} be satisfied. For each $\delta_0 \in (0,1/2]$ there is a constant $C>0$, independent of $\varepsilon \le 1$ but which depends on $L$ defined in \eqref{Tronc1}, such that for $\overline{z}= z + \psi_\varepsilon$ we have
\begin{equation}\label{Tronc3}
  |\mathbf{F}_\varepsilon (\mathbf{y}, z, x, t) |,   |G_\varepsilon (\mathbf{y}, z, x, t) |\leq C \Big(  \varepsilon^{2\delta_0}  + { \mathds{1}}_{\{  t\leq -\varepsilon \log \varepsilon/ \lambda_1 \}}(t) \cdot \min\{1,|\overline{z}|\} \Big)
\end{equation} 
for all  $(\mathbf{y}, z) \in \mathbb{R}^{m+1}, t \in \mathbb{R}_{\ge 0}$ and for a.e. $x\in \Omega$. Here, ${\mathds{1}}_{\{t \le c\}}$ is the characteristic function on the time interval $[0,c]$ and $\lambda_1$ is the first positive eigenvalue of $-\Delta$ endowed with zero Neumann boundary conditions on $L^2(\Omega)$.
\end{lemma}

\begin{proof} The first term in definition \eqref{Tronc2} can be estimated as $\overline{\mathbf{H}}_\varepsilon$. However, only $\varepsilon^{-\delta_0}|\overline{z}| \le 2L$ has to be considered due to $\rho \in C_c^{\infty}(\mathbb{R})$ defined in \eqref{Tronc1}. By Assumptions \ref{AssumptionN}--\ref{AssumptionB}, we infer the existence of a constant $C>0$ such that 
\begin{align*}
| \overline{\mathbf{H}}_\varepsilon (\mathbf{y} , z , \cdot, \cdot) | \le C |(\bm{\Theta}(\mathbf{y}), \overline{z})|^2.
\end{align*}
This constant $C$ depends on the time-independent bounds on $\mathbf{u}, v$ in Assumption \ref{AssumptionB}, on the one for $\psi_\varepsilon$ in estimate \eqref{indecay2}, and on Lipschitz bounds in Assumption \ref{AssumptionN} for the derivatives, but neither on diffusion nor on time $t$.  For the second term in definition \eqref{Tronc2}, we use the estimate
\[
\left|  \overline{\mathbf{H}}_\varepsilon( \mathbf{y}, z,x,t)  \rho \left(\frac{|\overline{z}|}{L}\right)  \rho \left( \frac{2\lambda_1 t}{-\varepsilon \log \varepsilon} \right) \right| \le C (\varepsilon^{2\delta_0} + |\overline{z}|^2) \rho \left(\frac{|\overline{z}|}{L}\right)   {\mathds{1}}_{\{  t\leq -\varepsilon \log \varepsilon/ \lambda_1 \}}(t)
\]
which results from the estimate for $\overline{\mathbf{H}}_\varepsilon$. The right-hand side is at most non-zero if $|\overline{z}| \le 2L$ and the quadratic term is (linearly) bounded.
\end{proof}

In the following, we study the localized problem using the truncation $\mathbf{H}_\varepsilon$ associated with the error system \eqref{errorsys}. Substituting $\mathbf{H}_\varepsilon$ for $\mathbf{H}$ in equation \eqref{Taylor1} and \eqref{errorsys} yields system \eqref{truncsys}, i.e., \label{truncsys2}
\begin{align*} 
\begin{split}
\frac{\partial \bm{\alpha}_{\varepsilon} }{\partial t} - \mathbf{D} \Delta  \bm{\alpha}_{\varepsilon}   & =\nabla_\mathbf{u} \mathbf{f} \cdot \bm{\alpha}_{\varepsilon} +
 \nabla_v \mathbf{f} \cdot (\beta_\varepsilon +  \psi_\varepsilon )  + \mathbf{F}_\varepsilon (\bm{\alpha}_{\varepsilon} , \beta_\varepsilon,x,t) \quad \mbox{in} \quad \Omega_T,  \\
\frac{\partial \beta_\varepsilon}{\partial t} - \frac{1}{\varepsilon} \Delta \beta_\varepsilon & =  \nabla_\mathbf{u} g \cdot  \bm{\alpha}_{\varepsilon} +
 \nabla_v g  \cdot (\beta_\varepsilon+ \psi_\varepsilon) + G_\varepsilon (\bm{\alpha}_{\varepsilon}, \beta_\varepsilon, x,t)  \quad  \mbox{in} \quad  \Omega_T,  \\
 \bm{\alpha}_{\varepsilon} (\cdot, 0)  & = \mathbf{0},  \quad \beta_\varepsilon (\cdot,0) = 0 \quad \text{in} \quad \Omega, \qquad   \frac{\partial \bm{\alpha}_{\varepsilon}}{\partial \nu}   =\mathbf{0},  \quad \frac{\partial \beta_\varepsilon }{\partial \nu} =0 \quad \mbox{on} \quad \partial \Omega \times (0,T).    
 \end{split}
 \end{align*}
Herein, we abbreviated the Jacobians $\nabla_\mathbf{u} \mathbf{f}, \nabla_v \mathbf{f}, \nabla_\mathbf{u} g, \nabla_v g$ which are evaluated at the shadow solution $(\mathbf{u}, v)$ and depend in general on space and time. Existence of a unique solution $(\bm{\alpha}_{\varepsilon}, \beta_\varepsilon)$ of system \eqref{truncsys} is deferred to Proposition \ref{coupledbeta}. Let us first describe the focal idea using cut-offs:

The cut-off is constructed in such a way that it has no effect within a small neighborhood of $\mathbf{0}$ with radius $\varepsilon^{\delta_0}$. Hence, we focus on finding estimates for the solution $(\bm{\alpha}_{\varepsilon}, \beta_\varepsilon)$ of the truncated problem \eqref{truncsys} to show that its solution is located within this small neighborhood for all large diffusions $1/\varepsilon$. This is motivated by the following sufficient condition for a removal of the truncation.

\begin{lemma} \label{Removetrunc}
Let Assumptions \ref{AssumptionN}--\ref{AssumptionB} hold, let $\delta_0 \in (0,1/2]$ be chosen as in the definition of the truncation \eqref{Tronc2} and let there exist an $\varepsilon_0>0$ and time $T>0$ such that the truncated error $(\bm{\alpha}_{\varepsilon}, \beta_\varepsilon)$ satisfies the estimate
\begin{equation}
\|\bm{\alpha}_{\varepsilon}\|_{L^\infty(\Omega_T)^m}  + \|\beta_\varepsilon\|_{L^\infty(\Omega_T)}\le \varepsilon^{\delta_0} \qquad \forall \; \varepsilon \le \varepsilon_0. \label{removetrunc}
\end{equation}
Then there exists some $\varepsilon_0 \ge \varepsilon_1>0$ such that $(\bm{\alpha}_{\varepsilon}, \beta_\varepsilon) =  (\mathbf{U}_\varepsilon, V_\varepsilon)$ on $\Omega_T$  for all $\varepsilon \le \varepsilon_1$ and the uniform error estimate \eqref{removetrunc} holds for the original error $(\mathbf{U}_\varepsilon, V_\varepsilon)$ on $\Omega_T$ for all $\varepsilon \le \varepsilon_1$.
\end{lemma}

\begin{proof}
 By definition \eqref{Ucut}, estimate \eqref{removetrunc} implies that $\bm{\Theta}$ is redundant in the definition of $\overline{\mathbf{H}}_\varepsilon$ and we obtain $\overline{\mathbf{H}}_\varepsilon = \mathbf{H}$ in $\Omega_T$. Next, we show  $\mathbf{H}_\varepsilon = \mathbf{H}$, i.e., the cut-off does not affect the right-hand side of the truncated problem \eqref{truncsys} if one restricts to the trajectory of the solution $(\bm{\alpha}_{\varepsilon}, \beta_\varepsilon)$. We deduce this by considering the two cases, above and below the critical time $t^\ast:= -\delta_0  \varepsilon \log \varepsilon/ \lambda_1$ for construction \eqref{Tronc2}. 
\begin{itemize}
\item If $t \le t^\ast$, we recall that $\rho(z) \equiv 1$ for all $|z| \le L$ and obtain $\rho\left( \frac{2\lambda_1 t}{-\varepsilon \log \varepsilon} \right) =1$ by $2\delta_0 \le L$. Thus, 
$
  \mathbf{H}_\varepsilon (\bm{\alpha}_{\varepsilon}, \beta_\varepsilon, x, t)  = \rho \left(|\overline{\beta_\varepsilon}|/L \right)  \mathbf{H}( \bm{\alpha}_{\varepsilon}, \beta_\varepsilon,x,t).
$
Additionally, there holds $\rho \left(|\overline{\beta_\varepsilon}|/L \right)  =1$ since by definition of $L=C_{v^0}+2$
\[
|\overline{\beta_\varepsilon}| \le |\beta_\varepsilon| + |\psi_\varepsilon | \le  \varepsilon^{\delta_0} + C_g \varepsilon + (L-2)\mathrm{e}^{-\lambda_1 t/\varepsilon}  \le L
\]
for all $\varepsilon \le \varepsilon_0$ small enough.

\item If $t>t^\ast$, then $\mathrm{e}^{-\lambda_1 t/\varepsilon} \le \varepsilon^{\delta_0}$ and thus, for small $\varepsilon$
\[
|\overline{\beta_\varepsilon}| \le |\beta_\varepsilon| + |\psi_\varepsilon | \le  \varepsilon^{\delta_0} + C_g \varepsilon + (L-2) \varepsilon^{\delta_0} \le L \varepsilon^{\delta_0}.
\]
Clearly, $|\overline{\beta}_\varepsilon| \le |\varepsilon^{-\delta_0} \overline{\beta}_\varepsilon| \le L$ and we find once again by definition \eqref{Tronc2} that
\begin{align*}
\mathbf{H}_\varepsilon (\bm{\alpha}_{\varepsilon}, \beta_\varepsilon, x, t) = \overline{\mathbf{H}}_\varepsilon( \bm{\alpha}_{\varepsilon}, \beta_\varepsilon,x,t) = \mathbf{H} (\bm{\alpha}_{\varepsilon}, \beta_\varepsilon, x, t).
\end{align*}
\end{itemize}
We have verified that both functions $(\bm{\alpha}_{\varepsilon}, \beta_\varepsilon)$ and $(\mathbf{U}_\varepsilon, V_\varepsilon)$ satisfy the same reaction-diffusion-type equation for all $\varepsilon \le \varepsilon_1$. By uniqueness of solutions to problem \eqref{truncsys}, we conclude $(\bm{\alpha}_{\varepsilon}, \beta_\varepsilon) = (\mathbf{U}_\varepsilon, V_\varepsilon)$ for $\varepsilon \le \varepsilon_1$ and estimate \eqref{removetrunc} %from Proposition \ref{LinftynD} are
is also valid for the original error functions $(\mathbf{U}_\varepsilon, V_\varepsilon)$ on the domain $\Omega_T$. 
\end{proof}

\begin{proposition}
 \label{coupledbeta}
Let Assumptions \ref{AssumptionN}--\ref{AssumptionB} hold and let $\varepsilon \le 1$. Then there exists a unique mild solution $(\bm{\alpha}_{\varepsilon},\beta_\varepsilon)$ of the truncated problem \eqref{truncsys} which exists globally-in-time. Furthermore, for all finite times $T>0$, we have $(\bm{\alpha}_{\varepsilon}, \beta_\varepsilon) \in L^\infty(\Omega_T)^{m+1}$ and $\beta_\varepsilon \in L^\infty(0,T;H^1(\Omega))$ is a weak solution with weak derivative $\partial_t \beta_\varepsilon \in L^2(\Omega_T)$. 
\end{proposition}

\begin{proof}
To apply Rothe's method as in Proposition \ref{prop1}, we write system \eqref{truncsys} as $m+1$ differential equations
\begin{align*}
\frac{\partial \bm{\Psi}_{\varepsilon}}{\partial t} - \mathbf{D}_\varepsilon \Delta \bm{\Psi}_{\varepsilon}  & = \mathbf{h}_\varepsilon(\bm{\Psi}_{\varepsilon},x,t) \quad \text{in} \quad \Omega_T, \qquad \bm{\Psi}_{\varepsilon}(\cdot, 0) =\mathbf{0} \quad \text{in} \quad \Omega.
\end{align*}
The diffusing components are endowed with zero Neumann boundary conditions and $\mathbf{h}_\varepsilon$ is given by 
\begin{align*}
\mathbf{h}_{\varepsilon}(\bm{\Psi}_{\varepsilon}, x,t) & = \mathbf{J}(x,t) \cdot \begin{pmatrix} \bm{\alpha}_{\varepsilon} \\
\beta_\varepsilon +  \psi_\varepsilon (x,t)
\end{pmatrix} + \begin{pmatrix}
\mathbf{F}_\varepsilon (\bm{\alpha}_{\varepsilon} , \beta_\varepsilon,x,t)\\
G_\varepsilon (\bm{\alpha}_{\varepsilon} , \beta_\varepsilon,x,t)
\end{pmatrix}.
\end{align*}
Here and in the sequel, we use the notation $\mathbf{J}$ for the Jacobian 
\begin{align}
\mathbf{J}(x,t) := \begin{pmatrix}
\mathbf{A}(x,t) & \mathbf{B}(x,t) \\
\mathbf{C}(x,t)  & D(x,t) 
\end{pmatrix}  := \begin{pmatrix}
\nabla_\mathbf{u} \mathbf{f}(\mathbf{u}(x,t), v(t), x,t) &  \nabla_v \mathbf{f}(\mathbf{u}(x,t), v(t), x,t)\\
 \nabla_\mathbf{u} g(\mathbf{u}(x,t), v(t), x,t) &  \nabla_v g(\mathbf{u}(x,t), v(t), x,t)
\end{pmatrix} \label{Jacobian}
\end{align}
evaluated at the shadow solution $(\mathbf{u}, v)$. Local Lipschitz continuity of $\mathbf{h}_\varepsilon$ in the variable $\bm{\Psi}_{\varepsilon}$ on bounded sets in $\overline{\Omega} \times \mathbb{R}_{\ge 0}$ carries over from $\mathbf{h}$ since $\mathbf{F}_\varepsilon$ and $G_\varepsilon$ are locally Lipschitz in the sense of estimates \eqref{Loclip}, see definition \eqref{Tronc2} of the truncation. Following the proof of Proposition \ref{prop1}, there exists a local-in-time mild solution in the sense of Definition \ref{mildsol}. We obtain the integral representation
\begin{align*}
\bm{\Psi}_{\varepsilon}(\cdot,t) = \int_0^t \mathbf{S}(t-\tau) \mathbf{h}_\varepsilon(\bm{\Psi}_{\varepsilon}(\cdot,\tau),\cdot,\tau) \; \mathrm{d}\tau.
\end{align*}
The function $\mathbf{h}_\varepsilon$ is linearly bounded in the variable $\bm{\Psi}_{\varepsilon}$ due to Lemma \ref{quadcut}. By the same reasoning as in \cite[Part II, Theorem 1]{Rothe}, this implies that no blow-up is possible and we obtain a unique mild solution with $\bm{\Psi}_{\varepsilon} \in L^{\infty}(\Omega_T)^{m+k}$ for all $T < \infty$. To improve regularity for the diffusing components, we apply parabolic $L^2$ theory performed in Proposition \ref{heatinhom}. Recall for this that $\beta_\varepsilon$ solves equation \eqref{eqzd} with $R \in L^{\infty}(\Omega_T)$.
%\[
%\frac{\partial \beta_\varepsilon}{\partial t} - \frac{1}{\varepsilon} \Delta \beta_\varepsilon  = R_\varepsilon \in L^{\infty}(\Omega_T)
%\]
%for an initial datum which is zero.
\end{proof}

The rest of this section is devoted to an estimation of the solution $(\bm{\alpha}_\varepsilon, \beta_\varepsilon)$ of the truncated system \eqref{truncsys}. Our strategy of the proof is depicted in Figure \ref{figstrategy}, starting in the left upper corner.\\
 In order to estimate the truncated solution $(\bm{\alpha}_{\varepsilon}, \beta_\varepsilon)$, we decompose $\beta_\varepsilon$ into its spatial mean $b_\varepsilon = \langle \beta_\varepsilon \rangle_\Omega$ and the residual $W_\varepsilon = \beta_\varepsilon - \langle \beta_\varepsilon \rangle_\Omega$ with $\langle W_\varepsilon \rangle_\Omega =0$ to exploit asymptotic properties of the heat semigroup \cite{Hale89, Winkler}. As already figured out in the linear case in \cite{KMCMlinear}, an additional stability condition is necessary to obtain long-time estimates. Assuming stability conditions \ref{AssumptionL1p}--\ref{AssumptionL0} for a linearized shadow system yields $L^\infty(\Omega_T)$ estimates for $(\bm{\alpha}_{\varepsilon}, \beta_\varepsilon)$ which are valid on long-time ranges $(0,T)$. Finally, using Lemma \ref{Removetrunc}, we return in Section \ref{sec:errorestimates} to estimates for the error functions $(\mathbf{U}_\varepsilon, V_\varepsilon)$ and prove Theorem \ref{TheoremToy} and Theorem \ref{Theorem}.

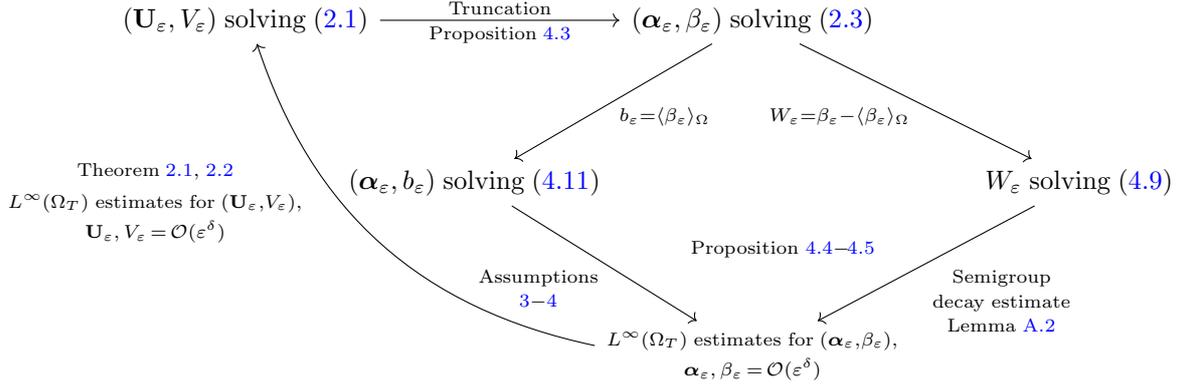
\begin{figure}[h]
\begin{center}
	\begin{tikzcd}
		& (\mathbf{U}_\varepsilon, V_\varepsilon) \; \text{solving} \; \eqref{errorsys} \arrow[rightarrow]{rr}{\text{Truncation}}[swap]{\text{Proposition} \; \ref{coupledbeta}}   
		& [-4em]
		&  [-3em] (\bm{\alpha}_{\varepsilon}, \beta_\varepsilon) \;   \text{solving} \; \eqref{truncsys}  \arrow[rightarrow]{rd}{}[swap]{W_\varepsilon = \beta_\varepsilon - \langle \beta_\varepsilon  \rangle_\Omega}  \arrow[rightarrow]{ld}{b_\varepsilon= \langle \beta_\varepsilon \rangle_\Omega}
		&  \mbox{}\\[6ex]
		&  [-4em]
		& (\bm{\alpha}_{\varepsilon}, b_\varepsilon) \;   \text{solving} \; \eqref{intalphaDbD} %
		\arrow[rightarrow]{rd}{ \qquad \quad \text{Proposition \ref{nonerrorest}} -}[swap]{\substack{\text{Assumptions} \\[0.5ex] \ref{AssumptionL1p}-\ref{AssumptionL0}}}
		&
		& W_\varepsilon 	\;   \text{solving} \; \eqref{betasplit}  
		 \arrow[rightarrow]{ld}{\substack{\text{Semigroup}\\[0.5ex] \text{decay estimate} \\[0.5ex] %\eqref{winlemmaapp}
		\text{Lemma \ref{Winterlemma}} }}[swap]{\substack{\ref{LinftynD} \quad \; \; \;  \\[-1ex] \mbox{}}} 
		\\[6ex]
		&  
		&
		& \substack{L^\infty(\Omega_T) \; \text{estimates for} \; (\bm{\alpha}_{\varepsilon}, \beta_\varepsilon),\\[0.5ex] \bm{\alpha}_{\varepsilon}, \, \beta_\varepsilon \, = \, \mathcal{O}(\varepsilon^\delta)} 
		\arrow[to=lluu, bend left, crossing over, near end, "\substack{\text{Theorem} \; \ref{TheoremToy}, \; \ref{Theorem}\\[1ex] L^\infty(\Omega_T) \; \text{estimates for} \; (\mathbf{U}_\varepsilon, V_\varepsilon),\\[0.5ex] \mathbf{U}_\varepsilon, \, V_\varepsilon \, = \, \mathcal{O}(\varepsilon^\delta)}"]
		& 
	\end{tikzcd}
\end{center}
       \caption{Strategy of the proof of error estimates.} \label{figstrategy}
\end{figure}

%To obtain estimates for very long time intervals of order $\varepsilon^{-\ell}$ for some $\ell>0$, we need a more detailed estimation of the quantities which are involved. We decompose $\beta_\varepsilon$ into its spatial mean $b_\varepsilon = \langle \beta_\varepsilon \rangle_\Omega$ and the residual $W_\varepsilon = \beta_\varepsilon - \langle \beta_\varepsilon \rangle_\Omega$ with $\langle W_\varepsilon \rangle_\Omega =0$ to exploit asymptotic properties of the heat semigroup. 
The differential equation of $\beta_\varepsilon$ from problem \eqref{truncsys} implies the following system
\begin{align} \label{betasplit}
\begin{split}
\frac{\partial W_\varepsilon }{\partial t} - \frac{1}{\varepsilon} \Delta W_\varepsilon & = \nabla_\mathbf{u} g \cdot \bm{\alpha}_{\varepsilon} - \langle \nabla_\mathbf{u} g \cdot \bm{\alpha}_{\varepsilon} \rangle_\Omega +   \nabla_v g \cdot b_\varepsilon - \langle \nabla_v g \cdot b_\varepsilon \rangle_\Omega  \\
& \quad  +   \nabla_v g \cdot (W_\varepsilon + \psi_\varepsilon) - \langle \nabla_v g \cdot (W_\varepsilon + \psi_\varepsilon)  \rangle_\Omega   + G_\varepsilon - \langle G_\varepsilon \rangle_\Omega \quad  \; \; \, \mbox{in} \quad  \Omega \times \mathbb{R}_{>0},\\
\frac{\mathrm{d} b_\varepsilon}{\mathrm{d} t} & =  \langle \nabla_\mathbf{u} g \cdot \bm{\alpha}_{\varepsilon} \rangle_\Omega +  \langle \nabla_v g \cdot b_\varepsilon \rangle_\Omega  +   \langle \nabla_v g \cdot (W_\varepsilon + \psi_\varepsilon)  \rangle_\Omega + \langle G_\varepsilon \rangle_\Omega \quad \mbox{in} \quad  \Omega \times \mathbb{R}_{>0}, \\
  W_\varepsilon (\cdot, 0)   & = 0  \quad \text{in} \quad \Omega, \qquad b_\varepsilon (0) = 0, \qquad  \frac{\partial W_\varepsilon }{\partial \nu} = 0 \quad  \mbox{on} \quad  \partial \Omega \times \mathbb{R}_{>0}. 
  \end{split}
\end{align}
We start estimating the term $W_\varepsilon$ using semigroup estimates from \cite[Lemma 1.3]{Winkler}. Denoting the right-hand side of the elliptic equation by $R_\varepsilon$, the solution $W_\varepsilon$ may be written as
\[
W_\varepsilon(\cdot, t) = \int_0^t S_{\Delta}((t-\tau)/\varepsilon) R_\varepsilon(\cdot, \tau) \; \mathrm{d}\tau.
\]
By Assumption \ref{AssumptionB}, the Jacobian $\mathbf{J} \in L^\infty(\Omega \times \mathbb{R}_{\ge 0})^{(m+1)\times (m+1)}$ is uniformly bounded. We infer from decay estimate \eqref{indecay2} for $\psi_\varepsilon$ and $\langle R_\varepsilon \rangle_\Omega = 0$ that
\begin{align}
\begin{split}
\| W_\varepsilon(\cdot,t) \|_{L^p(\Omega)}  & \le C \int_0^t \mathrm{e}^{-\lambda_1 (t-\tau)/\varepsilon} \Big(\|\bm{\alpha}_{\varepsilon}(\cdot,\tau)\|_{L^p(\Omega)^m} +  |\Omega|^{1/p} |b_\varepsilon(\tau)|  \\ 
& \quad  + \,  \|(W_\varepsilon + \psi_\varepsilon)(\cdot,\tau)\|_{L^p(\Omega)}+ \|G_\varepsilon(\cdot, \tau) \|_{L^p(\Omega)} \Big) \; \mathrm{d}\tau. \label{WDest} 
\end{split}
\end{align}
This estimate implies that $W_\varepsilon$ can always be estimated in terms of $\bm{\alpha}_{\varepsilon}, b_\varepsilon$ and powers of $\varepsilon$. Hence, it remains to control the terms $\bm{\alpha}_{\varepsilon}$ and $b_\varepsilon$. The components $\bm{\alpha}_{\varepsilon}, b_\varepsilon$ satisfy a shadow problem whose solution is given by 
\begin{align}
\begin{pmatrix}
\bm{\alpha}_{\varepsilon}(\cdot,t)\\
b_\varepsilon(t)
\end{pmatrix} = \int_0^t \mathbf{W}(t,\tau) \begin{pmatrix}
\mathbf{B}(\cdot, \tau) (W_\varepsilon + \psi_\varepsilon)(\cdot,\tau) + \mathbf{F}_\varepsilon(\cdot,\tau) \\
\langle D(\cdot,\tau)  (W_\varepsilon + \psi_\varepsilon)(\cdot,\tau) + G_\varepsilon(\cdot,\tau) \rangle_\Omega
\end{pmatrix} \; \mathrm{d}\tau \label{intalphaDbD}
\end{align}
where we used the evolution system $\mathcal{W}$ induced by the linearization \eqref{shadlinear} of the shadow problem. This system is given by evolution operators $\mathbf{W}(t,s)$ for $t,s \in \mathbb{R}_{\ge 0}, s \le t,$ defined by 
\begin{align}
\bm{\xi}(\cdot, t) = \mathbf{W}(t,s) \bm{\xi}(\cdot, s), \qquad \bm{\xi}(\cdot, 0) = \begin{pmatrix}
\bm{\xi}^0_1\\
 \xi_2^0
\end{pmatrix}, \label{evolsystemnonlinear}
\end{align}
where $\bm{\xi} \in C(\mathbb{R}_{\ge 0}; L^p(\Omega)^m \times \mathbb{R})$ is the unique solution of the homogeneous linear shadow problem \eqref{shadlinear}, i.e.,
\begin{align*}
\frac{\partial \bm{\xi}_1}{\partial t} - \mathbf{D} \Delta  \bm{\xi}_1   &= \mathbf{A}(x,t) \bm{\xi}_1 + \mathbf{B}(x,t) \xi_2 \qquad \quad  \, \text{in} \quad \Omega \times \mathbb{R}_{>0}, \qquad \bm{\xi}_1 (\cdot, 0)  =\bm{\xi}_1^0  \quad \text{in} \quad \Omega, \\
   \frac{\mathrm{d} \xi_2}{\mathrm{d} t} &= \langle \mathbf{C}(\cdot, t) \bm{\xi}_1\rangle_\Omega + \langle D(\cdot,t) \xi_2 \rangle_\Omega  \quad \mbox{in} \quad \mathbb{R}_{>0},  \qquad \qquad \; \, \,  \xi_2 (0) =  \xi_2^0.
\end{align*}
In order to control the growth of the errors $\bm{\alpha}_{\varepsilon}, b_\varepsilon$, we make use of the stability condition \ref{AssumptionL1p}.

\begin{proposition} \label{nonerrorest} 
 Let Assumptions \ref{AssumptionN}--\ref{AssumptionL1p} hold for some $1 \le p \le \infty$ and $\sigma \ge 0$.\\
 Then, if $\sigma = 0$, for any $\alpha \in (0, 1], \delta_0 \in (0, 1/2]$ with $2 \delta_0 + (\alpha-1) \in (0,1]$ there exist constants $C, \varepsilon_0>0$ independent of time $T$ and diffusion $1/\varepsilon$ such that for all times $T \le \varepsilon^{\alpha-1}$ and all $\varepsilon \le \varepsilon_0$ the solution $(\bm{\alpha}_{\varepsilon}, \beta_\varepsilon)$ of system \eqref{truncsys} satisfies
 \begin{align}
 \sup_{t \in [0,T]} \left( \| \bm{\alpha}_{\varepsilon}(\cdot,t) \|_{L^p(\Omega)^m} + \| \beta_\varepsilon(\cdot,t) \|_{L^p(\Omega)}\right) & \leq C \varepsilon^{2 \delta_0 + (\alpha-1)}. \label{nonerrorestLp}
 \end{align}
Moreover, if $\sigma>0$, then for any $\delta_0 \in (0, 1/2]$, there exist constants $C, \varepsilon_0>0$ independent of diffusion such that for all $\varepsilon \le \varepsilon_0$ the solution $(\bm{\alpha}_{\varepsilon}, \beta_\varepsilon)$ of the truncated problem \eqref{truncsys} satisfies 
 \begin{align}
\|\bm{\alpha}_{\varepsilon}\|_{L^p(\Omega \times \mathbb{R}_{\ge 0})^m} +  \|\beta_\varepsilon\|_{L^p(\Omega \times \mathbb{R}_{\ge 0})}   \le  C\varepsilon^{2\delta_0}. \label{nonerrorestLpexp}
 \end{align}
\end{proposition}

\begin{proof} In the following proof, the constant $C>0$ may vary from line to line, however, it is independent of $\varepsilon$ and time. We already estimated $W_\varepsilon$ and obtained a relation to $\bm{\alpha}_{\varepsilon}, b_\varepsilon$ in inequality \eqref{WDest}. In view of estimate \eqref{Tronc3} for $G_\varepsilon$, where $\overline{\beta_\varepsilon} = b_\varepsilon + W_\varepsilon + \psi_\varepsilon$,
and estimate \eqref{indecay2} for $\psi_\varepsilon$, we observe %for each $\delta_0 \ge 0$
\begin{align}
\begin{split}
\| W_\varepsilon(\cdot,t) \|_{L^p(\Omega)}  & \le C \int_0^t \mathrm{e}^{-\lambda_1(t-\tau)/\varepsilon} \left(\|\bm{\alpha}_{\varepsilon}(\cdot,\tau)\|_{L^p(\Omega)^m} +  |\Omega|^{1/p} |b_\varepsilon(\tau)|  \right) \; \mathrm{d}\tau \\ 
& \quad + C \int_0^t \mathrm{e}^{-\lambda_1 (t-\tau)/\varepsilon} \|W_\varepsilon(\cdot,\tau)\|_{L^p(\Omega)} \; \mathrm{d}\tau + C|\Omega|^{1/p} \varepsilon. \label{WDest2} 
\end{split}
\end{align}
To obtain a corresponding inequality for $\bm{\alpha}_{\varepsilon}, b_\varepsilon$, we consider the explicit formula \eqref{intalphaDbD}. Applying Assumption \ref{AssumptionL1p} on $\mathcal{W}$, estimate \eqref{indecay2} for $\psi_\varepsilon$ and estimate \eqref{Tronc3} for the truncation $\mathbf{H}_\varepsilon=(\mathbf{F}_\varepsilon, G_\varepsilon)$ to this integral representation yields
\begin{align*}
\|\bm{\alpha}_{\varepsilon}(\cdot,t)\|_{L^p(\Omega)^m} + |\Omega|^{1/p}  |b_\varepsilon(t)| & \le C \int_0^t \mathrm{e}^{-\sigma(t-\tau)} \left( \|(W_{\varepsilon} + \psi_{\varepsilon})(\cdot,\tau)\|_{L^p(\Omega)} + \|\mathbf{H}_{\varepsilon}(\cdot, \tau)\|_{L^p(\Omega)^{m+1}} \right) \, \mathrm{d}\tau\\
& \le C \int_0^t  \mathrm{e}^{-\sigma(t-\tau)}  \|W_\varepsilon(\cdot,\tau)\|_{L^p(\Omega)} \; \mathrm{d}\tau + C |\Omega|^{1/p} f_\varepsilon(t) \\
& \quad + C \int_0^t  \mathrm{e}^{-\sigma(t-\tau)}   { \mathds{1}}_{\{  \tau \leq -\varepsilon \log \varepsilon/ \lambda_1 \}}(\tau) |\Omega|^{1/p}  |b_\varepsilon(\tau)| \; \mathrm{d}\tau.
\end{align*}
Here we restrict ourselves to $\delta_0 \le 1/2$ and denote
\[
f_\varepsilon(t) = \int_0^t \mathrm{e}^{-\sigma(t-\tau)} \left(  \varepsilon^{2\delta_0} + \mathrm{e}^{-\lambda_1 \tau/\varepsilon} \right) \; \mathrm{d}\tau \le 
\begin{cases}
C \left( \varepsilon + \varepsilon^{2\delta_0} t\right) &  \text{if} \quad \sigma=0,\\
C  \varepsilon^{2\delta_0} & \text{if} \quad \sigma>0.
\end{cases}
\]
Since $-\varepsilon \log \varepsilon\to 0$ as $\varepsilon \to 0$ by L'Hospital's rule, we absorb $b_\varepsilon$ on the left-hand side and obtain
\begin{align} \label{estalphaDbD}
\begin{split}
\sup_{t \in [0,T]} \left( \|\bm{\alpha}_{\varepsilon}(\cdot,t)\|_{L^p(\Omega)^m} + |\Omega|^{1/p}  |b_\varepsilon(t)| \right) & \le C \sup_{t \in [0,T]} \left( \int_0^t  \mathrm{e}^{-\sigma(t-\tau)}  \|W_\varepsilon(\cdot,\tau)\|_{L^p(\Omega)} \; \mathrm{d}\tau + |\Omega|^{1/p} f_\varepsilon(t) \right)\\
& \le C g_{1,\varepsilon}(T) \sup_{t \in [0,T]}  \|W_\varepsilon(\cdot,t)\|_{L^p(\Omega)} + C |\Omega|^{1/p} g_{2, \varepsilon}(T) 
\end{split}
\end{align}
for 
\[
g_{1,\varepsilon}(T) = \begin{cases}
 T &  \text{if} \quad \sigma=0,\\
1 & \text{if} \quad \sigma>0
\end{cases} \qquad \text{and} \qquad g_{2,\varepsilon}(T) = \begin{cases}
\varepsilon + \varepsilon^{2 \delta_0} T &  \text{if} \quad \sigma=0,\\
 \varepsilon^{2 \delta_0} & \text{if} \quad \sigma>0.
\end{cases}
\]
Estimate \eqref{estalphaDbD} combined with inequality \eqref{WDest2} for $W_\varepsilon$ leads to an estimate for $W_\varepsilon$, more precisely, 
\begin{align*}
 \sup_{t \in [0,T]} \| W_\varepsilon(\cdot,t) \|_{L^p(\Omega)}  & \le C \varepsilon(1  + g_{1,\varepsilon}(T)) \sup_{t \in [0,T]}  \|W_\varepsilon(\cdot,t)\|_{L^p(\Omega)} + C|\Omega|^{1/p} \varepsilon  \left(1 + g_{2, \varepsilon}(T)  \right).
\end{align*}
%Taking the supremum over $t$ yields 
%\[
% \sup_{t \in [0,T]} \| W_\varepsilon(\cdot,t) \|_{L^p(\Omega)}  \leq C \varepsilon (1+T) \sup_{t \in [0,T]} \| W_\varepsilon(\cdot,t) \|_{L^p(\Omega)} + C\varepsilon (1+\varepsilon^{2\delta_0}T).
%\]
First we consider $\sigma>0$. Then absorbing the terms on the left-hand side yields $\sup_{t \in [0,T]} \| W_\varepsilon(\cdot,t) \|_{L^p(\Omega)}  \leq C \varepsilon$ and, by estimate \eqref{estalphaDbD},
\[
 \sup_{t \in [0,T]} \left( \|\bm{\alpha}_{\varepsilon}(\cdot,t)\|_{L^p(\Omega)^m} + |\Omega|^{1/p}  |b_\varepsilon(t)| \right)  \le C \varepsilon^{2 \delta_0} \qquad \forall \; T \ge 0.
 \]
%Reusing estimate \eqref{WDest2} actually shows that $W_\varepsilon$ is of order $\varepsilon$. \\
In the case $\sigma=0$, we consider the time restriction $T \le \varepsilon^{\alpha-1}$ to absorb the terms on the left-hand side. This implies for $\varepsilon\le \varepsilon_0$ that
\[
 \sup_{t \in [0,T]} \| W_\varepsilon(\cdot,t) \|_{L^p(\Omega)}  \leq C \varepsilon^{\min\{1, 2 \delta_0 + \alpha\}}.
\]
Using estimate \eqref{estalphaDbD} for $(\bm{\alpha}_{\varepsilon}, b_\varepsilon)$, for each $\gamma = 2 \delta_0 + (\alpha-1) \in (0,1]$ there holds 
\begin{align*}
\sup_{t \in [0,T]} \left( \|\bm{\alpha}_{\varepsilon}(\cdot,t)\|_{L^p(\Omega)^m} + |\Omega|^{1/p}  |b_\varepsilon(t)| \right) & \le  C  \left(  \varepsilon^{2\delta_0} T + \varepsilon T \right) \le C \varepsilon^{\gamma}
\end{align*}
for $\varepsilon \le \varepsilon_0=\varepsilon_0(\gamma)$. As a consequence, $\gamma>0$ implies $2 \delta_0 + \alpha> 1$ and $\|W_\varepsilon\|_{L^p(\Omega)} \le C\varepsilon$. The relation $\beta_\varepsilon=W_\varepsilon + b_\varepsilon$ finally implies an estimate for $\beta_\varepsilon$. 
\end{proof}

If Assumption \ref{AssumptionL1p} holds with $p= \infty$, we already obtain an estimate in $L^\infty(\Omega_T)$ for solutions of the truncated problem \eqref{truncsys}. In this case, we may directly proceed with the proof of Theorem \ref{TheoremToy} in Section \ref{sec:errorestimates} by removing the truncation as in Lemma \ref{Removetrunc}.\\

 If Assumption \ref{AssumptionL1p} holds with $p< \infty$, Proposition \ref{nonerrorest} states an estimate of the norms of $\bm{\alpha}_{\varepsilon}, \beta_\varepsilon$ in the parabolic space $L_{p, \infty}(\Omega_T)$. Consequently, H\"older's inequality yields bounds for the norms of $\bm{\alpha}_{\varepsilon}, \beta_\varepsilon$ in the parabolic space $L_{p, r}(\Omega_T)$ for each finite $1 \le r  < \infty$, too. We refer to definition \eqref{parabLpr} for a precise meaning of the spaces and norms. Using a parabolic bootstrap argument from Proposition \ref{thees}, either for $p \ge 1=n$ or for $p> n/2 \ge 1$ and an appropriate choice of $r$, the latter implies an $L^\infty(\Omega_T)$ estimate for the diffusing component $\beta_\varepsilon$ and any diffusing component of $\bm{\alpha}_{\varepsilon}$. For non-diffusing components, however, we have to apply the additional stability condition \ref{AssumptionL0} on the corresponding ODE subsystem according to \cite{KMCMlinear}. Under this additional assumption, we obtain the following result.
 
 \begin{proposition} \label{LinftynD} Let Assumptions \ref{AssumptionN}--\ref{AssumptionL0} hold for some finite $p \ge 1=n$ or $p>n/2$ for $n \ge 2$, $\sigma, \mu \ge 0,$ and choose $r$ as in relation \eqref{paramLady}. Then there exist triples $(\alpha, \delta_0, r) \in (0,1] \times (0,1/2] \times (1, \infty)$ and constants $C, \varepsilon_0>0$ independent of $T, \varepsilon$ such that for all $T \le \varepsilon^{\alpha-1}$ and $\varepsilon \le \varepsilon_0$ there holds 
 \begin{align}\label{Linftyerrest}
  \| \bm{\alpha}_{\varepsilon} \|_{L^\infty (\Omega_T)^m}  \leq C \varepsilon^{-2(1-\alpha)} \left( \varepsilon^{\gamma} + \varepsilon^{(2-\alpha)/r} \right), \qquad
  \| \beta_\varepsilon \|_{L^\infty (\Omega_T)}  \leq C \varepsilon^{-(1-\alpha)} \left( \varepsilon^{\gamma} + \varepsilon^{(2-\alpha)/r} \right).
 \end{align}
 According to Proposition \ref{nonerrorest}, we use $\gamma = 2\delta_0 + (\alpha-1) \in (0,1]$ if $\sigma=0$ and $\gamma = 2 \delta_0$ if $\sigma>0$. Moreover, for diffusing components of $\bm{\alpha}_{\varepsilon}$ we have the same estimate as for $\beta_\varepsilon$. The case $\mu>0$ implies that each component of $\bm{\alpha}_\varepsilon$ can be estimated as $\beta_\varepsilon$.
 \end{proposition}
 
 \begin{proof} 
First we consider the diffusing component $\beta_{\varepsilon}$ of the truncated problem \eqref{truncsys} solving
\begin{equation*}
\frac{\partial \beta_{\varepsilon}}{\partial t} - \frac{1}{\varepsilon} \Delta \beta_{\varepsilon} = R_\varepsilon(x,t) 
\end{equation*}
for some remainder $R_\varepsilon \in L^\infty(\Omega_T)$. Proposition \ref{thees} yields existence of a constant $C>0$ such that
 \begin{equation*}
   \|\beta_{\varepsilon} \|_{L^\infty (\Omega_T)} \leq C T^{1-1/r} \| R_\varepsilon  \|_{p, r }
 \end{equation*}
for some $r >1$ defined by relation \eqref{paramLady}. This constant $C$ is independent of time $T$ and only depends on a lower bound for the diffusion and parameters of the systems. Hence, all diffusing components of $(\bm{\alpha}_\varepsilon, \beta_\varepsilon)$ can be treated with this bootstrap argument which makes use of Proposition \ref{thees} and it remains to find an estimate for the remainder $\| R_\varepsilon  \|_{p, r }$ which has the same form in any case. We infer from the right-hand side of the truncated system \eqref{truncsys} that
\begin{align*}
\|R_\varepsilon \|_{p,r} & \le C \left( \| \bm{\alpha}_{\varepsilon} \|_{p,r} +  \| \beta_\varepsilon \|_{p,r} + \|\psi_\varepsilon\|_{p,r} + \|\mathbf{H}_\varepsilon\|_{p,r}\right).
\end{align*}
Estimate \eqref{indecay2} for $\psi_\varepsilon$ and Lemma \ref{quadcut} with $\overline{\beta}_\varepsilon = \beta_\varepsilon + \psi_\varepsilon$ yields
\begin{align*}
\|R_\varepsilon \|_{p,r} & \le  C \left( \| \bm{\alpha}_{\varepsilon} \|_{p,r} +  \| \beta_\varepsilon \|_{p,r} + \varepsilon^{1/r} + \varepsilon^{2\delta_0}T^{1/r} \right). %\label{esF1}
\end{align*}
By H\"older's inequality, we can relate the $L_{p,r}(\Omega_T)$ norms to the considered norms in Proposition \ref{nonerrorest} by
\begin{align*}
\|R_\varepsilon \|_{p,r} & \le  C \left( T^{1/r} \left( \| \bm{\alpha}_{\varepsilon} \|_{p,\infty} +  \| \beta_\varepsilon \|_{p,\infty} + \varepsilon^{2\delta_0} \right) + \varepsilon^{1/r}  \right). 
\end{align*}
%, we gain a priori estimates for $\bm{\alpha}_{\varepsilon}, \beta_\varepsilon$ in $L_{p,r}(\Omega_T)$.
%for instance, $\|\bm{\alpha}_{\varepsilon}\|_{p,r} \le T^{1/r} \|\bm{\alpha}_{\varepsilon}\|_{p,\infty}$.
%\[
%\|\bm{\alpha}_{\varepsilon}\|_{p,r} \le T^{1/r} \|\bm{\alpha}_{\varepsilon}\|_{p,\infty} \le C\varepsilon^{\gamma - \frac{1}{r}(1-\alpha)}. 
%\]
Applying Proposition \ref{nonerrorest} and $T \le \varepsilon^{\alpha-1}$ for some $\alpha \in (0,1]$ implies
\begin{align*}
 \|\beta_{\varepsilon} \|_{L^\infty (\Omega_T)} & \leq C \varepsilon^{-(1-\frac{1}{r})(1-\alpha)} \left( \varepsilon^{\gamma - \frac{1}{r}(1-\alpha)}  + \varepsilon^{2\delta_0 - \frac{1}{r}(1-\alpha)} + \varepsilon^{1/r} \right)
\end{align*}
where $\gamma = \gamma = 2\delta_0 + (\alpha-1) \le 2 \delta_0$ if $\sigma=0$ and $\gamma = 2 \delta_0$ if $\sigma>0$. These cases lead to the estimate
\begin{align*}
 \|\beta_{\varepsilon} \|_{L^\infty (\Omega_T)} & \leq C \varepsilon^{-(1-\alpha)} \left( \varepsilon^{\gamma} + \varepsilon^{(2-\alpha)/r} \right).
\end{align*}
As all diffusive components can be estimated in this manner, it remains to restrict our consideration to non-diffusing components of $\bm{\alpha}_\varepsilon$, denoted by $\bm{\alpha}_{\varepsilon, 0}$.  To be precise, we delete all rows and columns of $\nabla_\mathbf{u} \mathbf{f}(\cdot,t)$ in the Jacobian \eqref{Jacobian} for which the diffusion is positive to obtain the matrix-valued function $\mathbf{A}_0(\cdot,t) \in L^\infty(\Omega)^{m_0 \times m_0}$ for some $0 \le m_0 \le m$. To control the growth of the ODE components, we consider solutions to the initial value problem
\begin{equation*}
\frac{\partial \bm{\psi}}{\partial t} = \mathbf{A}_0(\cdot, t) \bm{\psi} \qquad \text{in} \quad \Omega \times \mathbb{R}_{>0}, \qquad \bm{\psi}( \cdot, 0) = \bm{\psi}^0 \quad \text{in} \quad \Omega.
\end{equation*}
These solutions define an evolution system $\mathbf{\mathcal{U}}$ consisting of evolution operators  $\mathbf{U}(t,s)$ on $L^\infty(\Omega)^{m_0}$ for $t,s \in \mathbb{R}_{\ge 0}$, $s \le t,$ given by $\bm{\psi}(\cdot, t) = \mathbf{U}(t,s) \bm{\psi}(\cdot, s)$. Using the notion of Assumption \ref{AssumptionL0} and the first equation of system \eqref{truncsys}, the non-diffusing error component $\bm{\alpha}_{\varepsilon,0}$ is given by 
\[
\bm{\alpha}_{\varepsilon,0}(\cdot,t) = \int_0^t \mathbf{U}(t, \tau) \mathbf{S}_\varepsilon(\cdot,\tau) \; \mathrm{d}\tau,
\]
for a right-hand side $\mathbf{S}_\varepsilon \in L^\infty(\Omega_T)^{m_0}$ which can be estimated in terms of diffusing components of $\bm{\alpha}_\varepsilon$, $\beta_\varepsilon + \psi_\varepsilon$ and components of $\mathbf{F}_\varepsilon$. Recall that we already estimated all diffusing components and further recall from Lemma \ref{quadcut} that $\mathbf{F}_\varepsilon$ can be estimated by $\overline{\beta}_\varepsilon = \beta_\varepsilon + \psi_\varepsilon$. Hence, uniform stability of the evolutionary system $\mathcal{U}$ leads to
\begin{align*}
  \|\bm{\alpha}_{\varepsilon, 0} \|_{L^\infty(\Omega_T)^m} & \leq C \sup_{t \in [0,T]} \int_0^t \mathrm{e}^{-\mu(t-\tau)} \| \mathbf{S}_\varepsilon(\cdot,\tau)\|_{L^\infty(\Omega)^{m_0}} \; \mathrm{d}\tau \\
  & \le C \varepsilon^{-(1-\alpha)} \left( \varepsilon^{\gamma} + \varepsilon^{(2-\alpha)/r} \right)  \sup_{t \in [0,T]} \int_0^t \mathrm{e}^{-\mu(t-\tau)} \; \mathrm{d}\tau.
\end{align*}
From this we infer that the convergence rate of the non-diffusing components is at least as good as the estimate for the diffusing components in the case of $\mu>0$, i.e., in the case of exponential stability of $\mathcal{U}$. In the case of $\mu=0$ we obtain an additional factor $T \le \varepsilon^{\alpha-1}$ and, thus, $\varepsilon^{-2(1-\alpha)} \left( \varepsilon^{\gamma} + \varepsilon^{(2-\alpha)/r} \right)$ as a convergence rate for the non-diffusing component $\bm{\alpha}_{\varepsilon,0}$.
 \end{proof}

\subsection{Error estimates} \label{sec:errorestimates}

We are now in a position to draw a conclusion for the original errors $(\mathbf{U}_\varepsilon, V_\varepsilon)$ using estimates for the truncated problem in the last section. In order to dispose of truncation, we infer from results of Proposition \ref{nonerrorest} and \ref{LinftynD} that the truncated solution $(\bm{\alpha}_{\varepsilon}, \beta_\varepsilon)$, for sufficiently large diffusion, is located in a neighborhood of $\mathbf{0}$ where the cut-off is not required by Lemma \ref{Removetrunc}.

\begin{proof}[Proof of Theorem \ref{TheoremToy}] 
If $\sigma>0$, %is true without requiring Assumption \ref{AssumptionL0}, and 
Proposition \ref{nonerrorest} yields a convergence rate of the order $\varepsilon^{2 \delta_0}$ for each component. Thus, Lemma \ref{Removetrunc} applies as we always can ensure $C \varepsilon^{2 \delta_0} \le \varepsilon^{\delta_0}$ for all sufficiently small $\varepsilon>0$. \\
 %Note that uniform boundedness of the evolution system $\mathcal{W}$ of the entire shadow system in $L^\infty(\Omega)^m \times \mathbb{R}$ is sufficient, thus no boundedness of the ODE subsystem stated in Assumption \ref{AssumptionL0} is needed.\\
If $\sigma=0$, %estimate \eqref{nonerrorestLp} holds true 
Proposition \ref{nonerrorest} yields a convergence rate of the order $\varepsilon^{\gamma}$ for each component and it remains to ensure $\gamma> \delta_0$ such that inequality \eqref{removetrunc} applies to remove the truncation. The latter condition is equivalent to $\alpha>1-\delta_0$ as $\gamma= 2 \delta_0 + (\alpha -1)$. Finally, choosing $\alpha_0 = 1-\delta_0$ yields estimate \eqref{longestToy}.
%\begin{align*}
%  \| \mathbf{U}_\varepsilon \|_{L^\infty (\Omega_T)^m}, \| V_\varepsilon \|_{L^\infty (\Omega_T)}  \leq C \varepsilon^{(1-\alpha)}.
% \end{align*}
%To get rid of truncation, it is sufficient to verify \eqref{removetrunc} as in the proof of Theorem \ref{Theorem}. If $p=\infty$, we can choose $2 \delta_0=1$ in estimate \eqref{nonerrorestLpexp} and remove the cut-off due to $2\delta_0 >\delta_0$ to obtain the global estimate.
 \end{proof}

\begin{proof}[Proof of Theorem \ref{Theorem}]  
The proof for the case of $p< \infty$ is analog to the proof of Theorem \ref{TheoremToy}, however, we use the results of Proposition \ref{LinftynD}. In view of Lemma \ref{Removetrunc} and estimate \eqref{Linftyerrest}, it remains to find triples $(\alpha, \delta_0, r)$ such that
\begin{align*}
\delta_0 < 2(\alpha-1) + \gamma \qquad \text{and} \qquad \delta_0 < 2(\alpha-1) + (2-\alpha)\frac{1}{r} % \label{parrestr}
\end{align*}
is satisfied, where $\gamma = 2\delta_0 + (\alpha-1) \in (0,1]$ if $\sigma=0$ and $\gamma = 2 \delta_0$ if $\sigma>0$. For existence of such triples in the case $\sigma=0$, we define the following two restrictive curves (one depends on the parameter $r$ as well)
\begin{align}
 \alpha > \ell(\delta_0) := 1 - \frac{1}{3}\delta_0 \qquad \text{and} \qquad \alpha > \ell_r(\delta_0) := \frac{r}{2r-1} \delta_0 + 1-\frac{1}{2r-1}. \label{opaque}
\end{align}
While the function $\delta_0 \mapsto \ell(\delta_0)$ is strictly monotone decreasing with $\ell(1/2) = 5/6<1$, the parameter-dependent function $\delta_0 \mapsto \ell_r(\delta_0)$ is strictly increasing with $\ell_r(0) \in (0,1)$ since $1<r< \infty$. Thus, we always find such triples for small enough $\delta_0>0$. \\%We have at least $1 > \alpha >5/6$ for both functions and we have a triangular restrictive area for $r \le 7/5$ given by $1>\alpha>\ell(\delta_0)$. For $7/5 < r < 2$ there is a quadrilateral area where both lines restrict the possible values of $\alpha <1$. A triangular area induced by $1>\alpha> \ell_r(\delta_0)$ restricts for $r \ge 2$.\\
Since the conditions \eqref{opaque} are quite opaque, our goal is to further simplify the estimate \eqref{Linftyerrest} under consideration of the particular case $\delta_0 \le 1/(2r)$. Recall that the intersection of the graphs of the functions $\ell_r$ and $\ell$ is given by one point for $\overline{\delta}_0 = \frac{3}{5r-1}$. In this case we have $\delta_0 <\overline{\delta}_0$ and the only restriction is given by $\ell(\delta_0)<\alpha<1$. Note that $\alpha > \ell(\delta_0)$ is equivalent to $\delta_0 > 3(1-\alpha)$ and the assertion follows from estimates obtained in Proposition \ref{LinftynD} and the identity $(\bm{\alpha}_\varepsilon, \beta_\varepsilon) = (\mathbf{U}_\varepsilon, V_\varepsilon)$ from Lemma \ref{Removetrunc}.\\
The case $\sigma>0$ is analog. In this case the function $\ell$ is replaced by $\ell(\delta_0):= 1- \frac{1}{2} \delta_0$ in \eqref{opaque}. The intersection of the graphs of $\ell_r$ and $\ell$ is determined by $\overline{\delta}_0 = \frac{2}{4r-1}$. A restriction to $\delta_0 \le 1/(2r) < \overline{\delta}_0$ yields $\| \mathbf{U}_\varepsilon\|_{L^\infty (\Omega_T)^m}  \leq C \varepsilon^{2(1-\alpha)}$ and 
$ \|  V_\varepsilon \|_{L^\infty (\Omega_T)}  \leq C \varepsilon^{3(1-\alpha) }$, using the estimates derived in Proposition \ref{nonerrorest}.
 \end{proof}
 
 \begin{remark} Since the evolution systems $\mathcal{U}$ and $\mathcal{W}$ are uniformly bounded, the growth in time in estimates of Theorem \ref{Theorem} is polynomial. Similar estimates can be derived if the evolution systems are only uniformly bounded by some polynomial, as defined in \cite[Definition 1.15]{Eisner}, \cite[Definition 2.7]{Hai}, we refer to \cite[p. 63]{Kowall}.
\end{remark}

\begin{remark}
Similar to calculations in \cite[Corollary 2, Proposition 5]{AMCAM:2015A}, we can obtain estimates for first-order derivatives of the errors. As this is a natural consequence of the weak formulation, we omit details.
\end{remark}

\section{Stability conditions}  \label{sec:stabverification}

To prove convergence results, we applied a stabilizing effect to ensure that diffusing solutions of system \eqref{fullsys} stay nearby its shadow limit which solves system \eqref{shadow}. Such a result is achieved via linearization at the shadow solution. 
A verification of the stability conditions \ref{AssumptionL1p}--\ref{AssumptionL0} using numerical simulations is often reasonable, however, simulations may be misleading. The aim of this section is to provide analytic tools for verifying the stability assumptions of the linear shadow problem \eqref{shadlinear} in two special situations: We consider a linearization which has time-independent coefficients, resulting from a stationary shadow limit, and a linearization which induces a dissipative (and thus stable) evolution system. \\

In case of a bounded stationary shadow solution, we give a complete characterization of the spectrum of the corresponding linearized shadow operator in the next section. As shown in Corollary \ref{stability} and Remark \ref{rem:stabilityshadow}, a negative spectral bound of the shadow operator allows verifying stability Assumption \ref{AssumptionL1p}. In addition to that, a linearized stability analysis allows deducing nonlinear stability results for bounded stationary solutions of the shadow system \eqref{shadow}. In fact, using \cite[Proposition 4.17]{Webb} for nonlinear stability and \cite[Theorem 1]{Shatah} for nonlinear instability, one may use the following spectral characterization to infer stability properties from the linearization at the stationary solution. %Consult also the stability results \cite[Ch. VII, Theorems 2.1, 2.3]{Daleckii} in the case of bounded shadow operators. 

\subsection{Spectral analysis} \label{sec:spectral}%spectral theory

We consider a bounded stationary solution $(\overline{\mathbf{u}}, \overline{v}) \in L^\infty(\Omega)^{m} \times \mathbb{R}$ of shadow system \eqref{shadow}, i.e., a bounded solution of the problem
\begin{align*}
-  \mathbf{D} \Delta \mathbf{u}  = \mathbf{f}( \mathbf{u} ,v,x) \quad \mbox{in} \quad \Omega, \qquad
 0  =  \langle g ( \mathbf{u} , v,\cdot) \rangle_\Omega.
\end{align*}
Then the linearization of the shadow problem at this stationary solution leads to problem \eqref{shadlinear} with coefficients that are time-independent. This system can be rewritten as
 \begin{align*}
\frac{\partial \bm{\xi}}{\partial t}  - \mathbf{D}_0 \Delta \bm{\xi} &= \mathbf{L} \bm{\xi}  \quad  \text{in} \quad \Omega \times \mathbb{R}_{>0}, \qquad \bm{\xi} (\cdot, 0)  =\bm{\xi}^0  \quad \text{in} \quad \Omega,
\end{align*}
for $\mathbf{D}_0 = \mathrm{diag}(\mathbf{D}, 0) \in \mathbb{R}_{\ge 0}^{(m+1) \times (m+1)}$ and the operator $\mathbf{L} \in \mathcal{L}(L^p(\Omega)^m \times \mathbb{R})$ given by
\begin{align}
(\mathbf{L}
\bm{\xi})(x) & = \begin{pmatrix} \mathbf{A}(x) \bm{\xi}_1(x) + \mathbf{B}(x) \xi_2\\
 \langle \mathbf{C}(\cdot) \bm{\xi}_1 \rangle_\Omega + \langle D(\cdot) \xi_2 \rangle_\Omega 
   \end{pmatrix} \qquad \forall \; \bm{\xi}=(\bm{\xi}_1, \xi_2) \in L^p(\Omega)^m \times \mathbb{R}, x \in \Omega. \label{Hardtint}
\end{align}
Here, we used the same notation as in the Jacobian \eqref{Jacobian} for the uniformly bounded entries $\mathbf{A}, \mathbf{B}, \mathbf{C}, D$ being the parts of the Jacobian of $(\mathbf{f}, g)$ evaluated at the shadow steady state $(\overline{\mathbf{u}}, \overline{v})$, based on Assumptions \ref{AssumptionN}--\ref{AssumptionB}.\\

It is well-known that the spectrum of the shadow operator $\mathbf{D}_0 \Delta + \mathbf{L}$ is discrete in the case when all diffusion coefficients are positive \cite{Miyamoto}. For analysis of the spectrum, we focus on the challenging case of reaction-diffusion-ODE systems, in which case the spectrum does not need to be purely discrete. As the proof is similar in the case where the shadow operator consists of diffusing and non-diffusing components, we omit details and refer to \cite[Proposition 5.7]{Kowall}. In order to characterize the spectrum of the shadow operator $\mathbf{L}$, we apply a spectral decomposition for block operator matrices according to \cite{Atkinson}. This approach is based on properties of a bounded multiplication operator which is induced by the ODE subsystem on $L^p(\Omega)^m$. Necessary properties of  multiplication operators are given in Section \ref{sec:multoperator}.

\subsubsection{Stationary shadow operator}

We consider a linearization in $L^p(\Omega)^{m} \times \mathbb{R}$ in the case of a reaction-diffusion-ODE system with $\mathbf{D} \equiv \mathbf{0}$. We refer to Remark \ref{rem:stabilityshadow} for the case $\mathbf{D}_0 \not\equiv \mathbf{0}$. According to Assumption \ref{AssumptionN} and the Fr{\'e}chet-derivative, one verifies that the linearization is determined by the bounded linear operator $\mathbf{L}$ defined above. Note that $\mathbf{L}$ is the generator of a uniformly continuous semigroup $(\mathrm{e}^{\mathbf{L}t})_{t \in \mathbb{R}_{\ge 0}}$ by \cite[Ch. II, Corollary 1.5]{Engel}. Concerning the stability condition we are faced with in Assumption \ref{AssumptionL1p}, it is clear that in the stationary case \eqref{Hardtint} the evolution system $\mathcal{W}$ consists of the evolution operators 
\[
\mathbf{W}(t,s) = \mathrm{e}^{\mathbf{L}(t-s)} \qquad \forall \; 0 \le s \le t \in \mathbb{R}_{\ge 0}.
\] 
Analyticity of this semigroup implies validity of the spectral mapping theorem for $1 \le p \le \infty$ by \cite[Ch. IV, Corollary 3.12]{Engel}. As a consequence, the growth bound of the semigroup equals the spectral bound  $s(\mathbf{L}) := \sup \{ \mathrm{Re} \, \lambda \mid \lambda \in \sigma(\mathbf{L})\}$ of the generator $\mathbf{L}$. Hence, to obtain uniform exponential stability of the semigroup $(\mathrm{e}^{\mathbf{L}t})_{t \in \mathbb{R}_{\ge 0}}$ resp. the evolution system $\mathcal{W}$, it suffices to infer a negative spectral bound $s(\mathbf{L})$ of the generator $\mathbf{L}$ \cite[Ch. V, Theorem 1.10]{Engel}, see Corollary \ref{stability}. For this reason it is of particular importance to characterize the spectrum $\sigma(\mathbf{L})$ of the shadow operator $\mathbf{L}$. \\

To study invertibility of the operator $\lambda I - \mathbf{L}$ for $\lambda \in \mathbb{C}$, we focus on the following system of equations
\[
(\lambda I - \mathbf{L}) \bm{\xi} = \bm{\psi}  \quad \Leftrightarrow \quad \begin{cases}
\qquad \qquad (\lambda I -\mathbf{A}) \bm{\xi}_1  -\mathbf{B} \xi_2 = \bm{\psi}_1,\\
\; \,   - \langle \mathbf{C} \bm{\xi}_1 \rangle_\Omega  + (\lambda I - \langle  D \rangle_\Omega) \xi_2 = \psi_2.
\end{cases}
\]
Let $\mathbf{A}$ also denote the bounded multiplication operator induced by the matrix $\mathbf{A}(x)$ on $L^p(\Omega)^m$, see Section \ref{sec:multoperator}. If $\lambda \notin \sigma(\mathbf{A})$, the first equation can be solved with respect to $\bm{\xi}_1$ and we obtain the following result. 

\begin{lemma} \label{Sigma}
Let $\mathbf{A}, \mathbf{B}, \mathbf{C}, D$ be matrix-valued functions with entries in $L^\infty(\Omega)$ according to the shadow operator $\mathbf{L}$ defined by \eqref{Hardtint} on $L^p(\Omega)^m \times \mathbb{R}$ for some $1 \le  p \le \infty$. Then 
\[
\Sigma:= \sigma(\mathbf{L}) \cap \rho(\mathbf{A}) \subset \sigma_p(\mathbf{L})
\]
is a discrete (probably empty) set of eigenvalues of $\mathbf{L}$. Moreover, $\sigma(\mathbf{L}) \subset \sigma(\mathbf{A}) \mathop{\dot{\cup}} \Sigma$.
\end{lemma}

\begin{proof}
Provided $\lambda \in \rho(\mathbf{A})$, we already have
\[
(\lambda I - \mathbf{L}) \bm{\xi} = \bm{\psi}  \qquad \Leftrightarrow \qquad    \bm{\xi}_1 = (\lambda I - \mathbf{A})^{-1} \left(\bm{\psi}_1 + \mathbf{B} \xi_2\right) , \quad H(\lambda) \xi_2 = \tilde{\psi}(\lambda)
\]
where 
\begin{align*}
H(\lambda) & =\lambda -\langle D \rangle_\Omega - \langle \mathbf{C}  (\lambda I - \mathbf{A})^{-1} \mathbf{B} \rangle_\Omega, \qquad  \tilde{\psi}(\lambda)  = \psi_{2} + \langle \mathbf{C}  (\lambda I - \mathbf{A})^{-1} \bm{\psi}_1\rangle_\Omega.
\end{align*}
We apply an analytic Fredholm theorem \cite[Theorem 4.34]{Hanson} to show discreteness of the remaining spectrum $\Sigma$ within the open set $\rho(\mathbf{A})$. Let us first note that the subsystem for $\xi_2$ can be solved for sufficiently large $\lambda>0$. To see this, recall that $\mathbf{A}$ is a bounded multiplication operator and hence, for $\lambda \in \rho(\mathbf{A})$ with $|\lambda| > 2\|A\|$, we obtain an estimate for $(\lambda I - \mathbf{A})^{-1}$ by Neumann's series which is of the form $2/|\lambda|$. This implies that $H(\lambda)>0$ for all sufficiently large $\lambda>0$. Moreover, analyticity of the resolvent map $\lambda \mapsto (\lambda I - \mathbf{A})^{-1}$ implies analyticity of the complex-valued function $H$ on $\rho(\mathbf{A})$. It is well-known that the set of zeros of the analytic function $H$ is a discrete set $\Sigma \subset \rho(\mathbf{A})$. Consequently, $H(\lambda)$ is invertible for all $\lambda \in \rho(\mathbf{A}) \setminus \Sigma$. For values $\lambda \in \Sigma$ we infer an eigenfunction $(\bm{\xi}_{1},\xi_2) \in L^\infty(\Omega)^m \times \mathbb{R}$ of the eigenvalue equation $(\lambda I - \mathbf{L}) \bm{\xi}=\mathbf{0}$ due to $(\lambda I - \mathbf{A})^{-1} \mathbf{B} \in L^\infty(\Omega)^{m}$ by  \cite[Proposition 2.2]{Hardt}. This shows $\Sigma \subset \sigma_p(\mathbf{L})$.
\end{proof}

\begin{remark} Note that the discrete set $\Sigma$ is not necessarily closed. However, all accumulation points in $\mathbb{C} \setminus \Sigma$ are included in $\sigma(\mathbf{A})$ by the following argument. A sequence of eigenvalues $\mu_j$ in $\Sigma \subset \sigma_p(\mathbf{L})$ has corresponding eigenfunctions such that the singular sequence of normalized eigenfunctions implies $\lim_{j \to \infty} \mu_j \in \sigma(\mathbf{L})$ which is a subset of $\sigma(\mathbf{A}) \cup \Sigma$, hence, $\lim_{j \to \infty} \mu_j \in \sigma(\mathbf{A})$. 
 \end{remark}

%Furthermore, the following results are valid for systems with a vector-valued shadow component $v \in \mathbb{R}^k$, see \cite[Section 5.2.1]{Kowall}.

Lemma \ref{Sigma} can be extended to shadow operators with additional diffusing components. As the Laplacian and the integral operator have compact resolvents, this part of the spectrum remains discrete. However, the multiplication operator induced by the ODE subsystem $\mathbf{A}$ causes problems while inverting the operator $\lambda I - \mathbf{L}$. As shown in \cite[Theorem B.1]{dynspike}, there holds $\sigma_p(\mathbf{A}) \subset \sigma_p(\mathbf{L})$. A characterization of the spectrum of the multiplication operator $\mathbf{A}$ is given in Proposition \ref{essspectrum}. The fact that the spectrum $\sigma(\mathbf{A})$ is essential enables us to verify that $\sigma(\mathbf{A})$ is a part of the spectrum of the shadow operator $\mathbf{L}$. 

\begin{proposition} \label{HardtLemma}
Let $\mathbf{A}, \mathbf{B}, \mathbf{C}, D$ be matrix-valued functions with entries in $L^\infty(\Omega)$ according to the shadow operator $\mathbf{L}$ defined by \eqref{Hardtint} on $L^p(\Omega)^m \times \mathbb{R}$ for some $1 \le  p \le \infty$. Then there holds
\begin{align*}
\sigma(\mathbf{L}) = \sigma (\mathbf{A}) \mathop{\dot{\cup}} \Sigma,
\end{align*}
where $\Sigma \subset \sigma_p(\mathbf{L})$ is the discrete (possibly empty) set defined in Lemma \ref{Sigma}. Moreover, the spectrum $\sigma(\mathbf{L})$ is independent of $1 \le p \le \infty$.
\end{proposition}

\begin{proof}
It remains to show $\sigma(\mathbf{A}) \subset \sigma(\mathbf{L})$, since from considerations in Lemma \ref{Sigma} we already obtained
\[
 \rho (\mathbf{A}) \cap \Sigma = \Sigma \subset \sigma(\mathbf{L}) \subset \sigma (\mathbf{A}) \mathop{\dot{\cup}} \Sigma.
\]
We apply \cite[Theorem 2.2]{Atkinson} to show equality of the essential Wolf spectrum of $\mathbf{L}$ and $\mathbf{A}$, i.e., $\sigma_{\mathrm{ess}}(\mathbf{A}) = \sigma_{\mathrm{ess}}(\mathbf{L})$. The result of Proposition \ref{essspectrum} shows 
\begin{align*}
\sigma(\mathbf{A}) = \sigma_{\mathrm{ess}} (\mathbf{A}) := \{\lambda \in \mathbb{C} \mid \lambda I - \mathbf{A} \; \text{is not a  Fredholm operator} \}.
\end{align*}
Hence, it remains to show $\sigma_{\mathrm{ess}}(\mathbf{A}) = \sigma_{\mathrm{ess}}(\mathbf{L})$. In order to apply \cite[Theorem 2.2]{Atkinson}, we permute the operator matrix $\mathbf{L}$ in \eqref{Hardtint}. Let us consider the invertible permutation matrix
\[
\mathbf{P} = \begin{pmatrix}
\mathbf{0} & 1 \\
I & \mathbf{0}
\end{pmatrix} \in \mathbb{R}^{(m+1) \times (m+1)} \qquad \text{with} \qquad \mathbf{P}^2 = I
\]
which permutes the single shadow component with the remaining $m$ ODE components. Then $\lambda I - \mathbf{L}$ is a Fredholm operator if and only if $\lambda I - \tilde{\mathbf{L}}$ is Fredholm where $\tilde{\mathbf{L}} = \mathbf{P}^{-1}\mathbf{L} \mathbf{P} \in \mathcal{L}(\mathbb{R} \times L^p(\Omega)^m)$, hence $\sigma_{\mathrm{ess}}(\mathbf{L}) = \sigma_{\mathrm{ess}}(\tilde{\mathbf{L}})$. This is a consequence of the fact that $\mathbf{P} \in \mathcal{L}(\mathbb{R} \times L^p(\Omega)^m; L^p(\Omega)^m \times \mathbb{R})$ is Fredholm by invertibility and $\lambda I - \tilde{\mathbf{L}} = \mathbf{P}^{-1} (\lambda I - \mathbf{L}) \mathbf{P}$ is a composition of Fredholm operators \cite[Ch. 6]{Brezis}. We apply the results of \cite{Atkinson} to the operator $\tilde{\mathbf{L}} \in \mathcal{L}(\mathbb{R} \times L^p(\Omega)^m)$ given by
\[
\tilde{\mathbf{L}} :=  \begin{pmatrix} \tilde{A} & \tilde{B}\\
\tilde{C} & \tilde{D}
   \end{pmatrix}
\]
where we take the bounded multiplication operators $\tilde{D} := \mathbf{A}$ on $X_2:= L^p(\Omega)^{m}$ and $\tilde{A}:=\langle D \rangle_\Omega$ on $X_1:=\mathbb{R}$, where the latter operator has a compact resolvent on $\mathbb{R}$. The operators $\tilde{B}:=\langle \mathbf{C} \cdot \rangle_\Omega: X_2 \to X_1$ and $\tilde{C} :=\mathbf{B}: X_1 \to X_2$ consist of bounded integral and multiplication operators. Note that the operators $S(\mu) = S_0 + M(\mu)$ for $\mu \in \rho(\tilde{A})$ in assumption (e) of their paper is given by $S_0 = \tilde{D}= \mathbf{A}$, and the operator $M(\mu)= -\tilde{C}(\mu I  - \tilde{A})^{-1}\tilde{B}$ which is compact for each $\mu \in \rho(\tilde{A})$ by \cite[Proposition 6.3]{Brezis}. Since $\mathbf{L}$ is bounded, the equality $\sigma_{\mathrm{ess}}(\tilde{\mathbf{L}}) = \sigma_{\mathrm{ess}}(S_0)$ is a consequence of \cite[Theorem 2.2]{Atkinson}. This shows $ \sigma_{\mathrm{ess}}(\mathbf{L})=  \sigma_{\mathrm{ess}}(\mathbf{A})$.\\
Finally, the spectrum is independent of $1 \le p \le \infty$ by Proposition \ref{essspectrum} for $\sigma(\mathbf{A})$ and Lemma \ref{Sigma} for $\Sigma$.
\end{proof}

As in the finite-dimensional case, one has to be careful while deducing stability if the spectral bound of $\mathbf{L}$ is zero \cite[Ch. III, Theorem 1.11]{Eisner}. However, the case $s(\mathbf{L})<0$ implies exponential stability of the corresponding evolution system $\mathcal{W}$.

\begin{corollary} \label{stability}
Let Assumptions \ref{AssumptionN}--\ref{AssumptionB} hold for a stationary shadow solution $(\overline{\mathbf{u}}, \overline{v}) \in L^\infty(\Omega)^{m} \times \mathbb{R}$ of shadow system \eqref{shadow} with $\mathbf{D} \equiv \mathbf{0}$. Then $s(\mathbf{L}) <0$ implies that Assumptions \ref{AssumptionL1p}--\ref{AssumptionL0} hold for $p=\infty$ and some $\mu, \sigma>0$ and, in particular, the assertion of Theorem \ref{TheoremToy} remains valid. Moreover, Assumption \ref{AssumptionL1p} does not hold if $s(\mathbf{L})>0$.
\end{corollary}

\begin{proof}
By Proposition \ref{HardtLemma}, we may consider $\mathbf{L}$ on $L^\infty(\Omega)^m \times \mathbb{R}$. By \cite[Ch. V, Theorem 1.10]{Engel}, uniform exponential stability can be deduced from a negative spectral bound $s(\mathbf{L})<0$. Hence, Assumption \ref{AssumptionL1p} is satisfied for $p=\infty$ and some $\sigma>0$. From the spectral decomposition in Proposition \ref{HardtLemma}, we infer that $s(\mathbf{A}) \le s(\mathbf{L})<0$, i.e., by the same arguments, $\mathcal{U}$ is uniformly exponentially stable on $L^\infty(\Omega)^m$.\\ Note that the growth bound of the semigroup equals the spectral bound $s(\mathbf{L})$. As Assumption \ref{AssumptionL1p} would imply a growth bound of the semigroup which is non-positive, this is impossible in the case $s(\mathbf{L})>0$. 
\end{proof}

\begin{remark} \label{rem:stabilityshadow}
The conclusion of Corollary \ref{stability} remains the same  for the general shadow system \eqref{shadow} if one replaces the spectral condition $s(\mathbf{L})<0$ by $s(\mathbf{D}_0 \Delta + \mathbf{L})<0$ and $p=\infty$ by any $1< p<\infty$. In this case $\mathbf{D}_0 \Delta + \mathbf{L}$ is considered as a bounded perturbation of the operator $\mathbf{D}_0 \Delta$. By Proposition \ref{heathom} and  \cite[Ch. III, Proposition 1.12]{Engel}, $\mathbf{D}_0 \Delta + \mathbf{L}$ generates an analytic semigroup on $L^p(\Omega)^m \times \mathbb{R}$ for each $1< p< \infty$. Finally, uniform exponential stability can be deduced from a negative spectral bound \cite[Ch. V, Theorem 1.10]{Engel}.
\end{remark}

\subsubsection{Multiplication operator} \label{sec:multoperator}

Each space-dependent matrix $\mathbf{A} \in L^\infty(\Omega)^{m \times m}$ induces a corresponding multiplication operator
\[
\mathbf{M}_{\mathbf{A}}: L^p(\Omega)^m \to L^p(\Omega)^m, \qquad \mathbf{z} \mapsto \mathbf{A} \mathbf{z}
\]
where $(\mathbf{A}\mathbf{z})(x) :=\mathbf{A}(x)\mathbf{z}(x)$ for each $\mathbf{z} \in L^p(\Omega)^m, x \in \Omega$. Since $\|\mathbf{M}_{\mathbf{A}}\| \le \|\mathbf{A}\|_\infty$, this is a bounded, linear operator for each $1 \le p \le \infty$. We simply write $\mathbf{A}$ instead of $\mathbf{M}_{\mathbf{A}}$. As Proposition \ref{HardtLemma} shows, knowledge of the spectrum of the multiplication operator $\mathbf{A}$ allows us to characterize the spectrum of the shadow operator $\mathbf{L}$ defined in \eqref{Hardtint}. We refer to \cite[Ch. IX]{KJEngel} and \cite[Sections 1-3]{Hardt} for several characterizations of the spectrum $\sigma(\mathbf{A})$ of the multiplication operator $\mathbf{A}$ on $L^p(\Omega)^{m}$ for $1 \le p < \infty$. The following result concerning the essential spectrum is known for the case $p=2$ by \cite[Proposition 3.2, Corollary 3.4]{Hardt} and for the scalar case by \cite[Proposition 3]{HTakagi}. A generalization to arbitrary exponents $1 \le p \le \infty$ is given next.

\begin{proposition} \label{essspectrum}
Let $\mathbf{A} \in L^\infty(\Omega)^{m \times m}, m \in \mathbb{N},$ and let $\mathbf{A}$ denote its corresponding multiplication operator on $L^p(\Omega)^{m}$ for some $1 \le p \le \infty$. Then there exists a null set $N \subset \Omega$ such that
\begin{align}
\sigma(\mathbf{A}) = \overline{ \bigcup_{x \in \Omega \setminus N} \sigma(\mathbf{A}(x)) }. \label{specmult}
\end{align}
Moreover, the whole spectrum is essential in the sense of Wolf, i.e.,
\begin{align*}
\sigma(\mathbf{A}) = \sigma_{\mathrm{ess}} (\mathbf{A}) = \{\lambda \in \mathbb{C} \mid \lambda I - \mathbf{A} \; \text{is not a  Fredholm operator} \}.
\end{align*}
\end{proposition}

\begin{proof}
Boundedness of the multiplication operator leads to a non-empty resolvent set $\rho(\mathbf{A}) \not= \emptyset$. For $1\le p < \infty$, \cite[Ch. IX, Theorem 2.4]{KJEngel} states
\[
\sigma(\mathbf{A}) = \left\{ \lambda \in \mathbb{C} \mid |N_{\lambda, \varepsilon}| >0 \quad \forall \; \varepsilon>0 \right\} =: \mathrm{ess-}\sigma(\mathbf{A}(\Omega)),
\]
for measurable sets 
\[
N_{\lambda, \varepsilon} := \{ x \in \Omega \mid  \mathrm{dist}(\lambda, \sigma(\mathbf{A}(x))) < \varepsilon \}. 
\]
On the one hand, along the same lines of that proof, $\sigma(\mathbf{A}) \subset \mathrm{ess-}\sigma(\mathbf{A}(\Omega))$ also holds for $p=\infty$. On the other hand, the proof of $\sigma(\mathbf{A}) \supset \mathrm{ess-}\sigma(\mathbf{A}(\Omega))$ given in \cite[Ch. IX, Theorem 2.4]{KJEngel} does not apply for $p=\infty$. In order to prove the above representation \eqref{specmult} of the spectrum, it remains to show the inclusion $\mathrm{ess-}\sigma(\mathbf{A}(\Omega)) \subset \sigma(\mathbf{A})$ and \cite[Ch. IX, Remark 2.3]{KJEngel} yields the result.\\
Using the idea of \cite[Theorem 3.3]{Hardt}, we show that each $\lambda \in \mathrm{ess-}\sigma(\mathbf{A}(\Omega))$ is in the spectrum of $\mathbf{A}$. As the characteristic polynomial of the matrix $\mathbf{A}(x)$ factorizes with eigenvalues $\lambda_i(x) \in \mathbb{C}$, we obtain
\begin{align}
|\det(\lambda  I - \mathbf{A}(x))| = \prod_{i=1}^{m} |\lambda - \lambda_i(x)| \ge \mathrm{dist}(\lambda, \sigma(\mathbf{A}(x)))^{m} \label{detdist}
\end{align}
 for a.e. $x \in \Omega$. This estimate yields the inclusion
\[
\Gamma_{\lambda, \varepsilon} := \left\{x \in \Omega \mid |\det(\lambda I -\mathbf{A}(x))| < \varepsilon^{m} \right\} \subset  N_{\lambda, \varepsilon} 
\] 
and we conclude that
\begin{align}
\Gamma_\lambda := \{x \in \Omega \mid \det(\lambda I -\mathbf{A}(x))=0\} \label{det=0}
\end{align}
satisfies $0 \le |\Gamma_\lambda| \le \lim_{\varepsilon \to 0} |N_{\lambda, \varepsilon}|$ as the limit of the above subsets of $N_{\lambda, \varepsilon}$ as $\varepsilon \to 0$. The sequence $(|\Gamma_{\lambda, \varepsilon}|)_{\varepsilon>0}$ of non-negative numbers is non-increasing as $\varepsilon \to 0$ with a limit which is either positive or zero. % The sequence $(|N_{\lambda, \varepsilon}|)_{\varepsilon>0}$ is non-increasing as $\varepsilon \to 0$ with a limit either being positive or zero.\\
In the former case, we conclude that $\Gamma_\lambda$ defined in \eqref{det=0} has positive measure 
%\begin{align}
%|\Gamma_\lambda| \le \lim_{\varepsilon \to 0} |N_{\lambda, \varepsilon}| >0 \label{pointspec}
%\end{align}
which is equivalent to $\lambda \in \sigma_p(\mathbf{A})$ using \cite[Theorem 2.1]{Heymann} or \cite[Theorem 2.5]{Hardt}. In the latter case, $|\Gamma_\lambda| =\lim_{\varepsilon \to 0} |\Gamma_{\lambda, \varepsilon}| =0$, we show that the injective operator $\lambda I - \mathbf{A}$ is not bounded from below, hence $\lambda \in \sigma(\mathbf{A})$. Although we know from $\lambda \in  \mathrm{ess-}\sigma(\mathbf{A}(\Omega))$ that $|N_{\lambda, \varepsilon}|>0$ for all $\varepsilon>0$, there are still two possibilities for the zero sequence $(|\Gamma_{\lambda, \varepsilon}|)_{\varepsilon>0}$: either $|\Gamma_{\lambda, \varepsilon}| >0$ for all $\varepsilon>0$ or the sequence becomes stationary in the sense that $|\Gamma_{\lambda, \varepsilon}| =0$ for all $0 <\varepsilon \le \varepsilon_0$ and some $\varepsilon_0>0$. In both cases we construct a sequence $(\mathbf{f}_j)_{j \in \mathbb{N}} \subset L^p(\Omega)^{m}$ with $\|\mathbf{f}_j\|_{L^p(\Omega)^{m}} =1$ for which $\|(\lambda I - \mathbf{A})\mathbf{f}_j\|_{L^p(\Omega)^{m}} \to 0$ as $j \to \infty$, hence $\lambda I - \mathbf{A}$ can not be bounded from below.
\begin{itemize}
\item[•] Let $|\Gamma_{\lambda, \varepsilon}| >0$ for all $\varepsilon>0$. Thus, we are able to extract a decreasing subsequence $(\Gamma_{\lambda, \varepsilon_j})_{j \in \mathbb{N}}$ with $\varepsilon_j \to 0$ as $j \to \infty$ such that 
\[
|\Gamma_{\lambda, \varepsilon_j}|>0, \qquad \Gamma_{\lambda, \varepsilon_{j+1}} \subset \Gamma_{\lambda, \varepsilon_{j}} \qquad \text{and} \qquad \left| \Gamma_{\lambda, \varepsilon_{j}} \setminus \Gamma_{\lambda, \varepsilon_{j+1}} \right| >0. 
\]
By choosing measurable sets $M_j \subset \Gamma_{\lambda, \varepsilon_{j}} \setminus \Gamma_{\lambda, \varepsilon_{j+1}}$ with $|M_j|>0$ for all $j \in \mathbb{N}$, we obtain the estimate
\begin{align}
\varepsilon_{j+1}^{m} \le |\det(\lambda I - \mathbf{A}(x))|   < \varepsilon_j^{m} \qquad \forall \;  x \in M_j. \label{invertibility}
\end{align}
This enables us to apply \cite[Lemma 3.1]{Hardt} to the matrix $(\lambda - \mathbf{A}(x))^{-1}$. Consequently, we find measurable vector-valued functions $\mathbf{v}_j: M_j \to \mathbb{C}^{m}$ satisfying 
\begin{align*}
|\mathbf{v}_j(x)|_2 = 1 \qquad \text{and} \qquad \left|(\lambda I - \mathbf{A}(x))^{-1}\mathbf{v}_j(x)\right|_2 = \left|(\lambda I - \mathbf{A}(x))^{-1}\right|_2
\end{align*}
for all $x \in M_j$, where we used $| \cdot|_2$ for the Euclidean norm on $\mathbb{C}^{m}$ and for the induced matrix norm. Define $\mathbf{u}_j(x) = (\lambda  I - \mathbf{A}(x))^{-1} \mathbf{v}_j(x)$ as well as functions $\mathbf{f}_j \in L^p(\Omega)^{m}$ by
\[
\mathbf{f}_j(x) = c_p(j) \frac{\mathbf{u}_j(x)}{|\mathbf{u}_j(x)|_2}  \mathds{1}_{M_j}(x)
\]
where $c_p(j)=|M_j|^{-1/p}$ for $p< \infty$ and $c_p(j)=1$ for $p=\infty$. %Here, we fix $|\cdot|_2$ as the vector norm on $\mathbb{C}^{m}$. 
This implies the normalization
$
\|\mathbf{f}_j\|_{L^p(\Omega)^{m}}^p = \int_\Omega |\mathbf{f}_j(x)|_2^p \; \mathrm{d}x =1,
$
with an obvious modification for $p=\infty$. Applying $\lambda I - \mathbf{A}$ to $\mathbf{f}_j$ yields 
\[
(\lambda I - \mathbf{A}(x))\mathbf{f}_j(x) = c_p(j) \mathds{1}_{M_j}(x) \mathbf{v}_j(x) \left|(\lambda I - \mathbf{A}(x))^{-1}\right|_2^{-1}.
\]
From the invertibility condition \eqref{invertibility} we infer 
\[
 \left|(\lambda I - \mathbf{A}(x))^{-1}\right|_2^{-1} \le \mathrm{dist}(\lambda, \sigma(\mathbf{A}(x))) \qquad \forall \;  x \in M_j
 \]
where we used \cite[Ch. IV, Corollary 1.14]{Engel}. A combination of estimates \eqref{detdist} and \eqref{invertibility} yields 
$
 \mathrm{dist}(\lambda, \sigma(\mathbf{A}(x))) < \varepsilon_j
$
for $x \in M_j$, which implies 
$
\|(\lambda I- \mathbf{A}) \mathbf{f}_j\|_{L^p(\Omega)^{m}} \le \varepsilon_j.
$
Since $\varepsilon_j \to 0$, $\lambda$ is an approximate eigenvalue of $\mathbf{A}$.%, i.e., $\lambda \in \sigma(\mathbf{A})$. 

\item[•] Let $|\Gamma_{\lambda, \varepsilon}| =0$ for all $0 <\varepsilon \le \varepsilon_0$. The definition of $\Gamma_{\lambda, \varepsilon}$ yields the pointwise invertibility condition 
\begin{equation*}
|\det(\lambda I - \mathbf{A}(x))|  \ge \varepsilon_0^{m} >0 \qquad \text{for a.e.} \quad  x \in \Omega.% \label{invertibility2}
\end{equation*}
Taking $M_j:= N_{\lambda, \varepsilon_j} \subset \Omega$ with $|M_j|>0$ for any zero sequence $(\varepsilon_j)_{j \in \mathbb{N}}$, we find, similar to the above reasoning, a sequence $(\mathbf{f}_j)_{j \in \mathbb{N}} \subset L^p(\Omega)^{m}$ satisfying 
$
\|(\lambda I- \mathbf{A}) \mathbf{f}_j\|_{L^p(\Omega)^{m}} \le \varepsilon_j.
$
Since $\varepsilon_j \to 0$, $\lambda$ is an approximate eigenvalue of $\mathbf{A}$. Note that in this case, $N_{\lambda, \varepsilon_j}$ cannot become stationary since then $M_j$ and $\mathbf{f}_j$ would become stationary which implies $(\lambda I - \mathbf{A}) \mathbf{f}_j = \mathbf{0}$ -- a contradiction to $\mathbf{f}_j \not= \mathbf{0}$.
\end{itemize}

\noindent It remains to show that $\lambda I - \mathbf{A}$ is not a Fredholm operator for all $\lambda \in \sigma(\mathbf{A})$. To do so, we prove that for each $\lambda \in \sigma(\mathbf{A})$ either $\lambda I - \mathbf{A}$ has no closed range or an infinite-dimensional kernel.\\ %This implies that $\lambda I - \mathbf{A}$ is not a Fredholm operator.\\
If $\lambda \in \sigma_p(\mathbf{A})$, we note that the results \cite[Lemma 2.4, Theorem 2.5]{Hardt} hold independently of $1 \le p \le \infty$. Hence, the first part of the proof of \cite[Proposition 3.2]{Hardt} is still applicable: we infer $\sigma_p(\mathbf{A}) \subset \sigma_{\mathrm{ess}}(\mathbf{A})$ from an infinite-dimensional kernel of $\lambda I - \mathbf{A}$ containing a subspace isomorphic to $L^p(\Gamma_\lambda)$ \cite[Corollary 2.6]{Hardt}.\\
If $\lambda \in \sigma(\mathbf{A}) \setminus \sigma_p(\mathbf{A})$, we necessarily have $|\Gamma_\lambda|=\lim_{\varepsilon \to 0} |\Gamma_{\lambda, \varepsilon}| =0$. From the above reasoning we know that $\lambda I -\mathbf{A}$ is not bounded from below. Thus, the injective operator $\lambda I - \mathbf{A}$ cannot have closed range by \cite[Theorem 2.19, Remark 18]{Brezis} and $\lambda I - \mathbf{A}$ is not a Fredholm operator. %semi-Fredholm for $\lambda \in \sigma(\mathbf{A}) \setminus \sigma_p(\mathbf{A})$. 
%Recall that the constructed sequence $(\mathbf{f}_j)_{j \in \mathbb{N}} \subset L^p(\Omega)^{m}$ is in fact singular, see \cite[Ch. 9, Definition 1.2]{Edmunds}, subject to a similar choice of disjoint sets $M_j$ in the second case above.
%Note that in case of $\lim_{\varepsilon \to 0} |\Gamma_{\lambda, \varepsilon}| =0$ we actually obtained a singular sequence above \cite[Ch. 9, Definition 1.2]{Edmunds}: a sequence $(\mathbf{f}_j)_{j \in \mathbb{N}} \subset L^p(\Omega)^{m}$ with $\|\mathbf{f}_j\|_{L^p(\Omega)^{m}} =1$ for which $\|(\lambda I - \mathbf{A})\mathbf{f}_j\|_{L^p(\Omega)^{m}} \to 0$ holds, but which does not have a convergent subsequence $(\mathbf{f}_{j_k})_{k \in \mathbb{N}}$. In the first case, $(\mathbf{f}_j)_{j \in \mathbb{N}} \subset L^p(\Omega)^{m}$ does not contain any convergent subsequence since the sets $M_j$ are pairwise disjoint: clearly, for $j \not=\ell$ we have $\|\mathbf{f}_j- \mathbf{f}_\ell\|_{L^\infty(\Omega)^{m}} = 1$ and 
%\[
%\|\mathbf{f}_j- \mathbf{f}_\ell\|_{L^p(\Omega)^{m}}^p = 2
%\]
%for all $1 \le p < \infty$. In the second case, the sets $M_j$ may be chosen disjoint in a similar way since $M_j \subset N_{\lambda, \varepsilon_j}$ is arbitrary.
\end{proof}

Note that the above proof may be shortened extremely for the cases $1 \le p< \infty$. One can essentially use the same method of proof from \cite[Proposition 3.2]{Hardt} for the case $p=2$ having the characterization for the dual multiplication operator from \cite[Ch. IX, Proposition 1.4]{KJEngel} in mind. %Unfortunately, this method does not work for $p=\infty$.

\subsection{Dissipative linear shadow system} 

Theorem \ref{Theorem} can be proven in a similar way for the Hilbertian case $p=2$ using energy estimates. The approach is based on a linear shadow system of the form \eqref{shadlinear} which is dissipative in $L^2(\Omega)^m \times \mathbb{R}$. Dissipative systems are a smaller class of systems for which uniform stability of the corresponding evolution system $\mathcal{W}$ is often easier to verify. In particular, one reduces considerations to the linearized shadow system without diffusion, compare Assumption \ref{AssumptionD1p}. This consideration is especially helpful when we look at classical shadow systems which contain diffusing components. \\

Let us first consider the evolutionary subsystem $\mathcal{U}$ from Assumption \ref{AssumptionL0} to describe the principle of dissipativity, see Assumption \ref{AssumptionDp}. %For simplicity, we consider $\mathbf{A}$ instead of $\mathbf{A}_0$ in Assumption \ref{AssumptionL0}, i.e., all $m$ components are non-diffusing.
Special cases of uniform bounded evolution systems are given by contractive or dissipative systems. Concerning contractivity in the time-independent case, \cite[Ch. III, Theorem 2.7]{Engel} applies the condition
\begin{equation}
\|\mathbf{y}\|_{L^p(\Omega)^{m_0}} \le \|(I-\lambda \mathbf{A}_0)\mathbf{y}\|_{L^p(\Omega)^{m_0}} \qquad \forall \; \lambda \in \mathbb{R}_{> 0}, \mathbf{y} \in \mathbf{L}^p(\Omega)^{m_0}. \label{dissip1}
\end{equation} 
The time-dependent case can be treated in a similar way, following \cite{Lovelady}. Therefore, consider the bounded multiplication operators $\mathbf{A}_0(\cdot,t)$ on $L^p(\Omega)^{m_0}$ of non-diffusing components as for Assumption \ref{AssumptionL0} and let us assume

 \begin{assprime}{AssumptionL0}[Dissipativity of ODE subsystem] \label{AssumptionDp}
There exist constants $1 \le p \le \infty, \mu \ge 0$ and a continuous function $\kappa: \mathbb{R}_{\ge 0} \to \mathbb{R}_{\ge 0}$ with $\kappa \in L^1(\mathbb{R}_{\ge 0})$ such that 
\begin{align}
(1-\lambda (\kappa(t)-\mu)) \|\mathbf{y}\|_{L^p(\Omega)^{m_0}} \le \|(I-\lambda \mathbf{A}_0(\cdot,t)) \mathbf{y}\|_{L^p(\Omega)^{m_0}}  \label{dissip2}
\end{align}
is satisfied for all $\mathbf{y} \in L^p(\Omega)^{m_0}$, $\lambda \in \mathbb{R}_{> 0}$, and $t \in \mathbb{R}_{\ge 0}$. 
 \end{assprime}

\noindent Note that $\kappa \equiv \mu \equiv 0$ corresponds to \cite[Theorem 1]{Kato}. %Before providing equivalent formulations of dissipativity condition \eqref{dissip2}, 
It is well-known that Assumption \ref{AssumptionDp} implies uniform boundedness of the corresponding evolution system. 

\begin{proposition} \label{Propdissbound}
Let $\mathbf{A}_0: \Omega \times \mathbb{R}_{\ge 0} \to \mathbb{R}^{m_0 \times m_0}$ be a measurable, locally bounded matrix-valued function satisfying Assumption \ref{AssumptionDp} for some $1 \le p \le \infty, \mu \ge 0$. Then the evolution system $\mathcal{U}$ induced by $\mathbf{A}_0(\cdot, t)$ is uniformly stable on $L^p(\Omega)^{m_0}$, i.e., Assumption \ref{AssumptionL0} is satisfied, for the same exponent $1 \le p \le \infty$ and $\mu \ge 0$.
\end{proposition}

\begin{proof}
See \cite[Remark 2.2, p. 119]{Pavel}.
\end{proof}

Before deriving equivalent formulations of dissipativity condition \eqref{dissip2}, we formulate a similar dissipativity assumption for the linearized shadow evolution system $\mathcal{W}$. Using the notation $\mathbf{D}_0 = \mathrm{diag}(\mathbf{D}, 0)$, let us rewrite the shadow problem \eqref{shadlinear} as an ordinary differential equation in the Banach space $L^p(\Omega)^m \times \mathbb{R}$,
\begin{align*}
\frac{\mathrm{d}}{\mathrm{d} t} \bm{\xi} & = 
\mathbf{D}_0\Delta \bm{\xi} + \mathbf{L}(t) \bm{\xi} \quad  \text{in} \quad \mathbb{R}_{>0}, \qquad \bm{\xi}(0) = \begin{pmatrix}
\bm{\xi}^0_1\\
 \xi_2^0 
\end{pmatrix}.
\end{align*}
By Assumption \ref{AssumptionN}, $(\mathbf{L}(t))_{t \in \mathbb{R}_{\ge 0}}$ is a family of bounded shadow operators on $L^p(\Omega)^m \times \mathbb{R}$. The full operator $\mathbf{D}_0\Delta + \mathbf{L}(t)$ can be seen as a perturbation of the matrix operator $\mathbf{D}_0\Delta$ which generates a contraction semigroup on $L^p(\Omega)^m \times \mathbb{R}$ for all $1 \le p \le \infty$ by Proposition \ref{heathom}. In analogy to Assumption \ref{AssumptionDp} we assume

\begin{assprime}{AssumptionL1p}[Dissipativity of shadow system] 
\label{AssumptionD1p} There exist $1 \le p \le \infty, \sigma \ge 0$ and a continuous function $\varrho: \mathbb{R}_{\ge 0} \to \mathbb{R}_{\ge 0}$ such that $\varrho \in L^1(\mathbb{R}_{\ge 0})$ and
\begin{align}
(1-\lambda (\varrho(t)- \sigma)) \|\mathbf{y}\|_{L^p(\Omega)^m \times \mathbb{R}} \le \|(I-\lambda \mathbf{L}(t)) \mathbf{y}\|_{L^p(\Omega)^m \times \mathbb{R}}  \label{dissip5}
\end{align}
is satisfied for all $\mathbf{y} \in L^p(\Omega)^m \times \mathbb{R}$, $\lambda\in \mathbb{R}_{> 0}$, and $t \in \mathbb{R}_{\ge 0}$. 
\end{assprime}

Using the duality map, one obtains an equivalent integral inequality of the form \eqref{dissip3} for $1 < p < \infty$. The latter has a quite convenient form for $L^2$ energy estimates:
\[
\int_\Omega \mathbf{y}^T \left(\mathbf{L}(\cdot,t)- (\varrho(t)- \sigma) I \right) \mathbf{y} \; \mathrm{d}x \le 0 \qquad \forall \; \mathbf{y} \in L^2(\Omega)^m \times \mathbb{R},  t \in \mathbb{R}_{\ge 0}
\]
Similar to the assertion of Proposition \ref{Propdissbound}, we reach at the following stability result which is proven at the end of this section.

\begin{proposition} \label{Propdissfull}
Let the linear operators $\mathbf{L}(t): L^p(\Omega)^m \times \mathbb{R} \to L^p(\Omega)^m \times \mathbb{R}$ defined above for $t \in \mathbb{R}_{\ge 0}$ satisfy Assumption \ref{AssumptionD1p} for some $1 \le p \le \infty, \sigma \ge 0$ and uniformly bounded coefficients $\mathbf{A}, \mathbf{B}, \mathbf{C}$ and $D$. Then the corresponding shadow evolution system $\mathcal{W}$ induced by $\mathbf{D}_0\Delta + \mathbf{L}(t)$ is uniformly stable on $L^p(\Omega)^m \times \mathbb{R}$, thus $\mathcal{W}$ satisfies Assumption \ref{AssumptionL1p}, for the same exponent $p$ and $\sigma$. Moreover, Assumption \ref{AssumptionL0} is satisfied for all $1 \le p \le \infty$ and $\mu= \sigma$.
\end{proposition}

Note that, as in the case of a stationary shadow solution, Assumption \ref{AssumptionD1p} implies both stability conditions \ref{AssumptionL1p}--\ref{AssumptionL0}. In view of Proposition \ref{Propdissfull}, Assumptions \ref{AssumptionL1p}--\ref{AssumptionL0} can be replaced by the above dissipativity assumption. 

\begin{corollary} \label{TheoremDissip}
Let Assumptions \ref{AssumptionN}--\ref{AssumptionB} and \ref{AssumptionD1p} hold for some $1 \le p \le \infty$ with $p>n/2$ if $n \ge 2$ and $\sigma \ge 0$. Then the assertion of Theorem \ref{Theorem} remains valid and, if $p= \infty$, the assertion of Theorem \ref{TheoremToy} is valid.
\end{corollary}

\begin{proof}
Let Assumption \ref{AssumptionD1p} hold for the exponent $1 \le p \le \infty$ and $\sigma \ge 0$ which are required in the theorems. By Proposition \ref{Propdissfull}, Assumption \ref{AssumptionL1p} is fulfilled with the same parameters and Assumption \ref{AssumptionL0} is satisfied.
\end{proof}

For the proof of Proposition \ref{Propdissfull}, we study several equivalent formulations of dissipativity condition \eqref{dissip2} and, in particular, we show that Assumption \ref{AssumptionDp} is independent of $1 \le p \le \infty$. \\

By \cite[Remark 1.2]{Pavel}, condition \eqref{dissip2} is equivalent to dissipativity of $\mathbf{A}_0(\cdot,t)-(\kappa(t)-\mu)I$ on $L^p(\Omega)^{m_0}$ in the sense of inequality \eqref{dissip1} for each time $t \in \mathbb{R}_{\ge 0}$. It turns out that there is even a simpler criterion for dissipativity of multiplication operators on $L^p(\Omega)^{m_0}$ which is independent of the exponent $1 \le p \le \infty$. More precisely, inequality \eqref{dissip2} can be verified via pointwise estimates of the corresponding quadratic form
\begin{align}
Q(x, t): \mathbb{R}^{m_0}  \to \mathbb{R}, \qquad \overline{\mathbf{y}} \mapsto  \overline{\mathbf{y}}^T (\mathbf{A}_0(x,t)- (\kappa(t)-\mu) I)  \overline{\mathbf{y}}. \label{quadform}
\end{align}
Such a condition is already used in the time-independent case in \cite[Propositions 6, 7]{Ouhabaz1999}. Moreover, pointwise estimates of the latter quadratic form are a well-known technique in the context of classical solutions to preserve contractivity of the corresponding evolution system \cite[Theorem 2.3]{Kresin}.

\begin{lemma} \label{quadformLem}
Let $\kappa, \mu$ and $1 \le p \le \infty$ be given by Assumption \ref{AssumptionDp} for some measurable, bounded function $\mathbf{A}_0: \Omega \times \mathbb{R}_{\ge 0} \to \mathbb{R}^{m_0 \times m_0}$ and let $Q$ be defined as above in \eqref{quadform}. Then dissipativity condition \eqref{dissip2} on $L^p(\Omega)^{m_0}$ is equivalent to $Q(x,t) \le 0$ for a.e. $(x,t) \in \Omega \times \mathbb{R}_{\ge 0}$. Moreover, inequality \eqref{dissip2} holds for all $1 \le p \le \infty$ if and only if it holds for one exponent $p$.
\end{lemma}

\begin{proof}
By \cite[Remark 1.2]{Pavel}, estimate \eqref{dissip2} is equivalent to dissipativity of $\mathbf{A}_0(\cdot, t)-(\kappa(t)-\mu)I$. The latter means that for all $\lambda \in \mathbb{R}_{> 0}$, $\mathbf{y} \in L^p(\Omega)^{m_0}$, and $t \in \mathbb{R}_{\ge 0}$ there holds
\begin{align}
\|\mathbf{y}\|_{L^p(\Omega)^{m_0}} \le \|(I-\lambda (\mathbf{A}_0(\cdot, t)-(\kappa(t)-\mu) I))\mathbf{y}\|_{L^p(\Omega)^{m_0}}. \label{dissip2a}
\end{align}
A further characterization of dissipativity via the duality map $J$ as in \cite[Ch. II, Proposition 3.23]{Engel} can be used to rewrite this condition. For $1 < p< \infty$, the duality set is just a singleton by \cite[Ch. II, Example 3.26 (ii)]{Engel}. More precisely, $J(\mathbf{y}) = \{\mathbf{y}^\ast\}$ for $\mathbf{y}^\ast \in L^q(\Omega)^{m_0}$ where $\mathbf{y}^\ast = \mathbf{0}$ for $\mathbf{y}=\mathbf{0}$ and
\[
\mathbf{y}^\ast = \frac{\mathbf{y}|\mathbf{y}|^{p-2}}{\|\mathbf{y}\|_{L^p(\Omega)^{m_0}}^{p-2}} \qquad \text{for} \quad \mathbf{y} \not=\mathbf{0}. 
\]
Remember that the dual pairing satisfies $\langle \mathbf{y}^\ast, \mathbf{y} \rangle = \|\mathbf{y}^\ast\|_{L^q(\Omega)^{m_0}}^2 = \| \mathbf{y}\|_{L^p(\Omega)^{m_0}}^2$ where $p$ and $q$ are conjugate exponents. Hence, inequalities \eqref{dissip2}--\eqref{dissip2a} are equivalent to the integral condition
\begin{equation}
\int_\Omega (\mathbf{y}^\ast)^T \left(\mathbf{A}_0(\cdot,t)- (\kappa(t)-\mu) I \right) \mathbf{y} \; \mathrm{d}x \le 0 \qquad \forall \; \mathbf{y} \in L^p(\Omega)^{m_0},  t \in \mathbb{R}_{\ge 0}. \label{dissip3}
\end{equation}
Let us first assume $Q \le 0$ for a.e. $(x,t)\in \Omega \times \mathbb{R}_{\ge 0}$. 
We multiply the inequality $Q \le 0$ with a symmetric choice of vectors $ \overline{\mathbf{y}}(x)$ instead of $\mathbf{y}^\ast$ and $\mathbf{y}$ and obtain inequality \eqref{dissip3} for $1 < p < \infty$ after integration. For $p \in \{1, \infty\}$, we use continuity of the $L^p$ norm with respect to $p$ since we already established estimate \eqref{dissip2a} for all $1 < p < \infty$. Since $\mathbf{A}_0-\kappa I$ is bounded, $I-\lambda(\mathbf{A}_0(t)-(\kappa(t)-\mu)I)$ is invertible for small $\lambda>0$, and %due to Neumann's series. 
by \cite[Ch. II, Proposition 3.14]{Engel}, the latter operator is invertible for all $\lambda>0$ and estimate \eqref{dissip2a} yields
\begin{align*}
\| (I-\lambda(\mathbf{A}_0(\cdot, t)-(\kappa(t)-\mu) I))^{-1} \mathbf{y} \|_{L^{p}(\Omega)^{m_0}} & \le \| \mathbf{y}\|_{L^{p}(\Omega)^{m_0}}
\end{align*}
for all $\mathbf{y} \in L^{p}(\Omega)^{m_0}$ and $1 < p< \infty$. The result follows by letting $p \to 1$ or $p \to \infty$ where \cite[Theorem 2.14]{Adams} applies. Thus, $Q \le 0$ implies dissipativity.\\
Now, let dissipativity inequality \eqref{dissip2} be fulfilled and assume $Q$ is not non-positive, i.e., there is a set $\Omega_1 \subset \Omega$ with $|\Omega_1|>0$ and some time point $t \ge 0$ as well as $\overline{\mathbf{y}} \in \mathbb{R}^{m_0} \setminus \{\mathbf{0}\}$ such that
\[
Q(x,t) \overline{\mathbf{y}} =\overline{\mathbf{y}}^T (\mathbf{A}_0(x,t)- (\kappa(t)-\mu) I)  \overline{\mathbf{y}} >0 
\]
holds for a.e. $x \in \Omega_1$. Possibly choosing a smaller set $\Omega_2 \subset \Omega_1$ with positive measure, we find uniform bounds 
$
0 < Q_0 \le Q(\cdot,t)\overline{\mathbf{y}} \le Q_1 < \infty
$
which hold almost everywhere on $\Omega_2$. Let us consider $\mathbf{y}_p \in L^p(\Omega)^{m_0}$ for $p< \infty$ given by
$
\mathbf{y}_p(x):= \mathds{1}_{\Omega_2}(x) (Q(x,t) \overline{\mathbf{y}})^{-1/p}\overline{\mathbf{y}}. 
$
This bounded vector-valued function satisfies
\begin{align*}
\int_\Omega (\mathbf{y}_p^\ast)^T \left(\mathbf{A}_0(\cdot,t)- (\kappa(t)-\mu) I \right) \mathbf{y}_p \; \mathrm{d}x & = \int_{\Omega_2} |\overline{\mathbf{y}}|^{p-2} \|\mathbf{y}_p\|_{L^p(\Omega_2)^{m_0}}^{2-p} \; \mathrm{d}x = \| (Q(\cdot,t) \overline{\mathbf{y}})^{-1} \|_{L^p(\Omega_2)^{m_0}}^{2-p} >0
\end{align*}
which is a contradiction to condition \eqref{dissip3}, and thus to \eqref{dissip2} for $1 < p < \infty$. For $p \in \{1, \infty\}$, let us consider $\mathbf{y}_2 \in L^\infty(\Omega)^{m_0}$ in preceding definition. Then
\begin{align*}
\left| (I-\lambda (\mathbf{A}_0(\cdot, t)-(\kappa(t)-\mu) I))\mathbf{y}_2\right|^2 & =|\mathbf{y}_2|^2+ \lambda^2 \left|(\mathbf{A}_0(\cdot, t)-(\kappa(t)-\mu) I)\mathbf{y}_2 \right|^2 - 2 \lambda
\end{align*}
holds on the set $\Omega_2$. For small enough $\lambda> 0$, the right-hand side of the latter equation is smaller than $|\mathbf{y}_2|^2$ since $\mathbf{A}_0, \kappa, \mathbf{y}_2$ are bounded functions on $\Omega_2$. This leads to a contradiction to condition \eqref{dissip2a}.% since
%\[
%\left\| \left(|\mathbf{y}_2|^2 \right)^{1/2} \right\|_{L^p(\Omega_2)} \le \left\| \left(|\mathbf{y}_2|^2 -  \lambda \right)^{1/2} \right\|_{L^p(\Omega_2)}
%\]
\end{proof}

Hence, Assumption \ref{AssumptionDp} is independent of $1 \le p \le \infty$. Note that Assumption \ref{AssumptionL0} only considers the subsystem $\mathbf{A}_0$ of $\mathbf{A}$ of non-diffusing components. %If we are able to verify dissipativity in $L^p(\Omega)^{m}$ for some (and hence all) $1 \le p \le \infty$, then Assumption  \ref{AssumptionL0} is satisfied. 
Since $\mathbf{A}_0$ is non-symmetric in general, one can verify definiteness of the corresponding quadratic form $Q$ defined in \eqref{quadform} by looking equivalently on the real eigenvalues of the symmetric part 
\[
\frac{1}{2} \left(\mathbf{A}_0(x,t) +  \mathbf{A}_0(x,t)^T \right) - (\kappa(t)-\mu) I \in \mathbb{R}^{m_0 \times m_0}.
\]
Non-positivity of its eigenvalues $\lambda(x,t)$ pointwise for a.e. $(x,t) \in \Omega \times \mathbb{R}_{\ge 0}$ implies the condition $Q \le 0$ and, hence, stability by Lemma \ref{quadformLem}.

\begin{proof}[Proof of Proposition \ref{Propdissfull}]
By the same reasoning as in the proof of Proposition \ref{Propdissbound}, we obtain uniform boundedness of the evolution system $\mathcal{W}$ for the chosen exponent $p$. Let us now focus on the evolutionary subsystem $\mathcal{U}$. In view of Lemma \ref{quadformLem} and Proposition \ref{Propdissbound}, it remains to check a dissipativity condition for the chosen value $p$. For $1 < p < \infty$, we infer from condition \eqref{dissip3} that dissipativity of $\mathbf{L}(t)- (\varrho(t)-\sigma) I$ in $L^p(\Omega)^m \times \mathbb{R}$ implies dissipativity of the corresponding subsystem $\mathbf{A}(\cdot,t)- (\varrho(t)-\sigma) I$ in $L^p(\Omega)^m$ since the duality set is given by $\mathbf{y}^\ast = (\mathbf{y}_1^\ast, \mathbf{y}_2)^T$ for $\mathbf{y} = (\mathbf{y}_1, \mathbf{y}_2)^T \in L^p(\Omega)^m \times \mathbb{R}$. For $p \in \{1, \infty\}$, we follow the proof of Lemma \ref{quadformLem}. In this way, condition \eqref{dissip2a} can be shown by contradiction, assuming $Q \le 0$ does not hold almost everywhere. Hence, Assumption \ref{AssumptionDp} is satisfied for all $1 \le p \le \infty$ and $\kappa \equiv \rho, \mu=\sigma$, and Proposition \ref{Propdissbound} yields the desired result.
\end{proof}

\section{Model examples} \label{sec:modelex}

The main results Theorem \ref{TheoremToy} and Theorem \ref{Theorem} provide information on the long-term dynamics of the reaction-diffusion-ODE system \eqref{fullsys} based on results obtained for its shadow limit \eqref{shadow}. In the context of classical shadow systems, there are diverse examples to which the results can be applied \cite{He3, Kondo, Miy05, Peng, Wei}. Recall that system \eqref{fullsys} allows for nonlinearities which explicitly depend on space and time. In this section, we present two further examples in more detail, including the case of non-diffusing components. The first example is a linear model which shows necessity of the stability assumption \ref{AssumptionL1p} in the case of a space-dependent shadow solution. The second model which is treated analytically is of predator-prey-type and exemplifies the global convergence result in Theorem \ref{TheoremToy}. 

\subsection{Linear model} \label{sec:expgrowthlin} 

We focus on an equation for $v_{\varepsilon}$ only,
\begin{align}
   \frac{\partial v_{\varepsilon}}{\partial t} - \frac{1}{\varepsilon} \Delta v_{\varepsilon} & =  D(x) v_{\varepsilon}  \quad  \mbox{in} \quad \Omega \times  \mathbb{R}_{>0},  \qquad v_{\varepsilon}(\cdot, 0) =v^0 \quad \text{in} \quad \Omega, \qquad  \frac{\partial v_{\varepsilon}}{\partial \nu}  = 0  \quad   \mbox{on} \quad \partial \Omega \times  \mathbb{R}_{>0}, \label{vspacedep}
\end{align}
for a space-dependent coefficient $D := w_1 + w_1^2 \in L^\infty(\Omega)$. Here, we take the eigenfunction $v^0:=w_1$ which corresponds to the first positive eigenvalue $\lambda_1$ of $-\Delta$ on $\Omega=(0,1)$, i.e., $w_1(x)= \sqrt{2} \cos(\pi x)$  and $\lambda_1 = \pi^2$ by equation \eqref{spect}. Of course, equation \eqref{vspacedep} can be extended to a full linear reaction-diffusion-ODE system by setting $A, B, C=0$ in \eqref{Jacobian}. The corresponding shadow limit is given by $v=0$ since $\langle v^0 \rangle_\Omega=0$ and $\langle D \rangle_\Omega=1$. Proposition \ref{HardtLemma} implies $\sigma(\mathbf{L}) = \{0,1\}$ and Assumption \ref{AssumptionL1p} is not satisfied by Corollary \ref{stability}. We will verify that the error $V_\varepsilon=v_{\varepsilon}-\psi_\varepsilon$ grows exponentially in time. By H\"older's inequality, it remains to show exponential growth of the spatial mean value $\langle V_\varepsilon \rangle_\Omega= \langle v_{\varepsilon} \rangle_\Omega$. \\

The solution $v_{\varepsilon}$ is given by the implicit integral equation 
\[
v_{\varepsilon}(x, t) = S_\Delta(t/{\varepsilon}) v^0(x) + \int_0^t S_\Delta((t-\tau)/{\varepsilon}) D(x) v_{\varepsilon}(x,\tau) \; \mathrm{d}\tau
\]
which can be solved by a Picard iteration. According to \cite[Part II, Theorem 1]{Rothe}, we define approximations $v_{\varepsilon}^{(j)}(\cdot,t) \in L^\infty(\Omega)$ for $j \in \mathbb{N}$ recursively given by
\begin{align*}
v_{\varepsilon}^{(1)}(\cdot, t) &= S_\Delta(t/{\varepsilon}) v^0 =  \mathrm{e}^{-\lambda_1t/\varepsilon} w_1,\\
v_{\varepsilon}^{(j+1)}(\cdot,t) &= S_\Delta(t/{\varepsilon}) v^0 + \int_0^t  S_\Delta((t-\tau)/{\varepsilon}) \left[ D(\cdot)  v_{\varepsilon}^{(j)}(\cdot,\tau) \right] \; \mathrm{d}\tau.
\end{align*}
We write $D= w_1 + w_1^2 =  w_0 + w_1 + \sqrt{2}^{-1} w_2$, using a product formula for the eigenfunctions $w_j$ of $-\Delta$ given by $w_j(x) = \sqrt{2} \cos(j \pi x)$ for $j \in \mathbb{N}$ and $w_0 \equiv 1$ for $j=0$, see equation \eqref{spect}. We iteratively multiply the coefficient $D$ with $v_{\varepsilon}^{(j)}$ and use that products $w_jw_i$ can be written as a linear combination of $w_{j+i}$ and $w_{|j-i|}$. This procedure yields
\begin{align*}
v_{\varepsilon}^{(2)}(\cdot,t) & = \mathrm{e}^{-\lambda_1 t/{\varepsilon}} w_1 +  \int_0^t  S_\Delta((t-\tau)/{\varepsilon}) f_1^{(1)}(\tau)\left[ w_1 + w_1^2 + \sqrt{2}^{-1} w_1w_2 \right] \; \mathrm{d}\tau\\
& =  \left( \int_0^t f_1^{(1)}(\tau) \; \mathrm{d}\tau \right) w_0 +  \left(\mathrm{e}^{-\lambda_1 t/{\varepsilon}} + \int_0^t\mathrm{e}^{-\lambda_1 (t-\tau)/{\varepsilon}} f_1^{(1)}(\tau) \; \mathrm{d}\tau \right) w_1 \\
& \quad + \int_0^t  S_\Delta((t-\tau)/{\varepsilon}) f_1^{(1)}(\tau) h^{(2)} \; \mathrm{d}\tau,
\end{align*}
where $v_{\varepsilon}^{(1)}(\cdot,t) =: f_1^{(1)}(t) w_1$, $h^{(2)} =  (w_1 +\sqrt{2} w_2 +w_3)/2$ and $w_0 \equiv 1$. To understand the next step, let us rewrite the second approximation as
\[
v_{\varepsilon}^{(2)} (\cdot,t) = f_0^{(2)}(t) w_0 + f_1^{(2)}(t) w_1 + f_2^{(2)}(t) h^{(2)}
\]
and note that the coefficients of the eigenfunctions in $h^{(2)}$ are all positive. % and $h^{(2)}$ includes $w_1$ as well. 
Considering spatial means, $\langle w_j \rangle_\Omega=0$ for all $j \in \mathbb{N}$ implies $\langle v_{\varepsilon}^{(1)} \rangle_\Omega=0$ and 
\[
\langle v_{\varepsilon}^{(2)} (\cdot, t) \rangle_\Omega =  \int_0^t \mathrm{e}^{-\lambda_1 \tau/{\varepsilon}} \; \mathrm{d}\tau. 
\] 
Using again $D = w_0 + w_1+ \sqrt{2}^{-1} w_2$, this leads to the third approximation of the form
\begin{align*}
v_{\varepsilon}^{(3)}(\cdot,t) & =  \left( \int_0^t f_0^{(2)}(\tau) + f_1^{(2)}(\tau) \; \mathrm{d}\tau \right) w_0\\
& \quad + \left( \mathrm{e}^{-\lambda_1t/{\varepsilon}} + \int_0^t  \mathrm{e}^{- \lambda_1(t-\tau)/{\varepsilon}} (f_0^{(2)}(\tau) + f_1^{(2)}(\tau) )\; \mathrm{d}\tau \right) w_1 + f_3^{(3)}(t) h^{(3)},
\end{align*}
where $h^{(3)}$ is a sum of positive multiples of $w_j$ for $j=0, \dots, 4$ and $f_3^{(3)} \ge 0$ is a continuous function in time. Estimating from below, we successively gain for all $j \in \mathbb{N}$ (by setting $f_0^{(1)} \equiv 0$)
\begin{align*}
f_0^{(j+2)}(t) \ge \int_0^t f_0^{(j+1)}(\tau) + f_1^{(j+1)}(\tau) \; \mathrm{d}\tau \ge \int_0^t \left( \mathrm{e}^{-\lambda_1 \tau/{\varepsilon}} +  \int_0^\tau f_0^{(j)}(r) + f_1^{(j)}(r) \; \mathrm{d}r \right) \mathrm{d}\tau. 
\end{align*}
Starting from the innermost double integral and applying Fubini's rule inductively, this yields
\[
 f_0^{(j+2)}(t) \ge \int_0^t f_0^{(j+1)}(\tau) + f_1^{(j+1)}(\tau) \; \mathrm{d}\tau \ge \int_0^t \sum_{i=0}^j \frac{(t-\tau)^i}{i!} \mathrm{e}^{-\lambda_1 \tau/{\varepsilon}}  \mathrm{d}\tau.
 \]
Since $v_{\varepsilon}^{(j)}$ converges to $v_{\varepsilon}$ in $L^\infty(\Omega_T)$ as shown in \cite[Part II, Theorem 1]{Rothe}, we obtain a lower bound due to 
\[
\int_0^t \sum_{i=0}^j \frac{(t-\tau)^i}{i!} \mathrm{e}^{-\lambda_1 \tau/{\varepsilon}}  \mathrm{d}\tau \le  f_0^{(j+2)}(t) \le \langle v_{\varepsilon}^{(j+2)}(\cdot,t) \rangle_\Omega \to \langle v_{\varepsilon}(\cdot,t) \rangle_\Omega.  
\]
The theorem of monotone convergence leads to exponential growth of $\langle V_{\varepsilon} \rangle_\Omega$ since
\begin{align*}
 \langle v_{\varepsilon}(\cdot,t) \rangle_\Omega \ge \int_0^t \mathrm{e}^{t-\tau} \mathrm{e}^{-\lambda_1 \tau/{\varepsilon}} \; \mathrm{d}\tau \ge C\varepsilon \left( \mathrm{e}^t - 1 \right).
\end{align*}
This induces exponential growth of $t \mapsto \|v_{\varepsilon}(\cdot, t)\|_{L^\infty(\Omega)}$ by H\"older's inequality.

\subsection{Predator-prey model} \label{sec:PPmodel} 
Consider a closed system describing predator-prey dynamics, with a predator denoted by $u_\varepsilon$ and a mobile prey $v_\varepsilon$. In fresh-water ecology, a biological example can be given by Hydra and Daphnia where the predator Hydra is sedentary, i.e., $D=0$ \cite[Example (b)]{Mott}. The following model adapted from \cite{Mott} includes both cases $D=0$ and $D>0$. The differential equations read
\begin{align} \label{PP}
\begin{split}
     \frac{\partial u_\varepsilon}{\partial t} -D \Delta u_\varepsilon &= -p u_\varepsilon + b v_\varepsilon  \qquad \quad \; \; \mbox{in} \quad \Omega_T,   \quad \; \; u_\varepsilon (\cdot, 0) =u^0  \quad \text{in} \quad \Omega,  \quad \;
 \frac{\partial u_\varepsilon}{\partial \nu} =0 \quad \mbox{on} \quad \partial \Omega \times (0,T), \\
        \frac{\partial v_\varepsilon}{\partial t} - \frac{1}{\varepsilon} \Delta v_\varepsilon &=   (d -a u_\varepsilon -c v_\varepsilon) v_\varepsilon  \quad \mbox{in} \quad \Omega_T, \quad \; \; v_\varepsilon (\cdot, 0) =v^0 \quad \text{in} \quad \Omega,  \quad \; \;
 \frac{\partial v_\varepsilon}{\partial \nu} =0 \quad \mbox{on} \quad \partial \Omega \times (0,T), 
 \end{split}
\end{align}
where $u_\varepsilon$ is endowed with a zero flux boundary condition if $D>0$. Here, $a,b,c,d,p>0$ are constants and the initial values $u^0, v^0 \ge 0$ are bounded as well as non-negativity almost everywhere in $\Omega$ with $\langle v^0 \rangle_\Omega>0$. The corresponding shadow limit is given by
\begin{align} \label{shadPP}
\begin{split}
\frac{\partial u}{\partial t} -D \Delta u &= -p u + b v \qquad \qquad \; \; \mbox{in} \quad \Omega_T, \qquad   \; \, u (\cdot, 0) =u^0  \quad \text{in} \quad \Omega, \\
   \frac{\mathrm{d} v}{\mathrm{d} t} &= (d - a \langle u \rangle_\Omega -cv)v  \quad \mbox{in} \quad (0,T), \qquad v (0) = \langle v^0 \rangle_\Omega 
   \end{split}
\end{align}
where $u$ is endowed with a zero flux boundary condition if $D>0$.\\
Let us first study dynamics of the shadow problem. If we integrate the mild solution $u$ of system \eqref{shadPP} over $\Omega$, we obtain an ODE system for the masses $(\langle u \rangle_\Omega, v)$. This system admits the global attractor $(\overline{u}, \overline{v})$ where
\[
\overline{u} =  \frac{dp}{cp + ab}  \qquad \text{and} \qquad  b\overline{v} = p \overline{u}.
\]
Convergence to the equilibrium is a consequence of the radially unbounded Lyapunov functional
\[
L(\langle u \rangle_\Omega, v) = \frac{a}{2}(\langle u  \rangle_\Omega - \overline{u})^2 + b (v - \overline{v} - \overline{v} \log(v/\overline{v}))
\]
adapted from \cite{Mott}, where $L$ is dissipative on trajectories, i.e.,
\[
\frac{\mathrm{d}L}{\mathrm{d}t} = - ap (\langle u \rangle_\Omega - \overline{u})^2 - bc (v-\overline{v})^2 \le 0.
\]
Thus, we obtain the asymptotic behavior $(\langle u\rangle_\Omega, v) \to (\overline{u}, \overline{v})$ as $t \to \infty$ as well as $u \to \overline{u}$, since uniformly in space
\[
u(\cdot,t)-\langle u \rangle_\Omega(t) = (S_\Delta(Dt)u^0 - \langle u^0 \rangle_\Omega) \mathrm{e}^{-pt} \to 0 \qquad \text{for} \quad t \to \infty. 
\]
Hence, Assumptions \ref{AssumptionN}--\ref{AssumptionB} is satisfied. For application of Theorem \ref{TheoremToy}, it remains to compute the Jacobian 
\[
\mathbf{J}(x,t) = \begin{pmatrix}
-p  & b \\
-a v(t)  & d - 2cv(t) - a u(x,t) 
\end{pmatrix}
\] 
at the shadow limit $(u,v)$. The shadow evolution system $\mathcal{W}$ defined in \eqref{evolsystemnonlinear} is induced by the operators
\begin{align*}
&\mathbf{D}_0\Delta + \mathbf{L}(t): L^p(\Omega) \times \mathbb{R}  \to L^p(\Omega) \times \mathbb{R}, \\
& \mathbf{L}(t)\begin{pmatrix}
\xi_1\\
\xi_2
\end{pmatrix}(x)  = \begin{pmatrix}
-p \xi_1(x) + b \xi_2 \\
 - av(t) \langle \xi_1 \rangle_\Omega + (d - 2cv(t) - a \langle u(\cdot,t) \rangle_\Omega) \xi_2 \\
\end{pmatrix}, \quad x \in \Omega, t \in \mathbb{R}_{\ge 0},
\end{align*}
where $\mathbf{D}_0= \mathrm{diag}(D, 0) \in \mathbb{R}_{\ge 0}^{2 \times 2}$ is a diagonal matrix.

\begin{lemma}
Let $(u,v)$ be a shadow solution of system \eqref{shadPP} for bounded initial conditions $u^0,v^0 \ge 0$ satisfying $\langle v^0 \rangle_\Omega>0$. Then Assumptions \ref{AssumptionL1p}--\ref{AssumptionL0} is satisfied for $p=\infty$. Moreover, the corresponding evolution system $\mathcal{U}$ and $\mathcal{W}$ is uniformly exponentially stable for the exponent $\eta=p>0$ and some $\sigma>0$, respectively. 
\end{lemma}

\begin{proof}
Assumption \ref{AssumptionL0} is satisfied since $(U(t))_{t \in \mathbb{R}_{\ge 0}}$ with $U(t)=S_\Delta(Dt)\mathrm{e}^{-pt}$ is uniformly exponentially stable with exponent $\eta=p>0$. This is a consequence of contractivity of the heat semigroup $(S_\Delta(t))_{t \in \mathbb{R}_{\ge 0}}$, see Proposition \ref{heathom}. Concerning the evolution system $\mathcal{W}$ in Assumption \ref{AssumptionL1p}, let us split the shadow operator as $\mathbf{L}(t) = \mathbf{L}_\infty + \mathbf{B}(t)$ for operator matrices
\begin{align*}
\mathbf{L}_\infty, \mathbf{B}(t): & \,  L^p(\Omega) \times \mathbb{R}  \to L^p(\Omega) \times \mathbb{R}, \\
\mathbf{L}_\infty\begin{pmatrix}
\xi_1\\
\xi_2
\end{pmatrix}(x) & = \begin{pmatrix}
-p \xi_1(x) + b \xi_2 \\
- a \overline{v} \langle \xi_1 \rangle_\Omega  -c\overline{v} \xi_2 \\
\end{pmatrix} =:  \begin{pmatrix} A_{\ast} \xi_1(x) + B_\ast \xi_2\\
C_\ast \langle  \xi_1 \rangle_\Omega + D_\ast \xi_2
   \end{pmatrix},\\
\mathbf{B}(t)\begin{pmatrix}
\xi_1\\
\xi_2
\end{pmatrix}(x) &= \begin{pmatrix}
0\\
 - a(v(t)-\overline{v}) \langle \xi_1 \rangle_\Omega + \left[- 2c(v(t)-\overline{v}) - a (\langle u(\cdot,t) \rangle_\Omega- \overline{u}) \right] \xi_2\\
\end{pmatrix}.
\end{align*}
Since $\lim_{t \to \infty} \mathbf{B}(t) = \mathbf{0}$ with respect to the operator norm on $L^\infty(\Omega) \times \mathbb{R}$, evolution systems induced by $\mathbf{D}_0\Delta + \mathbf{L}(t)$ and $\mathbf{D}_0\Delta + \mathbf{L}_\infty$ are asymptotically comparable. We will show that it remains to consider the latter semigroup for exponential stability of the former evolution system. To recognize this, we start from the definition of $\mathcal{W}$ in condition \ref{AssumptionL1p}. This evolution system is given by evolution operators $\mathbf{W}(t,s)$ for $t,s \in \mathbb{R}_{\ge 0}, s \le t,$ defined by equation \eqref{evolsystemnonlinear}, where $\bm{\xi} \in C(\mathbb{R}_{\ge 0}; L^p(\Omega) \times \mathbb{R})$ is the unique solution of the shadow problem 
\begin{align*}
\frac{\partial \bm{\xi}}{\partial t} -(\mathbf{D}_0\Delta+ \mathbf{L}_\infty) \bm{\xi} =  \mathbf{B}(t)\bm{\xi}  \quad \text{in} \quad \Omega \times \mathbb{R}_{>0}
\end{align*}
endowed with zero flux boundary conditions if necessary.
We split the full operator into a time-independent, possibly unbounded part $\mathbf{D}_0\Delta + \mathbf{L}_\infty$ and the bounded time-varying operator family $(\mathbf{B}(t))_{t \in \mathbb{R}_{\ge 0}}$. We are able to compare both evolution systems, the system $\mathcal{W}_\infty$ induced by a semigroup $(\mathbf{W}_\infty(t))_{t \in \mathbb{R}_{\ge 0}}$ which is generated by the operator $\mathbf{D}_0\Delta + \mathbf{L}_\infty$ and the full evolution system $\mathcal{W}$, using the integral representation
\[
\mathbf{W}(t,s) \bm{\xi}^0 = \mathbf{W}_\infty(t-s) \bm{\xi}^0 + \int_s^t \mathbf{W}_\infty(t-\tau) \mathbf{B}(\tau) \mathbf{W}(\tau, s) \bm{\xi}^0 \; \mathrm{d}\tau \qquad \forall \; 0 \le s \le t
\] 
from \cite[Ch. VI, Theorem 9.19]{Engel}. Once we have shown uniform exponential stability for $(\mathbf{W}_\infty(t))_{t \in \mathbb{R}_{\ge 0}}$, estimations in $L^p(\Omega) \times \mathbb{R}$ for $1 \le p \le \infty$ yield
\begin{align*}
\|\mathbf{W}(t,s) \bm{\xi}^0\|_{L^p(\Omega) \times \mathbb{R}} & \le  C \mathrm{e}^{-\sigma_\infty(t-s)} \|\bm{\xi}^0\|_{L^p(\Omega) \times \mathbb{R}}  + \int_s^t C \mathrm{e}^{-\sigma_\infty(t-\tau)} \|\mathbf{B}(\tau) \|_{L^\infty(\Omega) \times \mathbb{R}} \|\mathbf{W}(\tau,s) \bm{\xi}^0\|_{L^p(\Omega) \times \mathbb{R}}  \; \mathrm{d}\tau.
\end{align*}
Gronwall's inequality results in the estimate
\[
\|\mathbf{W}(t,s) \bm{\xi}^0\|_{L^p(\Omega) \times \mathbb{R}} \le  C \mathrm{e}^{-\sigma_\infty(t-s)} \exp \left( \int_s^t C\|\mathbf{B}(\tau) \|_{L^\infty(\Omega) \times \mathbb{R}} \; \mathrm{d}\tau \right) \|\bm{\xi}^0\|_{L^p(\Omega) \times \mathbb{R}}.
\] 
Although the theory of Bohl exponents was established for bounded operators in \cite[Ch. III, pp. 118]{Daleckii}, the same estimates used to prove \cite[Corollary 4.2]{Daleckii} apply to the above estimate in the context of semigroup theory, see further \cite[Corollary 4.2]{Schnaubelt}, \cite[Theorem 5]{Datko}. %or \cite[\S 7.1, p. 195]{Henry}.
More precisely, since $\lim_{t \to \infty} \mathbf{B}(t) = \mathbf{0}$, for each $\gamma \in (0,1)$ there is a $t_0>0$ such that $C\|\mathbf{B}(t)\|_{L^\infty(\Omega) \times \mathbb{R}} \le \gamma \sigma_\infty$ for all $t \ge t_0$. This implies the estimate
\[
\|\mathbf{W}(t,s) \bm{\xi}^0\|_{L^p(\Omega) \times \mathbb{R}} \le  \tilde{C} \mathrm{e}^{-\gamma \sigma_\infty(t-s)}  \|\bm{\xi}^0\|_{L^p(\Omega) \times \mathbb{R}}
\]
for $\tilde{C} =C \exp(\int_0^{t_0} C \|\mathbf{B}(\tau)\|_{L^\infty(\Omega) \times \mathbb{R}} \; \mathrm{d}\tau)$. Hence, uniform exponential stability carries over from the evolution system $\mathcal{W}_\infty$ to the full evolution system $\mathcal{W}$ on $L^p(\Omega) \times \mathbb{R}$, provided $\lim_{t \to \infty} \mathbf{B}(t) = \mathbf{0}$.  \\
Using the spectral mapping theorem \cite[Ch. IV, Corollary 3.12]{Engel} for analytical semigroups, it is well-known that exponential stability of the semigroup $(\mathbf{W}_\infty(t))_{t \in \mathbb{R}_{\ge 0}}$ can be verified via the spectrum of its generator $\mathbf{D}_0\Delta + \mathbf{L}_\infty$. It remains to show $s(\mathbf{D}_0\Delta + \mathbf{L}_\infty) <0$ for uniform exponential stability of the evolution system $\mathcal{W}_\infty$ resp. $\mathcal{W}$ \cite[Ch. V, Theorem 1.10]{Engel}. We infer from Proposition \ref{HardtLemma} that in case of $D=0$
\begin{align*}
 \sigma(\mathbf{L}_\infty)  = \{ A_{\ast}\} \cup \Sigma,
\end{align*}
where $\Sigma$ consists of all eigenvalues of the constant coefficient matrix 
\[
\mathbf{J}_\infty = \begin{pmatrix}
A_{\ast} & B_\ast\\
C_\ast & D_\ast
\end{pmatrix}.
\]
Note that $A_{\ast} =-p<0$ and both eigenvalues of $\mathbf{J}_\infty$ have negative real parts since $\mathrm{tr}(\mathbf{J}_\infty) =-p-c \overline{v}<0$ and $\det(\mathbf{J}_\infty) = (pc+ab)\overline{v}>0$. \\
From the reasoning in Lemma \ref{Sigma} we know that $\sigma(\mathbf{D}_0\Delta + \mathbf{L}_\infty)$ is a discrete set for $D>0$. %and one could follow the strategy of \cite[Appendix]{Miyamoto}. However, since the evolution of the linear shadow limit is quite simple, we apply a different approach. 
The semigroup $(\mathbf{W}_\infty(t))_{t \in \mathbb{R}_{\ge 0}}$ is defined by the solution $\bm{\xi}=(\xi_1,\xi_2)$ of the shadow problem $\partial_t \bm{\xi} - \mathbf{D}_0\Delta \bm{\xi} = \mathbf{L}_\infty \bm{\xi}$. While $\xi_1(\cdot,t)-\langle \xi_1 \rangle_\Omega(t) = (S_\Delta(Dt)\xi_1^0 - \langle \xi_1^0 \rangle_\Omega) \mathrm{e}^{-pt}$, integration yields
\[
\langle \bm{\xi} \rangle_\Omega = \begin{pmatrix}
\langle \xi_1 \rangle_\Omega\\
\xi_2
\end{pmatrix}(t) = \exp(\mathbf{J}_\infty t) \begin{pmatrix}
\langle \xi_1^0 \rangle_\Omega\\
 \xi_2^0 
\end{pmatrix}.
\]
It is well-known from the theory of ODEs that $\langle \bm{\xi} \rangle_\Omega$ decays exponentially to zero since $\mathbf{J}_\infty$ is a stable matrix \cite[Ch. I, Theorem 4.1]{Daleckii}. Choosing $\sigma \in \mathbb{R}_{>0}$ such that $\sigma< \min \{p, \min_{\lambda \in \sigma(\mathbf{J}_\infty)} |\mathrm{Re} \, \lambda|\}$ yields an estimation of both expressions, $\langle \bm{\xi} \rangle_\Omega$ and $\bm{\xi} - \langle \bm{\xi} \rangle_\Omega$. This results in
\[
\|\mathbf{W}_\infty(t) \bm{\xi}^0\|_{L^\infty(\Omega) \times \mathbb{R}} = \|\bm{\xi}(\cdot,t) \|_{L^\infty(\Omega) \times \mathbb{R}} \le C_\sigma \mathrm{e}^{-\sigma t} \|\bm{\xi}^0\|_{L^\infty(\Omega) \times \mathbb{R}} 
\]
for some constant $C_\sigma>0$. Thus, Assumption \ref{AssumptionL1p} is satisfied for $p=\infty, \sigma>0$ and each $D \ge 0$.
\end{proof}

In summary, Theorem \ref{TheoremToy} yields global estimates
 \begin{align*}
\|u_\varepsilon - u\|_{L^\infty(\Omega \times \mathbb{R}_{\ge 0})} +  \|v_\varepsilon-v-\psi_\varepsilon\|_{L^\infty(\Omega \times \mathbb{R}_{\ge 0})}   \le  C\varepsilon. 
 \end{align*}
Note that the results on Lyapunov functions in \cite[Proposition 2.1]{Hattaf} is also applicable to partly diffusing systems. The same Lyapunov function which is known from the theory of ODEs can be extended to the reaction-diffusion case. Consequently, $(\overline{u}, \overline{v})$ is the only positive attractor for the diffusing system \eqref{PP} and $(\overline{u}, \overline{v})$ is globally (for positive initial data) asymptotically stable by Lyapunov's direct method.

\section*{Acknowledgments}

{This work was supported by the Deutsche Forschungsgemeinschaft (DFG, German Research Foundation) under Germany's Excellence Strategy EXC 2181/1 - 390900948 (the Heidelberg STRUCTURES Excellence Cluster and SFB1324 (B05) and the Klaus Tschira Foundation, Germany (00.277.2015).}

\begin{appendices}

%\addappheadtotoc % \appendixtocname to TOC
%\addtocontents{toc}{\protect\setcounter{tocdepth}{-1}} % Einzelne Chapter nicht ins Inhaltsverzeichnis

\section{Parabolic theory} \label{AppendixA} 

The preceding results are based on properties of the heat semigroup $(S_\Delta(t))_{t \in \mathbb{R}_{\ge 0}}$ defined in \cite{Davies, Ouhabaz}. We study basic properties of this semigroup on the spaces $L^p(\Omega)$ for $1 \le p < \infty$. Additionally, $L^\infty(\Omega_T)$ estimates are derived for solutions of the inhomogeneous heat equation with explicit dependence on time $T$.

\begin{proposition} \label{heathom}
Let $\Omega \subset \mathbb{R}^n$ be a bounded domain with $\partial \Omega \in C^{0,1}$, $z^0 \in L^2(\Omega)$. Then the homogeneous heat equation 
  \begin{align*}
  \frac{\partial z}{\partial t} - \Delta z &= 0 \quad \mbox{in} \quad   \Omega_T,  \qquad z(\cdot,0) = z^0   \quad \mbox{in} \quad   \Omega, 
 \qquad    \frac{\partial z}{\partial \nu} =0 \quad \mbox{on} \quad  \partial \Omega \times (0,T)
  \end{align*}
has a unique variational solution $z \in C(\mathbb{R}_{\ge 0}; L^2 (\Omega ))$ satisfying $z \in L^2(0,T; H^1 (\Omega ))$ for each $T>0$. The solution is given by the Fourier expansion
  \begin{align}
  z(x,t) &= (S_\Delta(t) z^0)(x) =  \sum_{j \in \mathbb{N}_0} \mathrm{e}^{-\lambda_j t} (z^0 , w_j )_{L^2 (\Omega )}  w_j (x), \qquad x \in \Omega, t \in \mathbb{R}_{\ge 0},  \label{heatsemigroup}
 \end{align}
where $(\lambda_j, w_j)_{j \in \mathbb{N}_0}$ is a spectral basis of $-\Delta$ on $L^2(\Omega)$ solving problem \eqref{spect}.\\
Moreover, the heat semigroup $(S_\Delta(t))_{t \in \mathbb{R}_{\ge 0}}$ defined by \eqref{heatsemigroup} can be extended to a contraction semigroup on $L^p(\Omega)$ for each $1 \le p \le \infty$, which is strongly continuous for $1 \le p < \infty$ and analytic for $1 < p < \infty$.
\end{proposition}

\begin{proof}
Existence and contractivity of the heat semigroup $(S_\Delta(t))_{t \in \mathbb{R}_{\ge 0}}$ is shown in \cite[Theorem 1.3.9]{Davies}. By \cite[Theorems 1.4.1, 1.4.2]{Davies}, $(S_\Delta(t))_{t \in \mathbb{R}_{\ge 0}}$ is a strongly continuous semigroup for each $1 \le p < \infty$, which even can be extended analytically to some sector in the complex plane for $p>1$. The solution $z(t) = S_\Delta(t) z^0$ for $z^0 \in L^2(\Omega)$ satisfies $z \in C(\mathbb{R}_{\ge 0}; L^2(\Omega))$ due to strong continuity of the semigroup. Although the solution might lose its differentiability at $t=0$, we obtain higher regularity for $t>0$ by analyticity of the semigroup. This yields $z \in C(\mathbb{R}_{>0}; \mathcal{D}(\Delta^\ell))$ for all integers $k, \ell \in \mathbb{N}_0$, where $\mathcal{D}(\Delta^0) := L^2(\Omega)$ and $\mathcal{D}(\Delta)$ is the domain of the generator of the heat semigroup defined on $L^2(\Omega)$ \cite[Theorem 7.7]{Brezis}. Hence, the solution $z(\cdot,t)$ lies in $H^1(\Omega)$ for each $t>0$ and the boundary condition is satisfied in the sense of distributions by the trace operator $H^1(\Omega) \hookrightarrow W^{-1/2,2}(\partial \Omega)$ from \cite[Theorem 1.5.1.2]{Grisvard}.\\
To determine the Fourier coefficients, we recall that the unique solution $z$ solves the variational equation
\[
\frac{\mathrm{d}}{\mathrm{d}t} \int_\Omega z(x,t) \varphi(x) \; \mathrm{d}x + \int_{\Omega} \nabla z(x,t) \nabla \varphi(x) \; \mathrm{d}x =0 \qquad \forall \; \varphi \in H^1(\Omega), t \in \mathbb{R}_{>0}.
\]
Since $z(\cdot,t) \in L^2(\Omega)$, we can expand this function in a Fourier series using a spectral basis of $-\Delta$ from equation \eqref{spect}. This leads to the series representation \eqref{heatsemigroup}.\\
Concerning the regularity $z \in L^2(0,T; H^1(\Omega))$, we finally note that the partial sums 
%\[
%z_m(\cdot, t) = \sum_{j =0}^m \mathrm{e}^{-\lambda_j t} (z^0 , w_j )_{L^2 (\Omega )}  w_j
%\]
of the Fourier series form a Cauchy sequence in this space.% Indeed, for $\ell >m \ge 0$ 
%\begin{align*}
%\|z_m- z_\ell \|_{L^2(0,T;H^1(\Omega))}^2 & = \sum_{j=m + 1}^\ell   \int_0^T (1+\lambda_j) \mathrm{e}^{-2\lambda_j t} \; \mathrm{d}t\; |(z^0 , w_j )_{L^2 (\Omega )}|^2  \le C \sum_{j=m + 1}^\ell    |(z^0 , w_j )_{L^2 (\Omega )}|^2
%\end{align*}
%tends to $0$ as $m,\ell \to \infty$, due to Parseval's equality for $z^0 \in L^2(\Omega)$ and $C=1/2 + 1/(2\lambda_1)$.
\end{proof}

%Remind that the above proposition may be shown for different parabolic differential operators as well as other boundary conditions \cite[Theorem 2.4]{Daners}. \\
As used above, Galerkin's approximation is based on a spectral basis $(\lambda_j,w_j)_{j \in \mathbb{N}_0}$ of $-\Delta$ satisfying
  \begin{equation}\label{spect}
    -\Delta w_j=\lambda_j w_j \quad \mbox{in} \quad \Omega, \qquad \frac{\partial w_j}{\partial \nu} =0 \quad \mbox{on} \quad  \partial \Omega
  \end{equation}
in the weak sense \cite[Theorem 1.2.8]{Henrot}. The principal eigenvalue $\lambda _0 =0$ has eigenfunction $w_0 =|\Omega |^{-1/2}$. The other eigenvalues are strictly positive and tend to infinity as $j \to \infty$. The eigenfunctions $(w_j )_{j \in \mathbb{N}_0} \subset H^1(\Omega)$ form an orthonormal basis for $L^2 (\Omega )$ and an orthogonal basis for $H^1 (\Omega)$.\\ 
The first positive eigenvalue $\lambda_1>0$ of $-\Delta$ considered as an operator on $L^2(\Omega)$ is also fundamental for the decay estimate of the heat semigroup which is used crucially throughout this work. 

\begin{lemma} \label{Winterlemma}
Let $\Omega \subset \mathbb{R}^n$ be a bounded domain with $\partial \Omega \in C^{0,1}$ and $\lambda_1>0$ the first non-zero eigenvalue of $-\Delta$ endowed with zero Neumann boundary conditions. Then there exists a constant $C>0$, merely depending on $\Omega$, such that for all $1 \le q \le p \le \infty$
\begin{equation}
\|S_\Delta(\tau) z\|_{L^{p}(\Omega)} \le C m(\tau)^{-\frac{n}{2}\left( \frac{1}{q}- \frac{1}{p} \right)} \mathrm{e}^{-\lambda_1 \tau} \| z\|_{L^{q}(\Omega)} \qquad \forall \; \tau \in \mathbb{R}_{> 0} \label{winlemmaapp}
\end{equation}
holds for all $z \in L^q(\Omega)$ satisfying $\langle z \rangle_\Omega = 0$. Here, we denote $m(\tau) = \min\{1,\tau\}$. Especially for $p=q$, there exists a constant $\overline{C}>0$ independent of $\tau$ and the initial conditions $z$ such that
\begin{equation*}
\| S_\Delta(\tau) z \|_{L^{p}(\Omega)} \le \overline{C} \mathrm{e}^{-\lambda_1 \tau} \| z \|_{L^p(\Omega)} \qquad \forall \; \tau \in \mathbb{R}_{\ge 0}. 
\end{equation*}
\end{lemma}

\begin{proof}
This is essentially proven in \cite[Lemma 1.3]{Winkler}. Unfortunately, \cite{Winkler} uses estimates of the heat kernel for a bounded domain with boundary $\partial \Omega \in C^{2,\alpha}$ for some $\alpha \in (0,1)$. In order to relax this condition, we use heat kernel estimates for $\partial \Omega \in C^{0,1}$ from \cite[Theorem 3.2.9]{Davies}, which imply decay estimate \eqref{winlemmaapp} by the same steps as in the proof of \cite{Winkler}.
\end{proof}

By Duhamel's formula, the following result is obtained for the inhomogeneous heat equation.

\begin{proposition} \label{heatinhom}
Let $\Omega \subset \mathbb{R}^n$ be a bounded domain with $\partial \Omega \in C^{0,1}$, $z^0 \in L^2(\Omega)$ and $R \in L^2 (\Omega_T)$ be given. Then the inhomogeneous heat equation 
  \begin{align}
  \frac{\partial z}{\partial t} - \Delta z &= R(x,t) \quad \mbox{in} \quad   \Omega_T,  \qquad z(\cdot,0) = z^0   \quad \mbox{in} \quad   \Omega, 
 \qquad    \frac{\partial z}{\partial \nu} =0 \quad \mbox{on} \quad  \partial \Omega \times (0,T) \label{heatinhom1}
  \end{align}
  has a unique mild solution $z\in L^2 (0,T; H^1 (\Omega )) \cap C([0,T]; L^2 (\Omega ))$, given by the separation of variables formula
  \begin{align}\label{Fourierinhom}
  \begin{split}
  z(\cdot,t) &=   S_\Delta(t)z^0 + \int_0^t S_\Delta(t-s) R(\cdot,s) \; \mathrm{d}s \\ 
  & = \sum_{j \in \mathbb{N}_0} \mathrm{e}^{-\lambda_j t} (z^0 , w_j )_{L^2 (\Omega )}  w_j (x) + \int^t_0 \sum_{j \in \mathbb{N}_0}   \mathrm{e}^{-\lambda_j (t-s)} ( R(\cdot \ , s), w_j )_{L^2 (\Omega )}  w_j (x)  \; \mathrm{d}s. 
  \end{split}
 \end{align}
In addition, $z^0 \in H^1(\Omega)$ implies a weak solution $z \in L^\infty(0,T; H^1(\Omega))$ with weak derivative $\partial_t z \in L^2(\Omega_T)$, and the weak formulation
\[
(\partial_t z(\cdot,t), \varphi)_{L^2(\Omega)} + (\nabla z(\cdot,t), \nabla \varphi)_{L^2(\Omega)} = (R(\cdot,t), \varphi)_{L^2(\Omega)} \qquad \forall \; \varphi \in H^1(\Omega)
\]
holds for a.e. $t \in (0,T)$.
\end{proposition}

\begin{proof}
By \cite[Section 4.2]{Pazy}, there exists a unique mild solution $z \in C([0,T];L^2(\Omega))$ of problem \eqref{heatinhom1}, which is given by the integral formula \eqref{Fourierinhom}. To show that $z$ is an element of $L^2(0,T;H^1(\Omega))$ resp. $L^\infty(0,T;H^1(\Omega))$, it is again sufficient to prove the Cauchy property of the partial sums induced by expression \eqref{Fourierinhom}. %For $\ell >m \ge 0$, we obtain 
%\begin{align*}
%\|(z_m- z_\ell)(\cdot,t) \|_{H^1(\Omega)}^2 & \le \sum_{j=m + 1}^\ell   (1+\lambda_j) \left[ \mathrm{e}^{-2\lambda_j t} \; |(z^0 , w_j )_{L^2 (\Omega )}|^2 +\left( \int_0^t \mathrm{e}^{-\lambda_j (t-s)}  (R(\cdot,s) , w_j )_{L^2 (\Omega )} \; \mathrm{d}s \right)^2 \right] \\
%& \quad + \sum_{j=m + 1}^\ell   (1+\lambda_j)  \\
%& \le \sum_{j=m + 1}^\ell   (1+\lambda_j)  |(z^0 , w_j )_{L^2 (\Omega )}|^2   + (T+1) \int_0^t  \sum_{j=m + 1}^\ell   |(R(\cdot,s) , w_j )_{L^2 (\Omega )}|^2 \; \mathrm{d}s 
%\end{align*}
%by H\"older's inequality. The same reasoning as in the proof of Proposition \ref{heathom} applies to obtain a Cauchy sequence, hence, we omit details. 
Concerning the weak formulation, we use Galerkin's approximation \cite[Section 7.1]{Evans}. Following this classical approach from \cite[\S 7.1.3, Theorem 5]{Evans}, one establishes the result for more regular initial data $z^0 \in H^1(\Omega)$.
\end{proof}

Once we obtained a solution of problem \eqref{heatinhom1} for $z^0 =0$ in Proposition \ref{LinftynD}, we used an $L^\infty(\Omega_T)$ estimate for the solution of the inhomogeneous heat equation with explicit dependence on time $T$. This can be developed as in \cite[Ch. III, \S 7]{Ladyparab}, which essentially uses parabolic $L_{p,r}(\Omega_T)$ estimates in combination with the well-known truncation method of Stampacchia. For the convenience of the reader, we will use semigroup theory to prove such a result. Let us consider a bounded weak solution $z \in L^2 (0,T; H^1 (\Omega )) \cap C([0,T]; L^2(\Omega))$ which solves 
  \begin{align}
   \frac{\partial z}{\partial t} - d \Delta z &= R(x,t) \quad \mbox{in} \quad   \Omega_T,  \qquad z(\cdot,0) = 0  \quad \mbox{in} \quad   \Omega, 
 \qquad    \frac{\partial z}{\partial \nu} =0 \quad \mbox{on} \quad  \partial \Omega \times (0,T) \label{eqzd}
  \end{align}
with right-hand side $R \in L_{p,r}(\Omega_T)$ and diffusion parameter $d$. The Lebesgue space $L_{p,r}(\Omega_T)$ is defined by all measurable functions $\psi$ on $\Omega_T$ with finite mixed norm 
\begin{equation}
\|\psi\|_{p,r} := \left( \int_0^T \left( \int_\Omega |\psi(x,t)|^p \; \mathrm{d}x \right)^{r/p}  \mathrm{d}t \right)^{1/r} \qquad \text{for} \quad 1 \le p,r < \infty \label{parabLpr}
\end{equation}
and an obvious modification for $r= \infty$ \cite[Chapters I, II, \S1 in both cases]{Ladyparab}. %It is well known that $L_{p,r}(\Omega_T) = L^r(0,T;L^p(\Omega))$ for $p,r< \infty$ %since simple functions are dense in both spaces with the same norm 
%\cite[Section 1.1]{Arendt}.
If we do not specify the region of integration within the notation $\|\cdot\|_{p,r}$, we assume to integrate over $\Omega_T$. \\
The aim is to show $L^\infty(\Omega_T)$ estimates for the solution $z$ which depend explicitly on time $T$ and the mixed norm $\|R\|_{p,r}$. Within this procedure, the exponent $p$ is restricted due to Sobolev embeddings by $p> n/2$ for $n \ge 2$. Let us choose a parameter $1 \le r \le \infty$ according to \cite[Ch. III, \S 7]{Ladyparab}, i.e., we assume the relation
 \begin{align}
0 \le  \frac{1}{r} + \frac{n}{2p} <1 \label{paramLady}
\end{align} 
for given $1 \le p, r \le \infty$. Then we obtain the following result.

 \begin{proposition}\label{thees} Let $z \in L^2 (0,T; H^1 (\Omega )) \cap C([0,T]; L^2(\Omega))$ be a bounded weak solution of the initial boundary value problem \eqref{eqzd} with right-hand side $R \in L_{p,r}(\Omega_T)$ and parameter values $p,r$ satisfying relation \eqref{paramLady}. Then there exists a constant $C>0$ which only depends on $\Omega, n, p, r$ and a lower bound of the diffusion $d$ such that the diffusing component $z$ satisfies the estimate
 \begin{equation}\label{esinfty}
   \|z \|_{L^\infty (\Omega_T)} \leq C T^{1-1/r} \| R  \|_{p, r }.
 \end{equation}
 \end{proposition}
 
 \begin{proof}
The solution $z$ is given by the integral representation \eqref{Fourierinhom}, i.e.,
\[
 z(\cdot,t) =   \int_0^t S_\Delta(d(t-s)) R(\cdot,s) \; \mathrm{d}s =  \int_0^t S_\Delta(d(t-s)) \left( R(\cdot,s) - \langle R(\cdot,s) \rangle_\Omega \right) \; \mathrm{d}s + \int_0^t \langle R(\cdot,s) \rangle_\Omega \; \mathrm{d}s.
\]
The first integral can be estimated with the decay estimate \eqref{winlemmaapp} of the heat semigroup from Lemma \ref{Winterlemma}. In fact, we obtain 
\begin{align*}
\| S_\Delta(d(t-s)) \left( R(\cdot,s) - \langle R(\cdot,s) \rangle_\Omega  \right) \|_{L^\infty(\Omega)} & \le C \left( 1 + (d(t-s))^{-\frac{n}{2p}} \right)  \| R(\cdot,s) - \langle R(\cdot,s) \rangle_\Omega  \|_{L^p(\Omega)}\\
& \le 2C \left( 1 + (d(t-s))^{-\frac{n}{2p}} \right)  \| R(\cdot,s) \|_{L^p(\Omega)}
\end{align*}
for a constant $C>0$ which depends on $\Omega$ only. Condition \eqref{paramLady} yields that the function $\tau \mapsto  1 + (d\tau)^{-\frac{n}{2p}}$ is in $L^{\hat{r}}((0, T))$ where $\frac{1}{\hat{r}} + \frac{1}{r} = 1$. Since the second integral satisfies
\begin{align*}
\left| \int_0^T \langle R(\cdot,s) \rangle_\Omega \; \mathrm{d}s \right| & \le |\Omega|^{-1/p} T^{1/\hat{r}} \|R\|_{p,r},
\end{align*}
H\"older's inequality implies the desired inequality \eqref{esinfty}.
 \end{proof}

\end{appendices}

%\hypersetup{linkcolor=black, citecolor=black, urlcolor=black}	% Link-Farbe auf Schwarz fürs Literaturverzeichnis

% Umlaute werden alphabetisch wie der Vokal selbst eingeordnet, d.h. ä wie a oder ö wie o. Folglich kommt Köthe kurz nach Kondo und nicht vor Kondo.

\end{document}